\documentclass[11pt,lenq]{amsart}
\usepackage{amsmath}
\usepackage{amscd}
\usepackage{amssymb}
\usepackage{amsbsy}
\usepackage{amsfonts}
\usepackage{color}
\usepackage{latexsym}
\usepackage{graphicx}
\usepackage{amsmath,amscd,latexsym}
\usepackage{multirow}
\usepackage{array}
\usepackage{paralist}
\usepackage{titletoc}
\usepackage{url} 

\theoremstyle{plain}
\newtheorem{theorem}{Theorem}[section]
\newtheorem{proposition}[theorem]{Proposition}
\newtheorem{lemma}[theorem]{Lemma}
\newtheorem{corollary}[theorem]{Corollary}
\newtheorem{remark}[theorem]{Remark}
\newtheorem{definition}[theorem]{Definition}
\newtheorem{notation}[theorem]{Notation}

\newtheorem{main theorem}[theorem]{Main Theorem}

\newtheorem{convention}[theorem]{Convention}

\newtheorem{assumption}[theorem]{Assumption}
\newtheorem{claim}[theorem]{Claim}

\newlength\savewidth

\newcommand{\interior}{\operatorname{int}}
\newcommand{\fr}{\operatorname{fr}}

\newcommand{\cl}{\operatorname{cl}}

\newcommand{\NN}{\mathbb{N}}
\newcommand{\ZZ}{\mathbb{Z}}
\newcommand{\QQ}{\mathbb{Q}}
\newcommand{\RR}{\mathbb{R}}
\newcommand{\CC}{\mathbb{C}}
\newcommand{\HH}{\mathbb{H}}

\newcommand{\Di}{\mathbb{D}}
\newcommand{\RP}{\mathbb{RP}}
\newcommand{\EE}{\mathbb{E}}

\newcommand{\Conway}{\mbox{\boldmath$S$}^{2}}
\newcommand{\Conways}
{(\mbox{\boldmath$S$}^{2},\mbox{\boldmath$P$}^0)}
\newcommand{\PP}{\mbox{{\boldmath$P$}}^0}
\newcommand{\PConway}{\check{\mbox{\boldmath$S$}}^2}
\newcommand{\Ball}{\mbox{\boldmath$B$}^3}

\newcommand{\NE}{\mbox{$\mathrm{NE}$}}
\newcommand{\NW}{\mbox{$\mathrm{NW}$}}
\newcommand{\SW}{\mbox{$\mathrm{SW}$}}
\newcommand{\SE}{\mbox{$\mathrm{SE}$}}

\newcommand{\PSL}{\mbox{$\mathrm{PSL}$}}
\newcommand{\SL}{\mbox{$\mathrm{SL}$}}

\newcommand{\svert}{\,|\,}

\newcommand{\llangle}{\langle\langle}
\newcommand{\rrangle}{\rangle\rangle}

\newcommand{\OO}{{\mathcal{O}}}

\newcommand{\orbs}{\mathcal{S}}
\newcommand{\orbh}{\mathcal{H}}
\newcommand{\orbspherical}{\mathcal{S}}
\newcommand{\orbb}{\mbox{\boldmath$B$}}
\newcommand{\orbm}{\mathcal{M}}
\newcommand{\orbp}{(M_0,P)}
\newcommand{\RotG}{{\mathcal{J}}}
\newcommand{\Riley}{{\mathcal{R}}}

\newcommand{\fix}{\operatorname{Fix}}

\newcommand{\Isom}{\operatorname{Isom}}
\newcommand{\Ker}{\operatorname{Ker}}
\newcommand{\lcm}{\operatorname{lcm}}
\newcommand{\quaternion}{\mathcal{H}}
\newcommand{\LM}{{\mathrm{L}}}

\newcommand{\Stab}{\operatorname{Stab}}

\makeatletter
\renewcommand\subsection{\@startsection{subsection}{2}{0mm}
    {-10.5dd plus-8pt minus-4pt}{10.5dd}
     {\normalsize\upshape}}
\makeatother

\begin{document}

\title[Non-free Kleiniain groups generated by two parabolic transformations]
{Classification of non-free Kleinian groups generated by two parabolic transformations}

\author{Hirotaka Akiyoshi}
\address{Department of Mathematics\\
Graduate School of Science\\
Osaka City University\\
3-3-138, Sugimoto, Sumiyoshi-ku
Osaka, 558-8585, Japan}
\email{akiyoshi@sci.osaka-cu.ac.jp}

\author{Ken'ichi Ohshika}
\address{Department of Mathematics\\
Faculty of Science\\
Gakushuin University\\
Mejiro 1-5-1, Toshima-ku, 171-8588, Japan}
\email{ohshika@math.gakushuin.ac.jp}

\author{John Parker}
\address{Department of Mathematical Sciences\\ 
Durham University, Science Laboratories\\
South Road, Durham, DH1 3LE, United Kingdom}
\email{j.r.parker@durham.ac.uk}

\author{Makoto Sakuma}
\address{Department of Mathematics\\
Faculty of Science\\
Hiroshima University\\
Higashi-Hiroshima, 739-8526, Japan}
\email{sakuma@hiroshima-u.ac.jp}

\author{Han Yoshida}
\address{National Institute of Technology, Nara College\\
22 Yata-cho, Yamatokoriyama, Nara, 639-1058, Japan}
\email{han@libe.nara-k.ac.jp}



\begin{abstract}
We give a full proof to Agol's announcement on the classification
of non-free Kleinian groups generated by two parabolic transformations.
\end{abstract}

\maketitle

\tableofcontents

\section{Introduction}
\label{sec:introduction}

Motivated by knot theory, Riley studied 
Kleinian groups generated by two 
parabolic transformations 
(see \cite{Riley1972, Riley1975, Riley1975b, Riley1992, Riley_computer}).
In particular, the construction of the complete hyperbolic structure on the 
figure-eight knot complement \cite{Riley1975}
inspired Thurston to establish the uniformisation theorem of Haken manifolds.
The space of marked subgroups of $\PSL(2,\CC)$ generated by two non-commuting
parabolic transformations is parametrised by a non-zero complex number. 
There is 
an open set, $\Riley$, called the {\it Riley slice of Schottky space},
of Kleinian groups of this type that are free and discrete, and
for which the quotient of the domain of discontinuity is a four times punctured sphere.
For every group in $\Riley$,
the Klein manifold
(the quotient of union of the hyperbolic space and the domain of discontinuity) 
is homeomorphic to the complement of the 2-strand trivial tangle.
Keen and Series \cite{Keen-Series}
studied the Riley slice by applying their theory of pleating rays,
and it was supplemented by Komori and Series \cite{Komori-Series}.
Motivated by knot theory, Akiyoshi, Sakuma, Wada and Yamashita \cite{ASWY}
studied the combinatorial structures of the Ford domains, 
by extending Jorgensen's work \cite{Jorgensen} on punctured torus groups,
which leads to a natural tessellation of $\Riley$ (see Figure 0.2b in \cite{ASWY}).  
Ohshika and Miyachi \cite{Ohshika-Miyachi} proved
that the closure of $\Riley$ is equal to the space of 
marked Kleinian groups with two parabolic generators 
which are free and discrete.
Building on his joint work 
\cite{Gehring-Martin},
\cite{Hinkkanen-Martin}
and \cite{Martin-Marshall}
with Gehring, Hinkkanen and Marshall, respectively,
Martin \cite{Martin} identified the exterior of $\Riley$
as the Julia set of a certain semigroup of polynomials
and proved a \lq\lq supergroup density theorem'' for groups
in the exterior of $\Riley$.
The problem to detect freeness and non-freeness of 
(not necessarily discrete) groups generated by two non-commuting
parabolic transformations has attracted attention of various researchers 
(see \cite{Lyndon-Ullman, Gilman, Tan-Tan, Kim-Koberda} and references therein).

In this paper, we are interested in Kleinian groups that are in the complement of 
the closure of $\Riley$, namely
the groups that are discrete but not free.
The essential simple loops on the boundary of the complement of the 
2-strand trivial tangle, which are not null homotopic in the ambient space,
are parametrised by a slope $r$ in $\QQ/2\ZZ$. 
The Heckoid groups, introduced by Riley \cite{Riley1992}
and formulated by Lee and Sakuma \cite{Lee-Sakuma_2013} following
Agol \cite{Agol},
are Kleinian groups 
with two parabolic generators in which the element corresponding to the curve $\alpha_r$
of slope $r$ has finite order. The most extreme case is the group $G(r)$ where this element
is the identity, 
in which case, the quotient of hyperbolic space by this group is 
the complement of a 2-bridge knot or link.

In \cite[Theorem~4.3]{Adams1},
Adams proved that a non-free and torsion-free Kleinian group $\Gamma$ is
generated by two parabolic transformations
if and only if the quotient hyperbolic manifold $\HH^3/\Gamma$
is homeomorphic to the complement of a $2$-bridge link $K(r)$
which is not a torus link.
(We regard a knot as a one-component link.)
This refines
the result of Boileau and Zimmermann \cite[Corollary~3.3]{Boileau-Zimmermann}
that a link in $S^3$ is a $2$-bridge link if and only if its link group is generated by two meridians.

In 2002, Agol \cite{Agol}  announced the following classification theorem
of non-free Kleinian groups generated by two parabolic transformations,
which generalises Adams' result.
The main purpose of this paper is to give a full proof to this theorem.

\begin{theorem}
\label{main-theorem}
A non-free Kleinian group $\Gamma$ is generated by two non-commuting  
parabolic elements if and only if one of the following holds.
\begin{enumerate}[\rm (1)]
\item
$\Gamma$ is conjugate to the hyperbolic $2$-bridge link group, $G(r)$,
for some rational number $r=q/p$,
where $p$ and $q$ are coprime integers such that 
$q\not\equiv \pm 1 \pmod{p}$.
\item
$\Gamma$ is conjugate to the Heckoid group, $G(r;n)$,
for some $r\in\QQ$ and some $n\in \frac{1}{2}\NN_{\ge 3}$.
\end{enumerate}
\end{theorem}

\begin{figure}
\includegraphics[width=0.9\hsize, bb=0 0 3044 898]{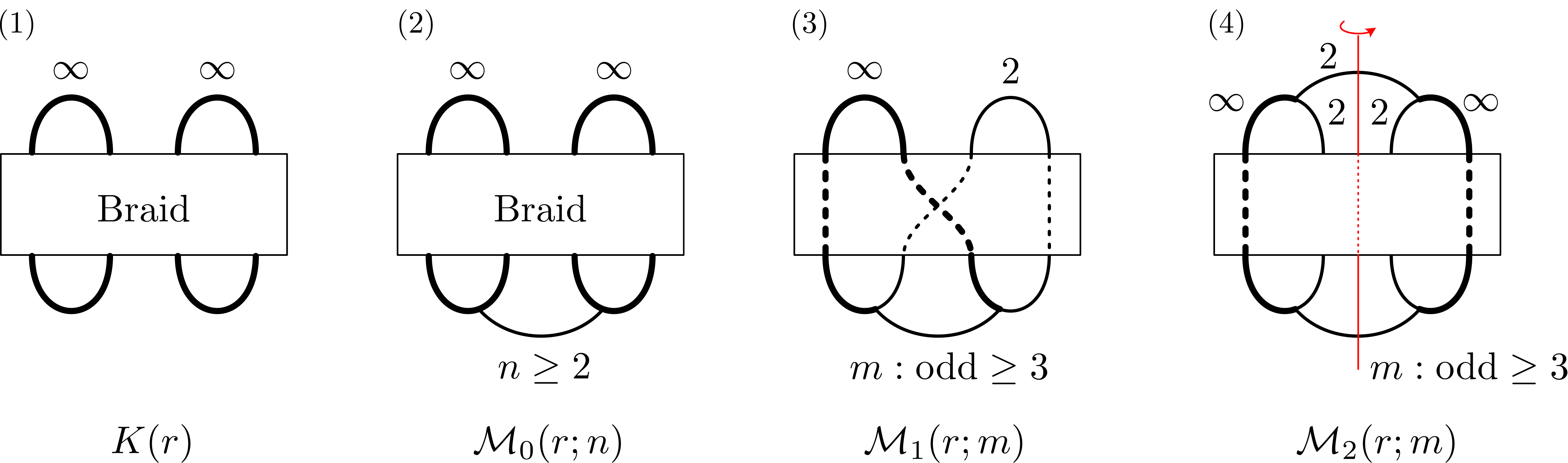}
\caption{Weighted graphs representing $2$-bridge links and Heckoid orbifolds,
where the thick edges with weight $\infty$ correspond to parabolic loci
and thin edges with integral weights represent the singular set.
See Definition \ref{def:Heckoid-orbifold} for the precise 
description of the weighted graphs.
}
\label{fig.Heckoid-orbifold}
\end{figure}

In the remainder of the introduction, 
we explain the meaning of the theorem more precisely.

Recall that a {\it $2$-bridge link} is a knot or a two-component link
which is represented by a diagram in the $x$-$y$ plane
that has two maximal points and two minimal points
with respect to the height function determined by the $y$-coordinate.
We may assume that the two maximal points 
and the two minimal points, respectively, have the same $y$-coordinates.
Such a diagram gives a plait (or plat) representation of the $2$-bridge link
consisting of two upper bridges, two lower bridges, and a $4$-strand braid
connecting the upper and lower bridges (see Figure \ref{fig.Heckoid-orbifold}(1)).
The $2$-bridge links are parametrized by the set $\QQ\cup\{\infty\}$, and
the $2$-bridge link corresponding to $r\in\QQ\cup\{\infty\}$
is denoted by $K(r)$ and is called the {\it $2$-bridge link of slope $r$}
(see Section 2 for the precise definition).
If $r=\infty$ then $K(r)$ is the $2$-component trivial link,
and if $r\in\ZZ$ then $K(r)$ is the trivial knot.
If $r=q/p\in \QQ$, where $p$ and $q$ are coprime integers,
then $K(q/p)$ is hyperbolic, i.e., $S^3-K(r)$ admits a complete hyperbolic structure
of finite volume,
if and only if $q\not\equiv \pm 1 \pmod{p}$.
In this case, there is a torsion-free Kleinian group $\Gamma$, 
unique up to conjugation,
such that $\HH^3/\Gamma$ is homeomorphic to
the link complement $S^3-K(r)$ as oriented manifold.
We denote the Kleinian group $\Gamma$, by $G(r)$,
and call
it the {\it hyperbolic $2$-bridge link group of slope $r$.}

The Heckoid groups were first introduced by Riley~\cite{Riley1992}
as an analogy of the classical Hecke groups 
considered by Hecke \cite{Hecke}.
The topological structure of their quotient orbifolds was worked out
by Lee and Sakuma \cite{Lee-Sakuma_2013},
following the description by Agol \cite{Agol}.
Specifically, they
showed that the Heckoid groups are the orbifold fundamental groups 
of the Heckoid orbifolds
illustrated in Figure \ref{fig.Heckoid-orbifold}(2)-(4).
(See \cite{BBMBP, Boileau-Porti, CHK} for basic terminologies and facts
concerning orbifolds.)
These figures illustrate weighted graphs $(S^3,\Sigma, w)$
whose explicit descriptions are given by Definition \ref{def:Heckoid-orbifold}.
For each weighted graph $(S^3,\Sigma, w)$ in the figure,
let $(M_0,P)$ be the pair of a compact $3$-orbifold $M_0$
and a compact $2$-suborbifold $P$ of $\partial M_0$
determined by the rules described below.
Let $\Sigma_{\infty}$ 
be the subgraph of $\Sigma$ consisting of the edges 
with weight $\infty$, 
and let $\Sigma_s$ be the subgraph of $\Sigma$ consisting of the edges 
with integral weight.
\begin{enumerate}
\item
The underlying space $|M_0|$ of the orbifold $M_0$ is the complement of an open regular neighbourhood of 
the subgraph $\Sigma_{\infty}$.
\item
The singular set of $M_0$ is 
$\Sigma_0:=\Sigma_s\cap |M_0|$, 
where the index of each edge of the singular set is given by the weight $w(e)$
of the corresponding edge $e$ of $\Sigma_s$.
\item
For an edge $e$ of $\Sigma_\infty$,
let $P$ be the 
$2$-suborbifold of $\partial M_0$ defined as follows.
\begin{enumerate}
\item
In Figure \ref{fig.Heckoid-orbifold}(2), $P$ consists of 
two annuli in $\partial M_0$
whose cores, respectively, are meridians of the two edges of $\Sigma_{\infty}$.
\item
In Figure \ref{fig.Heckoid-orbifold}(3), $P$ consists of an annulus in $\partial M_0$
whose core is a meridian of the single edge of $\Sigma_{\infty}$.
\item
In Figure \ref{fig.Heckoid-orbifold}(4), $P$ consists of two copies of the
annular orbifold 
$D^2(2,2)$ (the $2$-orbifold with underlying space the disc
and with two cone points of index $2$)
in $\partial M_0$
each of which is 
bounded by a meridian of an edge of $\Sigma_{\infty}$.
\end{enumerate}
\end{enumerate}

By \cite[Lemmas 6.3 and 6.6]{Lee-Sakuma_2013},
the orbifold pair $(M_0,P)$ is a Haken pared orbifold
(see Definition \ref{def:pared-orbifold} or \cite[Definition 8.3.7]{Boileau-Porti}) 
and admits a unique complete hyperbolic structure,
which is geometrically finite 
(see Section 3 or \cite[Proposition 6.7]{Lee-Sakuma_2013}).
Namely there is a geometrically finite Kleinian group $\Gamma$,
unique up to conjugation,
such that $M:=\HH^3/\Gamma$ is isomorphic to the 
interior of the compact orbifold $M_0$,
such that $P$ represents 
the parabolic locus. 
The pair $(M_0,P)$ is also regarded as a relative compactification
of the pair consisting of a non-cuspidal part of $M$ and its boundary
(see Section \ref{sec:Heckoid}). 

We denote the pared orbifold $\orbm:=(M_0,P)$ by 
$\orbm_{0}(r;n)$, $\orbm_{1}(r;m)$, or $\orbm_{2}(r;m)$
according as it is described by the weighted graph in 
Figure \ref{fig.Heckoid-orbifold}(2), (3), or (4).
We also denote the Kleinian group $\Gamma$ by $\pi_1(\orbm)$.

Then the assertion (2) of the main Theorem \ref{main-theorem} is equivalent to the
following assertion (2')

\begin{enumerate}
\item[(2')]
{\it
$\Gamma$ is conjugate to the Kleinian group $\pi_1(\orbm)$
for some pared orbifold $\orbm=\orbm_{0}(r;n)$, $\orbm_{1}(r;m)$, or $\orbm_{2}(r;m)$
in Definition \ref{def:Heckoid-orbifold}.
}
\end{enumerate}

\medskip

Agol \cite{Agol} also announced the following classification of parabolic generating pairs 
of the groups in Theorem \ref{main-theorem},
which refines and extends Adams' results that
every hyperbolic $2$-bridge link group has only finitely many 
parabolic generating pairs \cite[Corollary 4.1]{Adams1} and that
the figure-eight knot group has precisely two 
parabolic generating pairs up to equivalence
\cite[Corollary 4.6]{Adams1}.

\begin{theorem}
\label{main-theorem2}
(1) If $\Gamma$ is a hyperbolic $2$-bridge link group, 
then it has precisely two parabolic generating pairs, up to equivalence.

(2) If $\Gamma$ is a Heckoid group,
then it has a unique parabolic generating pair, up to equivalence.
\end{theorem}

Here, by a {\it parabolic generating pair} of a Kleinian group $\Gamma$,
we mean an unordered pair $\{\alpha,\beta\}$ 
of parabolic transformations $\alpha$ and $\beta$
that generate $\Gamma$.
Two parabolic generating pairs $\{\alpha,\beta\}$
and $\{\alpha',\beta'\}$ are said to be {\it equivalent}
if $\{\alpha',\beta'\}$ is equal to $\{\alpha^{\epsilon_1},\beta^{\epsilon_2} \}$
for some $\epsilon_1, \epsilon_2 \in \{\pm1\}$
up to simultaneous conjugacy.
In the companion \cite{ALSS} of this paper
by Shunsuke Aimi, Donghi Lee, Shunsuke Sakai and the fourth author,
an alternative proof of the theorem is given.

Theorems \ref{main-theorem} and \ref{main-theorem2}
are beautifully illustrated by a figure produced by
Yasushi Yamashita upon
request of Caroline Series, 
which is to be included in her article \cite{Series}
in preparation.
The figure is produced by using the results announced in 
\cite[Section 3 of Preface]{ASWY}.
(See also Figure 0.2b in \cite{ASWY}, which was also produced
by Yamashita.)
For further properties of Heckoid groups,
please see the article \cite{APS} in preparation. 

\medskip
This paper is organised as follows.
In Section \ref{sec:2-bridge},
we recall basic facts concerning $2$-bridge links.
In Section \ref{sec:Heckoid},
we give the 
precise definitions of
the Heckoid orbifolds and Heckoid groups.
In Section \ref{sec:classification-dihedral-orbifolds},
we give the classification of dihedral orbifolds,
i.e., good orbifolds with dihedral orbifold fundamental groups
(Theorem \ref{thm:dihedral-orbifold}),
which holds a key to the proof of the main theorem.
In Section \ref{sec:tameness},
we prove the relative tameness theorem for hyperbolic orbifolds
(Theorem \ref{thm:tameness}),
following Bowditch's proof of the tameness theorem for hyperbolic orbifolds (\cite{Bowditch}).
This theorem is used in the treatment of geometrically infinite 
two parabolic generator 
non-free Kleinian groups.
In fact, it turns out there is no such groups.
In Section \ref{sec:orbifold-surgery},
we introduce a convenient method for describing pared orbifolds
(Convention \ref{conv:pared-orbifold}) and
the concept of an orbifold surgery (Definition \ref{def:orbifold-surgery}),
and then prove a simple but useful lemma for orbifold surgeries 
(Lemma \ref{lem:relabeled-orbifold-generic}).
In Section \ref{Sec:canonical-horoball-pair},
we follow Adams \cite{Adams1}, and
recall basic facts concerning two parabolic generator Kleinian groups,
in particular an estimate of the length of parabolic generators
with respect to the maximal cusp (Lemma \ref{Lem:Brenner}).
In Section \ref{sec:outline}, we give an outline of the proof of the main theorem.
Sections \ref{sec:flexible-cusp-case}, \ref{sec:rigid-cusp-case}, and
\ref{sec:exceptional-flexible} are devoted to the proof of the main theorem.
In the appendix, which consists of Sections \ref{sec:dihedral-orbifold} 
and \ref{sec:dihedral-orbifold2},
we give the classification of geometric dihedral orbifolds
that is necessary for the proof Theorem \ref{thm:dihedral-orbifold}. 

\medskip
Throughout this paper, we use the following notation.
\begin{notation}
\label{notation}
{\rm
(1) For an orbifold $\OO$, the symbol
$\pi_1(\OO)$ denotes the orbifold fundamental group of $\OO$,
$H_1(\OO)$ denotes the abelianisation of $\pi_1(\OO)$,
and $H_1(\OO;\ZZ_2)$ denotes $H_1(\OO)\otimes\ZZ_2$.

(2) For a natural number $n$, $\ZZ_n$ denotes the cyclic group 
(or the ring) $\ZZ/n\ZZ$ of order $n$,
and $(\ZZ_n)^{\times}$ denotes the unit group of the ring $\ZZ/n\ZZ$.

(3) By a {\it dihedral group}, we mean a group
generated by two elements of order $2$.
Thus it is isomorphic to the group 
$D_n:=\langle a, b \svert a^2, b^2, (ab)^n\rangle$ for some $n\in\NN\cup\{\infty\}$.
Note that 
$D_n$ has order $2n$ or $\infty$ according to whether $n\in\NN$ or $n=\infty$.
Note also that the order $2$ cyclic group $D_1$ is also regarded as a dihedral group.
}
\end{notation}

\medskip
{\bf Acknowledgement.}
M.S. would like to thank Ian Agol 
for sending the slide of his talk \cite{Agol},
encouraging him (and any of his collaborators) to write up the proof,
and describing key ideas of the proof.
He would also like to thank Michel Boileau
for enlightening conversation in an early time.
His sincere thanks also go to all the other authors 
for joining the project to give a proof to Agol's announcement.
J.P. would like to thank Sadayoshi Kojima for supporting his trip to Japan.
H.A. was supported by JSPS Grants-in-Aid 19K03497.
K.O. was supported by JSPS Grants-in-Aid 17H02843
and 18KK0071.
M.S. was supported by JSPS Grants-in-Aid 15H03620.

\section{Basic facts concerning $2$-bridge links}
\label{sec:2-bridge}

In this section, we recall basic facts concerning 
$2$-bridge links,
which we use in the definitions
of the Heckoid orbifolds and the Heckoid groups.
The description of $2$-bridge links given in this section
is a mixture of 
those in \cite{Bonahon-Siebenmann, Sakuma-Weeks}.

Let $\RotG$ be the group of isometries of the Euclidean plane $\RR^2$ generated by the $\pi$-rotations around the points in $\ZZ^2$. 
Set $\Conways=(\RR^2,\ZZ^2)/\RotG$ and call it the {\it Conway sphere}. 
Then $\PP$ consists of four points in the $2$-sphere $\Conway$.
Let $\PConway:=\Conway-\PP$ be the complementary $4$-times punctured sphere. For each $s\in \QQ\cup\{\infty\}$, 
let $\alpha_s$ as be the simple loop in $\PConway$ obtained as the projection of a line in $\RR^2-\ZZ^2$ of slope $s$. 
Then $\alpha_s$ is {\it essential} in $\PConway$, i.e., it does not bound a disc nor a once-punctured disc in $\PConway$. 
Conversely, any essential simple loop in $\PConway$ 
is isotopic to $\alpha_s$ for a unique $s\in \QQ\cup\{\infty\}$:
we call $s$ the {\it slope} of the essential loop.
For each $s\in \QQ\cup\{\infty\}$, let $\delta_s$ be the pair of
mutually disjoint arcs in $\Conway$ with $\partial\delta_s=\PP$,
obtained as the image of the union of the lines in $\RR^2$ 
which intersect $\ZZ^2$.
Note that the union $\delta_{0/1}\cup\delta_{1/0}$
is a circle in $\Conway$ containing $\PP$,
which divides $\Conway$ into two discs
$\Conway_+:=pr([0,1]\times[0,1])$ and 
$\Conway_-:=pr([1,2]\times[0,1])$,
where $pr:\RR^2\to \Conway$ is the projection.

\begin{figure}
\includegraphics[width=0.75\hsize, bb=0 0 2515 840]{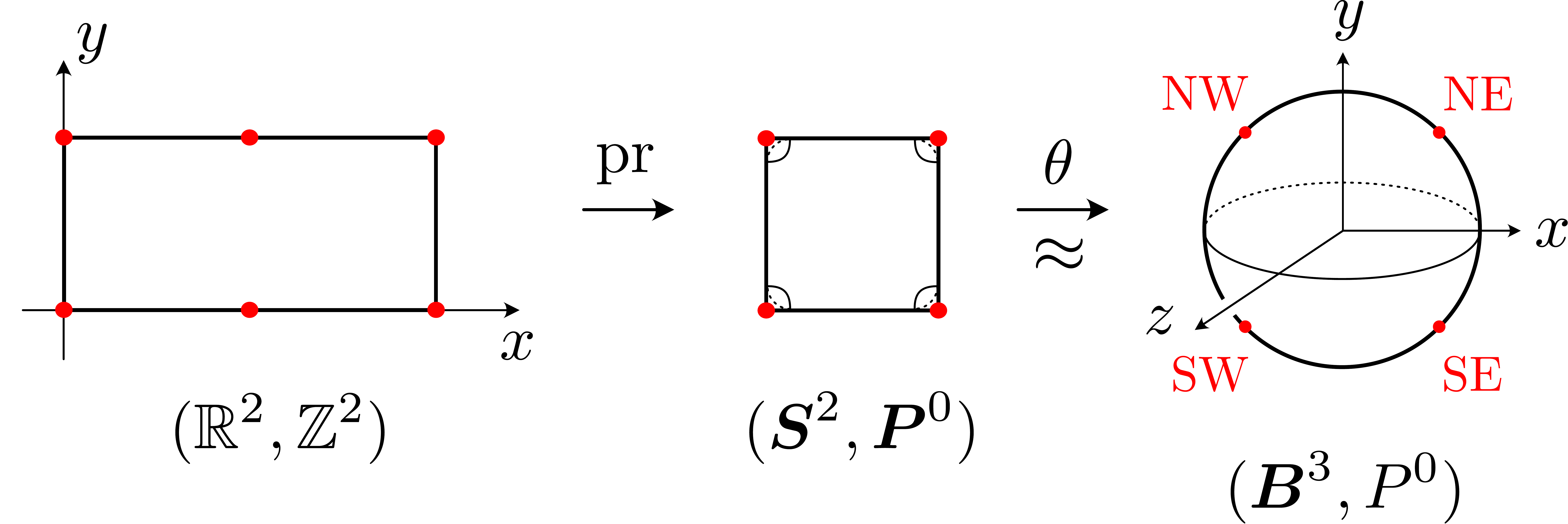}
\caption{Conway sphere $\Conways=(\RR^2,\ZZ^2)/\RotG$
and the homeomorphism $\theta:\Conways \to (\partial\Ball, \mbox{\boldmath$P$}^0)$
}
\label{fig.Ball}
\end{figure}

Let $\Ball:=\{(x,y,z)\in \RR^3 \svert x^2+y^2+z^2\le2\}$ be the round $3$-ball
in $\RR^3\subset\RR^3\cup\{\infty\}\cong S^3$,
whose boundary contains the set $P^0$ consisting of 
the four marked points
\[
\SW:=(-1,-1,0),\quad \SE:=(1,-1,0), \quad
\NE:=(1,1,0), \quad \NW:=(-1,1,0).
\]
Fix a homeomorphism 
$\theta:\Conways \to (\partial\Ball, P^0)$ satisfying the following conditions
(see Figure \ref{fig.Ball}).
\begin{enumerate}
\item
$\theta$ maps the quadruple 
$(pr(0,0), pr(1,0), pr(1,1), pr(0,1))$
to the quadruple
$(\SW,\SE,\NE,\NW)$.
\item
$\theta$ maps the circle $\delta_{0/1}\cup\delta_{1/0}$
to the equatorial circle $\partial\Ball \cap (\RR^2\times\{0\})$,
and maps the hemispheres $\Conway_+$ and $\Conway_-$
onto the hemispheres $\partial\Ball \cap (\RR^2\times\RR_{\ge 0})$
to $\partial\Ball \cap (\RR^2\times\RR_{\le 0})$, respectively.
\item
$\theta$ is equivariant with respect to the natural 
$(\ZZ_2)^2$-actions on $\Conways$ 
and $(\partial\Ball, P^0)$.
Here the natural $(\ZZ_2)^2$-action on $\Conways$ 
is that which lifts to the group
of isometries of the Euclidean plane $\RR^2$ generated by the $\pi$-rotations around the points in $(\frac{1}{2}\ZZ)^2$,
and the natural $(\ZZ_2)^2$-action on $(\partial\Ball, P^0)$
is that generated by the $\pi$-rotations about 
the coordinate axes of $\RR^3$.
\end{enumerate}
We identify $(\partial\Ball, P^0)$ with $\Conways$
through the homeomorphism $\theta$.
Thus for $s\in \QQ\cup\{\infty\}$,
$\alpha_s$ is regarded as an essential simple loop in 
$\partial \Ball-P^0$,
and $\delta_s$ is regarded as a 
union of two disjoint arcs in 
$\partial \Ball$ such that $\partial \delta_s=P^0$.
Moreover, we can choose $\alpha_s$ and $\delta_s$ 
so that they are $(\ZZ_2)^2$-invariant.

For a rational number $r=q/p\in\QQ\cup\{\infty\}$,
let $t(r)$ be a pair of arcs properly embedded in $\Ball$
such that $t(r)\cap\partial\Ball=\partial t(r)= P^0$,
which is obtained from $\delta_r$ by pushing its interior into $\interior\Ball$.
The pair $(\Ball, t(r))$ is called the 
{\it rational tangle of slope $r$}. 
We may assume $t(r)$ is invariant by the natural $(\ZZ_2)^2$-action
on $\Ball$.
In particular, the $x$-axis intersects $t(r)$ transversely in two points:
Let  $\tau_r$ be
the subarc of the $x$-axis they bound, and
call it the {\it core tunnel} of $(\Ball, t(r))$
(see Figure \ref{fig.trivial-tangle}).
Two meridional circles of $t(r)$ near $\partial\tau_r$
together with a subarc of $\tau_r$ forms a graph in $\Ball-t(r)$
homeomorphic to a pair of eyeglasses.
This determines a {\it canonical generating meridian pair} of 
the rank $2$ free group
$\pi_1(\Ball-t(r))\cong \pi_1(\PConway)/\llangle \alpha_r\rrangle$.

By gluing the boundaries of the rational tangles
$(\Ball, t(\infty))$ and $(\Ball, t(r))$ by the identity map,
we obtain a link in the $3$-sphere:
we denote it by $(S^3, K(r))$,
and call it the {\it $2$-bridge link of slope $r=q/p$}.
The number of components, $|K(r)|$, of $K(r)$ is one or two
(i.e., $K(r)$ is a knot or a two-component link)
according to
whether the denominator $p$ is odd or even.
The images of the core tunnels $\tau_{\infty}$ and $\tau_r$
in $(S^3, K(r))$ are called the
{\it upper tunnel} and the {\it lower tunnel} of $K(r)$, respectively.
We denote them by $\tau_+$ and $\tau_-$, respectively.
The canonical generating meridian pairs 
of $\pi_1(\Ball-t(\infty))$ and $\pi_1(\Ball-t(r))$
descend to generating meridian pairs
of the link group 
$\pi_1(S^3-K(r))\cong 
\pi_1(\PConway)/\llangle \alpha_{\infty},\alpha_r\rrangle$:
we call them the {\it upper meridian pair} and 
the {\it lower meridian pair}, respectively.

When we need to care about the orientation
of the ambient $3$-sphere $S^3$,
we regard $(S^3, K(r))$ as being obtained from
$(-\Ball,t(\infty))$ and $(\Ball, t(r))$,
where $\Ball$ inherits the standard orientation of $\RR^3$.
In other words, we identify the ambient $3$-sphere
$S^3$ with the one-point compactification $\RR^3\cup\{\infty\}$
of $\RR^3$,
in such a way that the $\Ball$ containing $t(r)$ is identified 
with the original round ball $\Ball$ via the identity map,
whereas the $\Ball$ containing $t(\infty)$ is identified with 
$\cl(\RR^3\cup\{\infty\}-\Ball)$ via the inversion
$\iota$ in $\partial \Ball$.
Thus $K(r)= t(r)\cup \iota(t(\infty))\subset \RR^3\cup\{\infty\}=S^3$.
Under this orientation convention,
a regular projection is read from the
continued fraction expansion 
\begin{center}\begin{picture}(230,70)
\put(0,48){$\displaystyle{
r=[a_1,a_2, \cdots,a_{n}] =
\cfrac{1}{a_1+
\cfrac{1}{ \raisebox{-5pt}[0pt][0pt]{$a_2 \, + \, $}
\raisebox{-10pt}[0pt][0pt]{$\, \ddots \ $}
\raisebox{-12pt}[0pt][0pt]{$+ \, \cfrac{1}{a_{n}}$}
}} \ ,}$}
\end{picture}\end{center}
in such a way that $a_i$ corresponds to 
the $a_i$ right-hand or left-hand half-twists 
according to whether $i$ is odd or even
(see Figure \ref{fig.2-bridge link diagram}).

\begin{figure}
\includegraphics[width=0.6\hsize, bb=0 0 1913 1570]{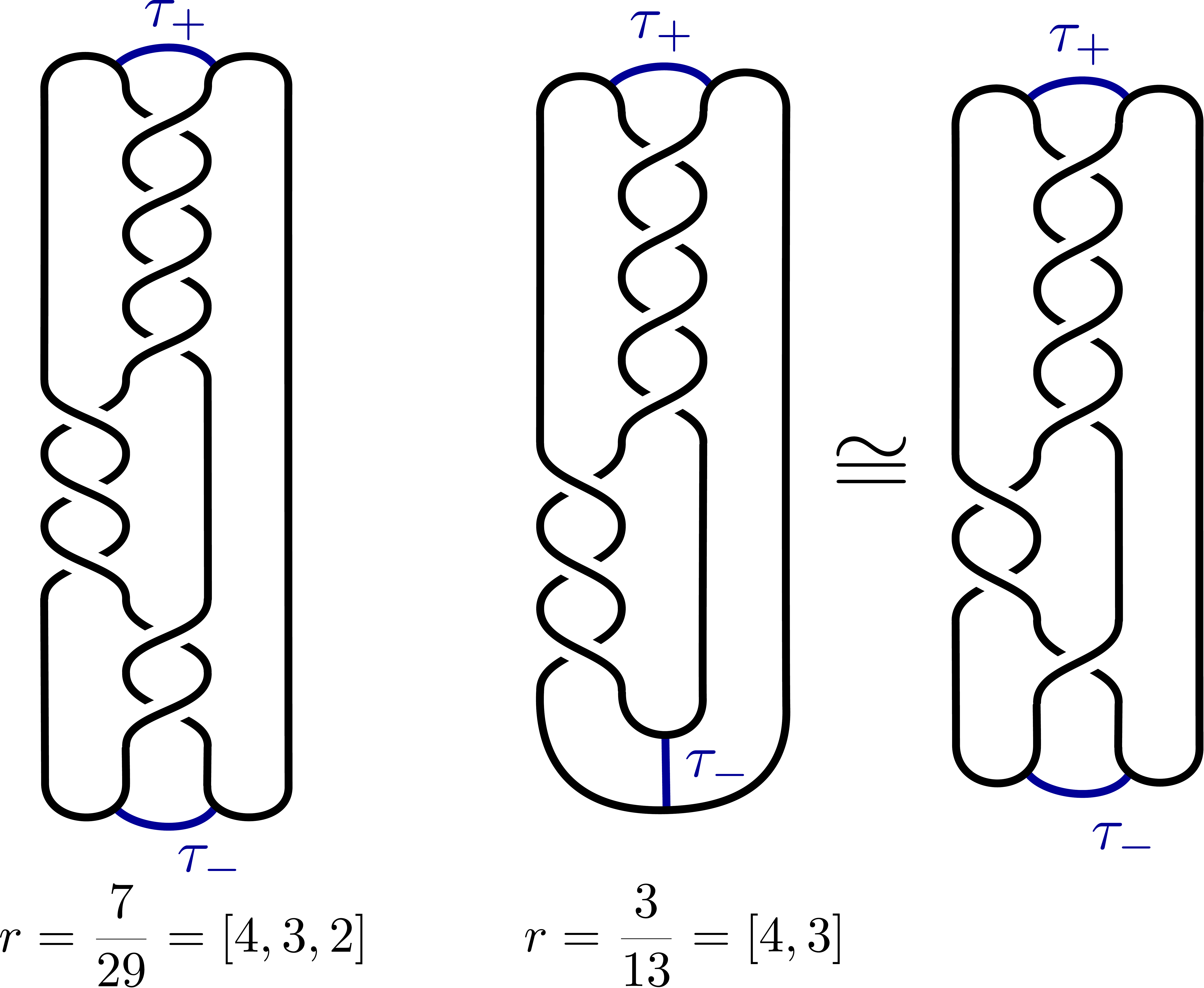}
\caption{$2$-bridge link diagram}
\label{fig.2-bridge link diagram}
\end{figure}

The natural $(\ZZ_2)^2$-actions on
$(\Ball,t(\infty))$ and $(\Ball, t(r))$
can be glued to produce a $(\ZZ_2)^2$-action on $(S^3,K(r))$.
Let $f$ and $h$ be the generators of the action
whose restrictions to $(\Ball,t(\infty)))$ are
the $\pi$-rotations about the $y$-axis and $x$-axis, respectively
(see Figure \ref{fig.trivial-tangle}).
We call $f$, $h$, and $fh$, respectively,
the {\it vertical involution}, the {\it horizontal involution}, and
the {\it planar involution} of $K(r)$.
They are characterized by the following properties.
\begin{enumerate}
\item
$\fix(h)$ contains $\tau_+$, whereas 
each of $\fix(f)$ and $\fix(fh)$ intersects $\tau_+$ transversely in a single point.
\item
The horizontal simple loop $\alpha_0$ in $\partial (\Ball-t(\infty))$
is mapped by $f$ to itself preserving orientation,
and it is mapped by $fh$ to itself reversing orientation.
\end{enumerate}

\begin{figure}
\includegraphics[width=0.4\hsize, bb=0 0 715 721]{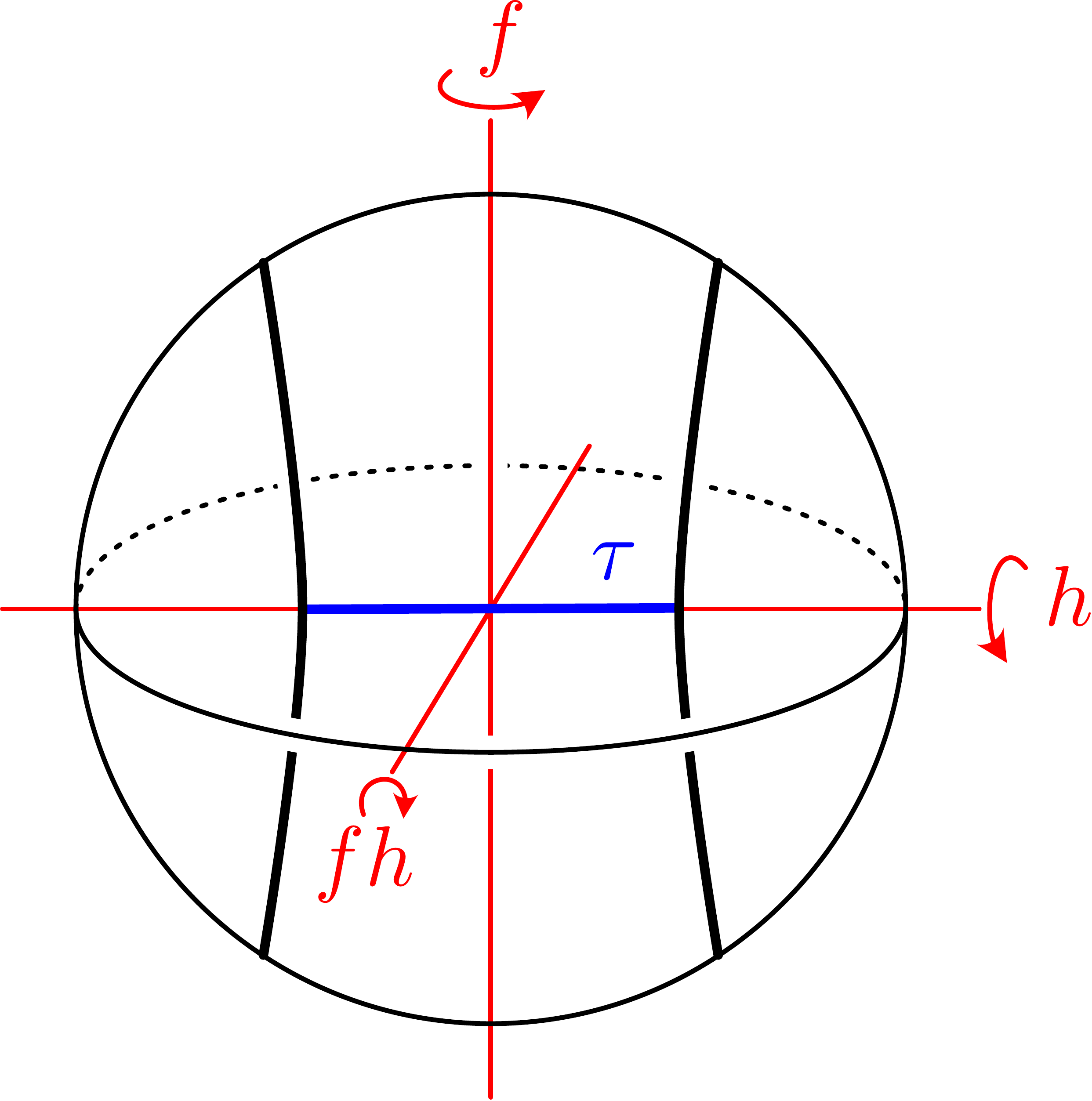}
\caption{Natural $(\ZZ_2)^2$-actions on
$(\Ball,t(\infty))$ consisting of the vertical involution $f$,
the horizontal involution $h$,
and the planar involution $fh$
}
\label{fig.trivial-tangle}
\end{figure}

If the rational number $r=q/p$ satisfies the congruence
$q^2\equiv 1 \pmod p$, then $K(r)$ admits an additional
orientation-preserving symmetry which interchanges 
$(\Ball, t(\infty))$ and $(\Ball, t(r))$.
For a description of such symmetries, see e.g. 
\cite[Sections 4 and 6]{ALSS}, \cite[Section 3]{Sakuma1}.

We finally recall the classification theorem for $2$-bridge links due to 
Schubert \cite{Schubert}
(cf. \cite[Chapter 12]{BZH}).

\begin{proposition}
\label{prop:classification-2-bridgelinks}
For two rational numbers $r=q/p$ and $r'=q'/p'$, with $p$ and $p'$ positive,
the following holds.

(1) There is an orientation-preserving auto-homeomorphism $\varphi$ of $S^3$
which maps $K(r)$ to $K(r')$ if and only if $p=p'$ and either
$q\equiv q' \pmod{p}$ or $qq'\equiv 1 \pmod{p}$.
Moreover the following hold.

\begin{enumerate}
\item[(a)]
If $p=p'$ and $q\equiv q' \pmod{p}$, then
there there is an orientation-preserving auto-homeomorphism $\varphi$ of $S^3$
which maps $(K(r),\tau_+,\tau_-)$ to $(K(r'),\tau_+,\tau_-)$
and 
respects the $(\ZZ_2)^2$-action.
Moreover, the conjugate of the vertical involution of $K(r)$ by $\varphi$
is either the vertical or planar involution of $K(r')$,
according to whether $q'\equiv q \pmod{2p}$ or $q'\equiv q+p \pmod{2p}$.
\item[(b)]
If $p=p'$ and $qq'\equiv 1 \pmod{p}$, then
there there is an orientation-preserving auto-homeomorphism of $S^3$
which maps $(K(r),\tau_+,\tau_-)$ to $(K(r'),\tau_-,\tau_+)$
which respects the $(\ZZ_2)^2$-action.
\end{enumerate}

(2) There is an orientation-reversing auto-homeomorphism $\varphi$ of $S^3$
which maps $K(r)$ to $K(r')$ if and only if $p=p'$ and either
$q\equiv -q' \pmod{p}$ or $qq'\equiv -1 \pmod{p}$.
\end{proposition}

\section{Heckoid orbifolds and Heckoid groups}
\label{sec:Heckoid}

In this section, we recall the definition of Heckoid orbifolds and Heckoid groups
given by \cite[Section 3]{Lee-Sakuma_2013}.

Consider the quotient orbifold
$(\Ball-t(\infty))/(\ZZ_2)^2$,
where $(\ZZ_2)^2$ is the natural action 
illustrated in Figure \ref{fig.trivial-tangle}.
Note that its boundary is identified with 
$\PConway/(\ZZ_2)^2\cong S^2(2,2,2,\infty)$,
which is the quotient of
$\RR^2-\ZZ^2$ by the group generated by the
$\pi$-rotations around the points in $(\frac{1}{2}\ZZ)^2$.
Note that $\pi_1(\PConway)$ 
is identified with a normal subgroup of $\pi_1(\PConway/(\ZZ_2)^2)$
of index $4$.
For each $s\in \QQ\cup\{\infty\}$, 
let $\beta_s$ be the simple loop in $\PConway/(\ZZ_2)^2$ 
obtained as the projection of a line in $\RR^2-(\frac{1}{2}\ZZ)^2$ of slope $s$. 
The simple loop $\alpha_s$ in $\PConway$ doubly covers $\beta_s$,
and so we have $\alpha_s=\beta_s^2$ as conjugacy classes in $\pi_1(\PConway/(\ZZ_2)^2)$.
 
For $r\in\QQ$ and $m \in \NN_{\ge 3}$,
consider the $3$-orbifold
$\orbb(\infty;2):=\cl(\Ball-N(t_{\infty}))/(\ZZ_2)^2$,
attach a 2-handle orbifold $D^2(m)\times I$ to it along the simple loop $\beta_r$.
Since $\beta_r$ divides $\PConway/(\ZZ_2)^2\cong S^2(2,2,2,\infty)$
into $D^2(2,2)$ and $D^2(2,\infty)$,
the resulting $3$-orbifold has a spherical boundary $S^2(2,2,m)\cong S^2/D_m$,
where $D_m$ is the dihedral group of order $2m$ (cf. Notation \ref{notation}(3)).
Cap this spherical boundary with the $3$-handle orbifold
$B^3(2,2,m)\cong B^3/D_m$,
and denote the resulting $3$-orbifold by $\orbh(r;m)$.
(Though this orbifold was denoted by $\OO(r;m)$ 
in \cite{Lee-Sakuma_2013}, 
we employ this symbol,
because we use the symbol $\OO$ to mainly denote spherical dihedral orbifolds.)
Then we have
\[
\pi_1(\orbh(r;m))\cong \pi_1(S^2(2,2,2,\infty))/\llangle \beta_{\infty}^2,\beta_{r}^m\rrangle.
\]
Let $P$ be the annular orbifold $\fr N(t_{\infty})/(\ZZ_2)^2 \cong D^2(2,2)$
on $\partial \orbh(r;m)$,
and continue to denote the orbifold pair $(\orbh(r;m),P)$ by 
the symbol $\orbh(r;m)$.

In \cite[Section 6]{Lee-Sakuma_2013}, it is proved that
the orbifold pair $\orbh(r;m)$
is a pared $3$-orbifold
(see \cite[Definition 8.3.7]{Boileau-Porti}).

\begin{definition}
\label{def:pared-orbifold}
{\rm
An orbifold pair $(M_0,P)$ is a {\it pared $3$-orbifold}
if it satisfies the following conditions
\begin{enumerate}
\item
$M_0$ is a compact, orientable, irreducible $3$-orbifold which is very good
(i.e., $M_0$ has a finite manifold cover).
\item
$P\subset \partial M_0$ is a disjoint union of 
incompressible toric and annular $2$-suborbifolds.
\item
Every rank $2$ free abelian subgroup of $\pi_1(M_0)$ is conjugate to a
subgroup of some $\pi_1(P_i)$, where $P_i\subset P$ is a connected component.
\item
Any properly embedded annular $2$-suborbifold $(A,\partial A)$ of $\orbp$
whose boundary rests on essential loops in $P$ is parallel to $P$.
\end{enumerate}
}
\end{definition}

It is also observed in \cite[Section 6]{Lee-Sakuma_2013} that
$\orbh(r;m)=(\orbh(r;m),P)$ is a
Haken pared orbifold (see \cite[Definitions 8.0.1 and 8.3.7]{Boileau-Porti}).
Hence,  
by the hyperbolization theorem 
of Haken pared orbifolds \cite[Theorem 8.3.9]{Boileau-Porti},
the pared orbifold $\orbh(r;m)$
admits a geometrically finite complete hyperbolic structure,
namely, 
the interior of the orbifold $\orbh(r;m)$ admits a geometrically finite complete hyperbolic structure
such that $P$ represents the parabolic locus
(see Section \ref{sec:tameness} for definitions).

Moreover, such a hyperbolic structure is unique,
because the ends of the non-cuspidal part of $\orbh(r;m)$
are isomorphic to $(\mathrm{a\ turnover})\times [0,\infty)$, which are quasi-isometrically rigid, and every orbifold homeomorphism between two geometrically finite structures preserving the parabolicity in both directions is isotopic to a quasi-isometry, as can be seen by the same argument as Marden's theorem \cite{Marden}.
We denote the unique (up to conjugation) Kleinian group that uniformises the pared orbifold $\orbh(r;m)$ by the symbol $\pi_1(\orbh(r;m))$.

Now the Heckoid groups and the Heckoid orbifolds are defined as follows
\cite[p.242 and Definition 3.2]{Lee-Sakuma_2013}.

\begin{definition}
\label{Def:Heckoid1}
{\rm
For $r\in\QQ$ and $n=\frac{m}{2}\in\frac{1}{2}\NN_{\ge 3}$,
the {\it Heckoid group $G(r;n)$ of slope $r$ and index $n$} 
is the Kleinian group that is obtained as the image of
the natural homomorphism 
\[
\psi: \pi_1(\cl(\Ball-N(t_{\infty})))
\to 
\pi_1(\cl(\Ball-N(t_{\infty}))/(\ZZ_2)^2)
\to
\pi_1(\orbh(r;m))<\PSL(2,\CC).
\]
The {\it Heckoid orbifold $\orbs(r;n)$ of slope $r$ and index $n$}
is the pared orbifold, that is obtained as the covering of 
the pared orbifold $\orbh(r;m)$
associated with the subgroup $G(r;n)<\pi_1(\orbh(r;m))$.
We also denote the Kleinian group $G(r;n)$ by $\pi_1(\orbs(r;n))$.
}
\end{definition}

Then we have the following proposition.
(The main Theorem \ref{main-theorem} implies that 
the converse to the first assertion of the proposition holds.)

\begin{proposition}
\label{prop:two-generator-Heckoid}
For any $r\in\QQ$ and $n=\frac{m}{2}\in\frac{1}{2}\NN_{\ge 3}$,
the Heckoid group is 
a (non-free) Kleinian group with nontrivial torsion which is
generated by two non-commuting parabolic transformations.
Moreover, the image of the conjugacy class of 
the simple loop $\alpha_r$ in $G(r;n)$
is an elliptic transformation of rotation angle $\frac{2\pi}{n}=\frac{4\pi}{m}$.
\end{proposition}

\begin{proof}
Let $\{x,y\}$ be the canonical generating meridian pair
of the rank $2$ free group $\pi_1(\cl(\Ball-N(t_{\infty})))$
(see Section \ref{sec:2-bridge}).
Then $G(r;n)$ is generated by the image 
$\{\psi(x),\psi(y)\}\subset\pi_1(\orbh(r;m))$.
Since $\pi_1(\orbh(r;m))$ is the Kleinian group
which uniformises the pared orbifold $\orbh(r;m)$,
the generating pair of $G(r;n)$ consists of 
non-commuting parabolic transformations.
Since $\alpha_r=\beta_r^2$ and since $\beta_r$ is a meridian of the singular set
of $\orbh(r;n)$ of index $2n=m$,
it follows that $\psi(\alpha_r)$ is an elliptic transformation of rotation angle $\frac{2\pi}{n}=\frac{4\pi}{m}$.
\end{proof}

Next, we recall the topological description of the Heckoid orbifolds.
In Definition \ref{Def:Heckoid1}, 
the Heckoid orbifold $\orbs(r;n)$ is defined as a covering of 
the pared orbifold $\orbh(r;m)$.
Their explicit topological description is
given by \cite[Propositions 5.2 and 5.3]{Lee-Sakuma_2013},
which says that 
the Heckoid orbifold $\orbs(r;n)$ 
is isomorphic to one of the orbifold pairs depicted in Figure \ref{fig.Heckoid-orbifold},
that is specified by the following formula.

\begin{align*}
\orbs(r;n) \cong 
\begin{cases}
\orbm_{0}(r;n)
& \text{if $n\in \NN_{\ge 2}$,}\\
\orbm_{1}(\hat r;m)
& \text{if $n=m/2$ for some odd $m> 2$ and if $p$ is odd,}\\
\orbm_{2}(\hat r;m)
& \text{if $n=m/2$ for some odd $m> 2$ and if $p$ is even,}\\
\end{cases}
\end{align*}
where $\hat r$ is defined from $r=q/p$ by the following rule.
\[
\hat r
=
\begin{cases}
\frac{q/2}{p}
& \text{if $p$ is odd and $q$ is even,}\\
\frac{(p+q)/2}{p}
& \text{if $p$ is odd and $q$ is odd,}\\
\frac{q}{p/2}
& \text{if $p$ is even.}
\end{cases}
\]

\noindent
Thus the following precise definition of the orbifold pairs
in Figure \ref{fig.Heckoid-orbifold} gives an explicit topological picture of
the Heckoid orbifold $\orbs(r;n)$.

\begin{definition}
\label{def:Heckoid-orbifold}
{\rm
(1)
For $r\in\QQ$ and for a positive integer $n\ge 2$,
$\orbm_{0}(r;n)$ denotes the orbifold pair determined by the weighted graph
$(S^3,K(r)\cup\tau_-,w_0)$, where $w_0$ is given by
\[
w_0(K(r))=\infty, \quad w_0(\tau_-)=n. 
\]

(2)
For $r=q/p\in\QQ$ with $p$ odd and an odd integer $m\ge 3$,
$\orbm_{1}(r;m)$ denotes the orbifold pair determined by the weighted graph
$(S^3,K(r)\cup\tau_-,w_1)$, where $w_1$ is given by the following rule.
Let $J_1$ and $J_2$ be the edges of the graph $K(r)\cup \tau_-$  distinct from $\tau_-$.
Then
\[
w_1(J_1)=\infty, \quad w_1(J_2)=2, \quad w_1(\tau_-)=m.
\]

(3)
For $r=q/p\in\QQ$ and an odd integer $m\ge 3$,
$\orbm_{2}(r;m)$ denotes the orbifold pair determined by the weighted graph
$(S^3,K(r)\cup\tau_+\cup\tau_-,w_2)$, where $w_2$ is given by the following rule.
Let $J_1$ and $J_2$ be unions of two mutually disjoint edges of the graph
$K(r)\cup\tau_+\cup\tau_-$ distinct from $\tau_{\pm}$.
Moreover, if $p$ is even, then 
both $J_1$ and $J_2$ are preserved by the vertical involution $f$ of $K(r)$.
(Thus $f$ interchanges the two components of each of $J_1$ and $J_2$.)
Then 
\[
w_2(J_1)=\infty, \quad w_2(J_2)=2, \quad w_2(\tau_+)=2, \quad w_2(\tau_-)=m.
\]
}
\end{definition}

In Defunition \ref{def:Heckoid-orbifold}(3),
the \lq identity' $w_2(J_1)=\infty$ means that
$w_2(e)=\infty$ for each edge $e$ contained in $J_1$.
Similarly, $w_2(J_2)=2$ means that $w_2(e)=2$ for each edge $e$ contained in $J_2$.
We employ this kind of convention throughout the paper.

\begin{remark}
\label{rem.well-defined-orbifold}
{\rm
(1) Because of the $(\ZZ_2)^2$-symmetry of $2$-bridge links,
the choice of the edges $J_1$ and $J_2$ in (2) and (3) 
does
not affect the isomorphism class of the resulting orbifolds
(see \cite[Remark 5.4]{Lee-Sakuma_2013}).

(2) Suppose $p$ is odd.
Then, in the definition of $\orbm_{2}(r;m)$,
the disjointness condition of $J_1$ and $J_2$ 
determines the pair $(J_1,J_2)$ up to the horizontal involution $h$
of $(S^3,K(r)\cup\tau_+\cup\tau_-)$. 
Moreover, according to whether $q$ is odd or even,
both $J_1$ and $J_2$ are preserved by $f$ or $fh$, respectively
(see Figure \ref{fig.Heckoid-orbifold2}(1),(2)).

(3) Suppose $p$ is even.
Then, in the definition of $\orbm_{2}(r;m)$,
the condition that both $J_1$ and $J_2$ are preserved by $f$ is not essential
in the following sense.
Let $J_{i,1}$ and $J_{i,2}$ be the components of $J_{i}$ for $i=1,2$,
such that $J_{1,1}\cap J_{2,1}=\emptyset$ and $J_{1,2}\cap J_{2,2}=\emptyset$.
Set $J_1'=J_{1,1}\cup J_{2,1}$ and $J_2'=J_{1,2}\cup J_{2,2}$.
Then $J_1'$ and $J_2'$ 
are unions of two mutually disjoint edges of the graph
$K(r)\cup\tau_+\cup\tau_-$ distinct from $\tau_{\pm}$,
such that
both $J_1'$ and $J_2'$ are preserved by 
the planar involution $fh$, instead of the vertical involution $f$
(see Figure \ref{fig.Heckoid-orbifold2}(3),(4)).
Let $w_2'$ be the weight function on the graph $K(r)\cup\tau_+\cup\tau_-$ defined by
\[
w_2'(J_1')=\infty, \quad w_2'(J_2')=2, \quad w_2'(\tau_+)=2, \quad w_2'(\tau_-)=m.
\]
Then $(S^3,K(r)\cup\tau_+\cup\tau_-,w_2')$
represents the orbifold $\orbm_{2}(r';m)$,
where $r'=(p+q)/p$ for $r=q/p$.
This follows from the fact that
there is a homeomorphism from $(S^3,K(r)\cup\tau_+\cup\tau_-)$ to
$(S^3,K(r')\cup\tau_+\cup\tau_-)$
sending $(\tau_{\pm}, J_1, J_2)$
to $(\tau_{\pm}, J_1', J_2')$
(see Proposition \ref{prop:classification-2-bridgelinks}(1a)).
}
\end{remark}

\begin{figure}
\includegraphics[width=0.9\hsize, bb=0 0 2915 784]{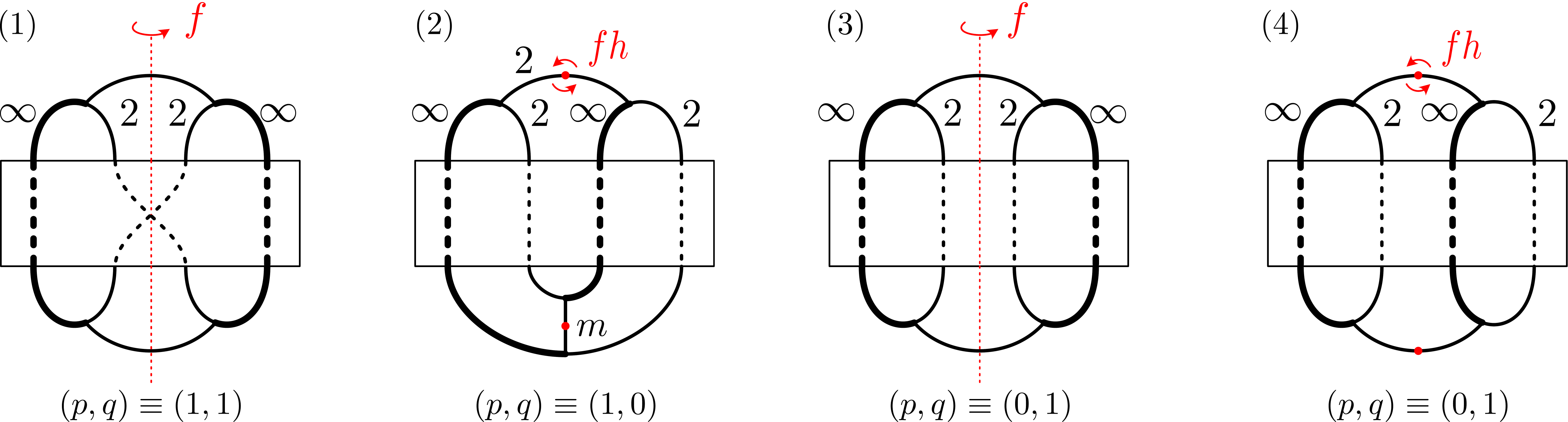}
\caption{
The first two figures (1) and (2) illustrate
Remark \ref{rem.well-defined-orbifold}(2),
and the last two figures (3) and (4) illustrate
Remark \ref{rem.well-defined-orbifold}(3).
See also \cite[Figures in Section 7]{ALSS}.
}
\label{fig.Heckoid-orbifold2}
\end{figure}

\section{Classification of dihedral orbifolds}
\label{sec:classification-dihedral-orbifolds}
In this section, we give a classification of the dihedral orbifolds,
which plays a key role in the proof of the main theorem.
We refer to \cite{BLP, BMP, CHK} 
for standard terminologies for orbifolds.

By using the the orbifold theorem, 
the geometrisation theorem of compact 
orientable $3$-manifolds, and 
the classification of geometric dihedral orbifolds (see Appendix),
we obtain the following classification of good orbifolds
with dihedral orbifold fundamental groups.

\begin{theorem}
\label{thm:dihedral-orbifold}
Let $\OO$ be a compact orientable $3$-orbifold
with nonempty singular set
satisfying the following conditions.
\begin{enumerate}
\item[{\rm(i)}]
$\OO$ does not contain a bad $2$-suborbifold.
\item[{\rm(ii)}]
Any component of $\partial\OO$ is not spherical.
\item[{\rm(iii)}]
$\pi_1(\OO)$ is a dihedral group.
\end{enumerate}
Then $\OO$ is isomorphic to one of the following orbifolds.

\begin{enumerate}
\item
The spherical dihedral orbifold $\OO(r;d_+,d_-)$
represented by the weighted graph $(S^3,K(r)\cup\tau_+\cup\tau_-,w)$
for some $r\in\QQ$ and coprime positive integers $d_+$ and $d_-$,
where $w$ is given by the following rule
(see Figure \ref{fig.Dihedral-orbifold}).
\[
w(K(r))=2, \quad w(\tau_+)=d_+, \quad w(\tau_-)=d_-.
\]
\item
The $S^2\times\RR$ orbifold $\OO(\infty)$ 
represented by the weighted graph $(S^3,K(\infty),w)$,
where $w$ takes the value $2$ at 
each component of
the $2$-bridge link $K(\infty)$ of slope $\infty$, i.e.
the  $2$-component trivial link.
\item
The $S^2\times\RR$ orbifold $\OO(\RP^3,O)$ 
represented by the weighted graph $(\RP^3, O,w)$,
where $O$ is the trivial knot in the projective $3$-space $\RP^3$ with $w(O)=2$.
\item
The orbifold $D^2(2,2)\times I$.
\end{enumerate}
\end{theorem}

\begin{remark}
\label{rem:pi-orbifold}
{\rm
For the orbifold $\OO(r;d_+,d_-)$,
if $d_+=1$ (resp. $d_-=1$), then $\tau_+$ (resp. $\tau_-$)
does not belong to the singular set (cf. Convention \ref{conv:pared-orbifold2}(1)).
In particular, $\OO(r):=\OO(r;1,1)$ is the $\pi$-orbifold associated with the $2$-bridge link $K(r)$
in the sense of \cite{Boileau-Zimmermann},
i.e. the orbifold with underlying space $S^3$ and with singular set $K(r)$,
whose index is $2$. 
In Adam's classification of torsion-free Kleinian groups generated by two parabolic 
transformations 
\cite[Theorem 4.3]{Adams1}, the $\pi$-orbifolds $\OO(r)$ played a key role, whereas
the orbifolds $\OO(r;d_+,d_-)$ play the corresponding key role in this paper. 
}
\end{remark}

\begin{figure}
\includegraphics[width=0.25\hsize, bb=0 0 559 672]{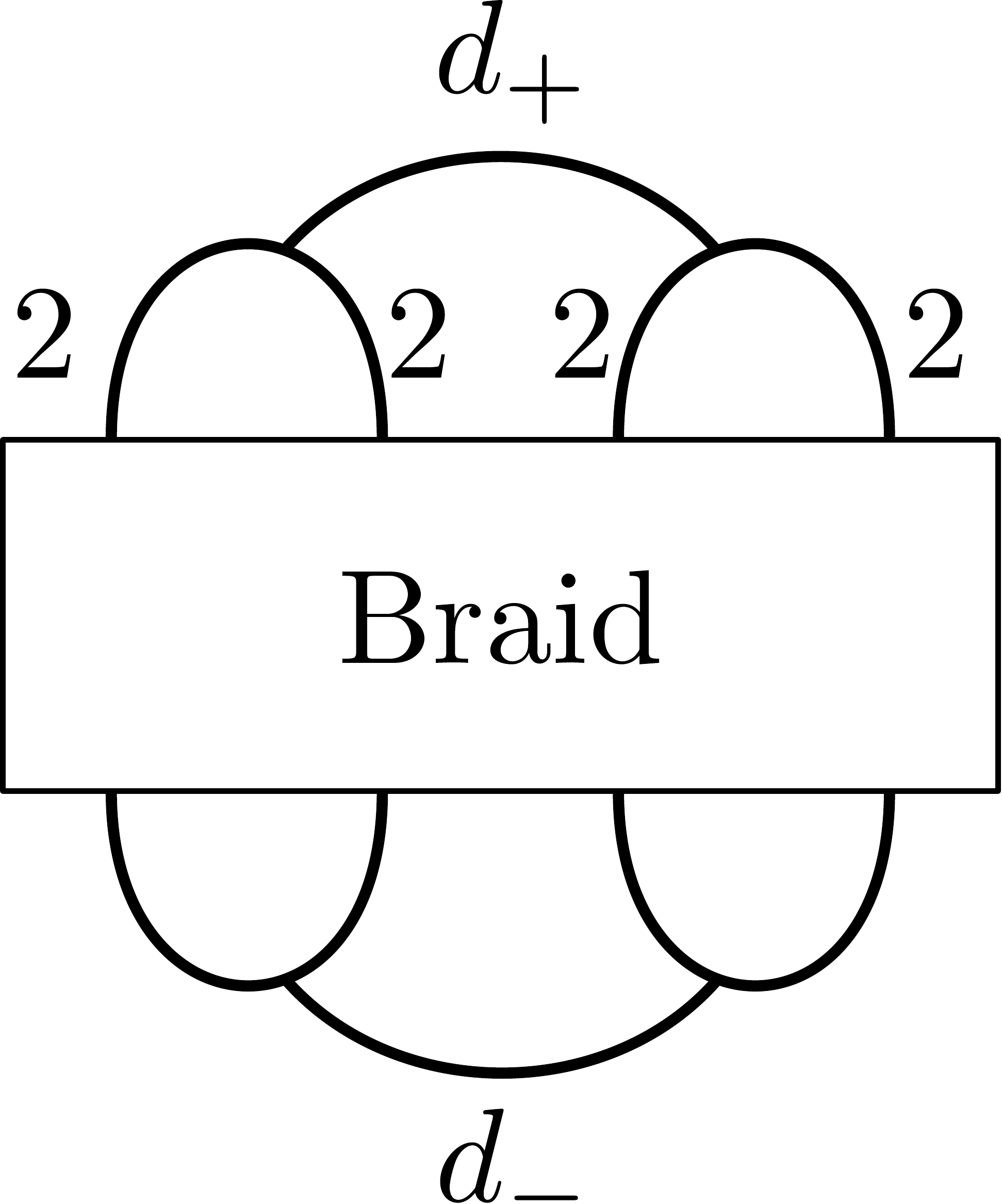}
\caption{The spherical dihedral orbifold $\OO(r;d_+,d_-)$}
\label{fig.Dihedral-orbifold}
\end{figure}

\begin{proof}
Let $\OO$ be a $3$-orbifold satisfying the three conditions.
We first treat the case where 
$\OO$ is irreducible, i.e.,
any spherical $2$-suborbifold of $\OO$ bounds a {\em discal $3$-suborbifold}
(a quotient of a $3$-ball by a finite orthogonal group).
We can observe that $\OO$ is topologically atoroidal as follows.
Suppose on the contrary that $\OO$
contains an essential toric suborbifold $F$.
Then the inclusion map induces an injective homomorphism
from $\pi_1(F)$ into $\pi_1(\OO)$, as explained below.
Since $\OO$ does not contain a bad $2$-suborbifold by the condition (i),
$\OO$ is very good,  
by \cite[Corollary 1.3]{BLP}. 
Thus by applying the equivariant loop theorem to the group action,
$\pi_1(F)$ embeds into $\pi_1(\OO)$ (see \cite[Corollary 3.20]{BMP}). 
This contradicts the fact that
the dihedral group $\pi_1(\OO)$ 
does not contain $\ZZ^2$.

Hence, by the orbifold theorem \cite[Corollary 1.2]{BLP},
$\OO$ is geometric,
i.e., either $\interior \OO$ admits one of Thurston's geometry
or $\OO$ is a discal $3$-orbifold.
The latter possibility does not happen by the assumption (ii),
and so $\interior \OO$ admits one of Thurston's geometry.
If the geometry is $S^3$, then by Proposition \ref{prop:dihedral-orbifold1}, 
$\OO$ is isomorphic to the orbifold $\OO(r;d_+,d_-)$ in (1).
If the geometry is $S^2\times \RR$,
then by Proposition \ref{prop:dihedral-orbifold2},
$\OO$ is isomorphic to the orbifold $\OO(\infty)$ in (2)
or the orbifold $\OO(\RP^3,O)$ in (3).
(But this does not happen,
because these orbifolds are reducible
whereas we currently assume that $\OO$ is irreducible.)
If the geometry is one of the remaining 6 geometries,
then by  Proposition \ref{prop:dihedral-orbifold3},
$\OO$ is isomorphic to the orbifold $D^2(2,2)\times I$ in (4).

Next, we treat the case when $\OO$ is reducible.
Note that 
$\OO$ does not contain a non-separating spherical $2$-suborbifold,
because $H_1(\OO)$ is finite.
Thus we do not need to worry about the paradoxical problems 
concerning spherical splitting of $3$-orbifolds
pointed out by Petronio \cite{Petronio}.
By \cite[Theorems 0.1]{Petronio},
there is a finite system of spherical $2$-suborbifolds $\orbspherical$
such that (a) no component of $\OO-\orbspherical$ is {\em punctured discal}
(a discal $3$-orbifold minus regular neighbourhoods of a finite set)
and 
(b) all {\em prime factors} of $\OO$
(the orbifolds obtained from the components of $\OO-\orbspherical$
by capping the boundary components with discal orbifolds) are irreducible.
It should be noted that some prime component may be a manifold,
i.e., its branching locus is empty.
By the geometrisation theorem of compact orientable $3$-manifolds (see e.g. \cite{BBMBP}) 
and the the geometrisation theorem of compact orientable $3$-orbifolds
(see e.g. \cite[Theorem 3.27]{BMP}),
each prime component of $\OO$ admits a canonical decomposition
into geometric pieces by a family of essential toric $2$-orbifolds.
In particular, each prime factor has a nontrivial orbifold fundamental group.
Since 
the only nontrivial free product decomposition of a dihedral group
is the decomposition of the infinite dihedral group $D_{\infty}$ into
the free product $\ZZ_2 * \ZZ_2$,
$\OO$ is the connected sum (along a $2$-sphere with empty branching set)
of two irreducible $3$-orbifolds $\OO_1$ and $\OO_2$,
such that $\pi_1(\OO_i)\cong \ZZ_2$.
Since $\OO_i$ is geometric, $\OO_i$ is isomorphic to
(a) the discal $3$-orbifold $B^3/\ZZ_2$,
(b) the orbifold $(S^3,O,w)$, where $O$ is a trivial knot and $w(O)=2$, 
or (c) $\RP^3$.
By condition (ii), $\OO_i$ cannot be a discal orbifold.
Since $\OO=\OO_1\#\OO_2$ has nonempty ramification locus,
at least one of $\OO_i$ is not isomorphic to $\RP^3$.
Hence, $\OO$ is isomorphic to 
the orbifold $(S^3,O,w) \# (S^3,O,w)\cong \OO(\infty)$ in (2)
or the orbifold $(S^3,O,w) \# \RP^3\cong \OO(\RP^3,O)$ in (3).
\end{proof}

\begin{remark}
\label{remark:upside-down}
{\rm
By considering the image of $\OO(r;d_+,d_-)$
by a $\pi$-rotation about a horizontal axis in Figure \ref{fig.Dihedral-orbifold},
we can interchange the role of $d_+$ and $d_-$.
To be precise, we can see from Proposition \ref{prop:classification-2-bridgelinks}(1b)
that $\OO(q/p;d_+,d_-)\cong \OO(q'/p;d_-,d_+)$
if $qq'\equiv 1 \pmod p$.
}
\end{remark}

\section{Relative tameness theorem for hyperbolic orbifolds}
\label{sec:tameness}
We first recall basic terminology for hyperbolic orbifolds,
following \cite[Chapter 6]{BMP}.
Let $\Gamma$ be a finitely generated Kleinian group
and $M=\HH^3/\Gamma$ the quotient hyperbolic orbifold.
For a real number $\epsilon>0$,
the $\epsilon$-thin part $M_{(0,\epsilon]}$ of $M$
is the set of all points $x\in M$
such that $d(\tilde x, \gamma\tilde x)\leq \epsilon$
for some lift $\tilde x$ of $x$ to $\HH^3$ and some $\gamma\in\Gamma$
of order $>1/\epsilon$ (including $\infty$). 
By the Margulis Lemma, there is a constant $\mu>0$, such that
for any real number $\epsilon\in (0,\mu]$,
each component $X$ of $M_{(0,\epsilon]}$ is either a 
Margulis tube or a cuspidal end.
Here a {\it Margulis tube} is a compact quotient of the $r$-neighbourhood
of a geodesic in $\HH^3$ by an elementary subgroup of $\Gamma$
which preserves the geodesic,
and a {\it cuspidal end} is the quotient of a horoball in $\HH^3$
by an elementary parabolic subgroup of $\Gamma$
which preserves the horoball.

Topologically, a cuspidal end is a product $F\times [0,+\infty)$,
where $F$ is a Euclidean $2$-orbifold.
Thus we have the following possibilities for $F$.
\begin{enumerate}
\item
$F$ is the open annulus $S^1\times \RR$
or $S^2(2,2,\infty)$, the quotient of $S^1\times \RR$ by an involution.
\item
$F$ is the torus $T^2$
or $S^2(2,2,2,2)$, the quotient of $T^2$ by an involution.
\item
$F$ is 
$S^2(2,3,6)$, $S^2(2,4,4)$ or $S^2(3,3,3)$,
the quotient of $T^2$ by a finite cyclic group action
of order $6$, $4$ or $3$, respectively.
\end{enumerate}
A cusp $F\times [0,+\infty)$ is said to be {\it rigid}
if $F\cong S^2(2,3,6)$, $S^2(2,4,4)$ or $S^2(3,3,3)$.
Otherwise it is said to be {\it flexible}.
It is well-known that a cusp $F\times [0,+\infty)$ is rigid
if and only if the holonomy representation of 
the orbifold fundamental group $\pi_1(F\times [0,+\infty)$
admits no nontrivial deformation 
(see \cite[Proposition 1]{Matsuzaki}).

Let $M^{\mathrm{cusp}}_{(0,\epsilon]}$ be the union of 
the cuspidal ends of $M_{(0,\epsilon]}$,
and let $M_0:=M-\interior M^{\mathrm{cusp}}_{(0,\epsilon]}$
be the non-cuspidal part of $M$.
Then $P:=\partial M_0$ is a disjoint union of
euclidean $2$-orbifolds,
and is called the {\it parabolic locus} of $M_0$.
Note that $M\cong \interior M_0$ and that $P$ consists of ({\it closed}) toric orbifolds
(closed $2$-orbifolds obtained as quotients of the $2$-dimensional torus)
and {\it open} annular orbifolds
(open $2$-orbifolds obtained as quotients of 
the open annulus $S^1\times \RR$).

The following theorem is an orbifold version
of (the relative version of) the tameness theorem
established by Agol \cite{Agol} 
and Calegari-Gabai \cite{Calegari-Gabai}
(see also Soma \cite{Soma} and Bowditch \cite{Bowditch}).

\begin{theorem}
\label{thm:tameness}
Let $M=\HH^3/\Gamma$ be a hyperbolic $3$-orbifold
with finitely generated orbifold fundamental group $\Gamma$.
Then there is a compact $3$-orbifold $\bar M_0$
and a compact suborbifold $\bar P$ of $\partial \bar M_0$,
such that (i) $\interior \bar M_0=\interior M_0 \cong M$
and 
(ii) the interior of $\bar P$ in $\partial \bar M_0$
is equal to $P=\partial M_0$.
\end{theorem}

\begin{proof}
We give a proof following the arguments of 
Bowditch \cite[Section 6.6]{Bowditch}
(cf. \cite[Lemma 14.3]{Agol}).
By Selberg's lemma, 
$M$ admits a finite regular 
manifold cover,
namely there is a complete hyperbolic manifold $N$
and a finite group $G$ of orientation-preserving isometries
of $N$ such that $N/G \cong M$.
The inverse image, $N_0$, of $M_0$ in $N$
forms a $G$-invariant non-cuspidal part of $N$,
and we have $N_0/G\cong M_0$.
By the relative version of the tameness theorem
\cite[Theorem 7.3]{Calegari-Gabai} 
(cf. \cite[Section 6]{Bowditch}),
there is a compact $3$-manifold $\bar N_0$
and a compact submanifold $\bar Q$ of $\partial \bar N_0$,
such that (i) $\interior \bar N_0=\interior N_0$ and
(ii) the interior of $\bar Q$ in $\partial \bar N_0$
is equal to $\partial N_0$.
Let $D(N_0)$ and be the double of $N_0$ along $\partial N_0$.
Then the action of $G$ on $N_0$ extends to an action on $D(N_0)$,
and $D(N_0)/G$ is isomorphic to 
the double, $D(M_0)$, of $M_0$ along $\partial M_0$.
Consider the double, $D(\bar N_0)$, of $\bar N_0$ along $\bar Q$.
Then $D(\bar N_0)$ is a compact manifold
with interior $D(N_0)$.
By \cite[Theorem 8.5]{Meeks-Scott},
the action of $G$ on $D(N_0)$
extends to an action on $D(\bar N_0)$,
and $\interior(D(\bar N_0)/G)=D(N_0)/G$ is identified with $D(M_0)$.
Let $\bar M_0$ be the closure in $D(\bar N_0)/G$
of one of the two copies of $M_0$ in $D(M_0)\subset D(\bar N_0)/G$,
and let $\bar P$ be the image of $\bar Q\subset D(\bar N_0)$ in
$D(\bar N_0)/G$.
Then the pair $(\bar M_0,\bar P)$ satisfies the desired conditions.
\end{proof}

The above theorem together with the following theorem
enables us to reduce the treatment of geometrically infinite case
to that of geometrically finite case.

\begin{theorem}
\label{thm:geom-infinite}
Under the setting of Theorem \ref{thm:tameness},
$(\bar M_0, \bar P)$ is a pared orbifold.
Moreover, the pared orbifold $(\bar M_0, \bar P)$ 
admits a geometrically finite
complete hyperbolic structure.
Namely, there is a geometrically finite Kleinian group $\Gamma'$ 
such that 
(i) the orbifold $\HH^3/\Gamma'$
is isomorphic to the orbifold $\interior \bar M_0\cong M$
and 
(ii) $P$ is the parabolic locus of $\Gamma'$.
\end{theorem}

\begin{proof}
The first assertion that $(\bar M_0,\bar P)$ is a pared orbifold 
can be proved as in the proof of 
\cite[Corollary 6.10 in Chapter V]{Morgan-Bass}.
So we prove the second assertion that 
the pared orbifold $(\bar M_0, \bar P)$ 
admits a geometrically finite hyperbolic structure.
If the orbifold $\bar M_0$ is Haken in the sense of \cite[Definition 8.0.1]{Boileau-Porti}
then it follows from 
\cite[Theorem 8.3.9]{Boileau-Porti}
that the pared orbifold $(\bar M_0,\bar P)$ admits a geometrically finite hyperbolic structure, as desired.
So we may assume the orbifold $\bar M_0$ is non-Haken,
i.e., either it contains no essential
$2$-suborbifold or it contains an essential turnover.
In the first case, $\partial M_0$ consists only of turnovers
by \cite[Proposition 9.4]{BMP}.
This implies that every end of $M\cong \interior \bar M_0$ has a neighbourhood
isomorphic to the product of $(\mbox{a turnover})\times [0,\infty)$.
Since a hyperbolic turnover is always realised by a totally geodesic surface, each end has a neighbourhood containing no closed geodesics.
Thus every end of the hyperbolic orbifold $M$ is geometrically finite and rigid.
Thus $M$ admits a unique complete hyperbolic structure, 
and it is geometrically finite.
In the latter case, by the turnover splitting theorem \cite[Theorem 4.8]{BMP},
$\bar M_0$ admits a decomposition by a finite
disjoint family of essential hyperbolic turnovers 
into Haken orbifolds and small orbifolds.
By the orbifold theorem,
each piece admits a geometrically finite hyperbolic structure,
respecting the parabolic locus.
By gluing these hyperbolic structures along the totally geodesic
hyperbolic turnovers, we obtain a geometrically finite 
hyperbolic structure on $(\bar M_0,\bar P)$.
\end{proof}

\begin{remark}
{\rm
In \cite{Agol}, Agol suggested to prove the last assertion of 
Theorem \ref{thm:geom-infinite}
by using a relative version
of the work of Feighn and Mess \cite[Theorem 2]{Feighn-Mess}
which proves the existence of a compact core
of an orbifold $M=\HH^3/\Gamma$
with a finitely generated orbifold fundamental group $\Gamma$.
Such a relative version is proved by
Matsuzaki \cite[Lemma 2]{Matsuzaki}
under the assumption that   
$\Gamma$ is indecomposable 
(over finite cyclic groups and with respect to the parabolic subgroups) 
in the sense of \cite[Definition in p.26]{Matsuzaki}.
But we are not sure if non-free two-parabolic generator Kleinian groups 
satisfy this property. 
Though Theorem \ref{thm:tameness}, 
which is proved by using the deep tameness theorem,
of course, guarantees the existence of a relative core
of complete hyperbolic orbifolds with finitely generated fundamental groups,
we are not sure if more \lq elementary' proof is possible.
}
\end{remark}

\section{Orbifold surgery}
\label{sec:orbifold-surgery}

In this section, we introduce a convenient method 
for representing 
pared orbifolds by weighted graphs,
generalising the convention in the introduction
(Convention \ref{conv:pared-orbifold}).
Then we introduce the concept of an orbifold surgery
(Definition \ref{def:orbifold-surgery}),
which is a
key ingredient of the proof of the main theorem,
and prove a basic Lemma \ref{lem:relabeled-orbifold-generic} 
for the orbifold surgery.
At the end of this section, we also state
another basic Lemma \ref{lem:homology}
concerning the $\ZZ_2$-homology of an orbifold,
which is repeatedly used in the proof of the main theorem.

\begin{convention}
\label{conv:pared-orbifold}
{\rm
Consider a triple $(W,\Sigma, w)$,
where $W$ is a compact oriented $3$-manifold, 
$\Sigma$ is a finite trivalent graph properly embedded in $W$,
and $w$ is a function on the edge set of $\Sigma$ which takes value in
$\NN_{\ge 2}\cup\{\infty\}$.
Here, a loop component of $\Sigma$ is regarded as a single edge, 
$\Sigma\cap\partial M$ is the set of degree $1$ vertices of $\Sigma$,
and all other vertices have degree $3$.
For each edge $e$ of $\Sigma$, its value $w(e)$ by $w$ is called 
the {\it weight} of the edge.
We call the triple $(W,\Sigma, w)$ a {\it weighted graph} and
call $w$ the {\it weight function} of the weighted graph.
Let $\Sigma_{\infty}$ 
be the subgraph of $\Sigma$ consisting of the edges 
with weight $\infty$, 
and let $\Sigma_s$ be the subgraph of $\Sigma$ consisting of the edges 
with integral weight.

We regard each component, $F$, of $\partial W$ as a $2$-orbifold as follows:
the underling space is the complement of an open regular neighbourhood of 
$F\cap \Sigma_{\infty}$ in $F$,
and the singular set is $F\cap \Sigma_s$, where the index of a singular point 
is given by the weight of the corresponding edge of $\Sigma_s$.
We assume that the following condition (SC) is satisfied.

\begin{itemize}
\item[(SC)]
For any sphere component $S$ of $\partial W$,
the corresponding $2$-orbifold
is not a bad orbifold, a spherical orbifold, a discal orbifold, nor an annulus.
Namely, 
(i) $|S\cap W|\ge 3$ 
and (ii) if $|S\cap W|= 3$ then $\sum_{i=1}^3\frac{1}{w(e_i)}\le 1$,
where 
$e_i$ ($i=1,2,3$) are the (germs of) edges of $\Sigma$ 
which have an endpoint in $F$.
\end{itemize}

A trivalent vertex $v$ of $\Sigma$ is said to be 
{\it spherical}, {\it euclidean} or {\it hyperbolic}
according to whether
$\sum_{i=1}^3\frac{1}{w(e_i)}$ is bigger than, equal to, or smaller than $1$,
where 
$e_i$ ($i=1,2,3$) are the (germs of) edges incident on $v$.
Let $V_E$ (resp. $V_H$) be the set of the euclidean (resp. hyperbolic) vertices. 

Let $M_0$ be the complement of an open regular neighbourhood of 
$\Sigma_{\infty}\cup V_E\cup V_H$ in $M$.
Then $M_0$ has the structure of an orbifold,
with singular set $\Sigma_0:=M_0\cap \Sigma_s$,
where the indices of the edge of $\Sigma_0$ are given by $w$. 

For each edge $e$ of $\Sigma_{\infty}$, let $m_e\subset \partial M_0$
be a meridian loop of $e$,
let $P_{\infty}$ be the disjoint union
of the regular neighbourhoods in $\partial M_0$ of 
$m_e$, where $e$ runs over the edges of $\Sigma_{\infty}$.
The condition (SC)
implies that  
each component of $\cl(\partial M_0-P_{\infty})$ is
either a euclidean or hyperbolic $2$-orbifold.
Let $P$ be the union of $P_{\infty}$
and the euclidean components of $\cl(\partial M_0-P_{\infty})$.
Then $P$ is a disjoint union of euclidean $2$-orbifolds.

We call $\orbp$ the {\it orbifold pair determined by the weighted graph $(M,\Sigma,w)$}.
}
\end{convention}

\begin{convention}
\label{conv:pared-orbifold2}
{\rm
It is sometimes convenient to employ the following slight extension of
Convention \ref{conv:pared-orbifold}.

(1) 
We allow $w$ to have an edge $e$ with $w(e)=1$.
In this case, we consider the weighted graph $(W,\Sigma',w')$,
where $\Sigma'$ is the subgraph of $\Sigma$ consisting of 
those edges with $w(e)\ne 1$
and $w'$ is the restriction of $w$ to $\Sigma'$.
If $\Sigma'$ is also trivalent graph properly embedded in $W$
and the condition (SC) is satisfied,
then we define the orbifold pair determined by $(W,\Sigma,w)$
to be that determined by $(W,\Sigma',w')$.

(2)
We allow a quadrivalent vertex, $v$,  
such that the four edge germs incident on it have index $2$.
In this case, $v$ represents a parabolic locus, $P(v)$,
isomorphic to $S^2(2,2,2,2)$.
}
\end{convention}

A key ingredient of the proof of the main theorem 
is an orbifold surgery.

\begin{definition}
\label{def:orbifold-surgery}
{\rm
Let $\orbp$ be a pared orbifold, 
represented by a weighted graph $(W,\Sigma,w)$
satisfying the condition (SC).
By replacing the weight function $w$ with another weight function $w'$
(which also takes value in $\NN_{\ge 2}\cup\{\infty\}$),
we obtain another weighted graph $(W,\Sigma,w')$.
This fails to satisfy the condition (SC)
only when some sphere component $S$ of the topological boundary $\partial W$
determines a spherical $2$-orbifold with three singular points.
In this case, we cap all such sphere boundaries of $W$
with a cone over $(S,S\cap\Sigma)$ to obtain a new 
weighted graph,
which we call the {\it augmentation} of $(W,\Sigma,w')$.
It satisfies the condition (SC),
and determines an orbifold pair $(N_0,Q)$.
We call the $3$-orbifold $\OO:=N_0$
the orbifold obtained from $(M_0,P)$ by the {\it orbifold surgery}
determined by the replacement of the weight function $w$ with 
the new weight function $w'$.
}
\end{definition}

The following simple lemma is used repeatedly in the proof of the main theorem.

\begin{lemma}
\label{lem:relabeled-orbifold-generic}
Let $\orbp$ be a pared orbifold,
and let $\OO$ be the orbifold obtained from $\orbp$
by an orbifold surgery.
Then $\OO$
does not contain a bad $2$-suborbifold
and $\partial\OO$ does not contain a spherical component.
In particular, $\OO$ is very good.
\end{lemma}

\begin{proof}
Let $(W,\Sigma,w)$ be a weighted graph representing the pared orbifold $\orbp$,
and let $w'$ be the weight function on $\Sigma$ 
that gives the orbifold $\OO=N_0$,
where $(N_0,Q)$ is the orbifold pair
that is represented by 
the augmentation of $(W,\Sigma,w')$.
Assume to the contrary that
$N_0$ contains a bad $2$-suborbifold,
$S$, which is either a teardrop $S^2(n)$ or a spindle $S^2(m,n)$
for some integers $m>n\ge 2$.
Since the underlying space $|S|$ is disjoint from
the vertex set of 
the singular set, $\Sigma(N_0)$, of $N_0$,
we may assume $|S|$ is a submanifold of $W$
transversal to $\Sigma$.
Then it determines a suborbifold, $S^*$, of $M_0$,
such that $|S^*|=|S|\cap |M_0|$.
The singular set of $S^*$ is 
equal to $|S^*|\cap\Sigma_s$, 
where $\Sigma_s$ is the subgraph of $\Sigma$
consisting of the edges of integral $w$-weight,
and the index of each singular point 
is given by the $w$-weight of the corresponding edge of $\Sigma_s$.

First, suppose that $S\cong S^2(n)$ for $n\ge 2$.
Let $e$ be the edge of $\Sigma$ such that $|S|\cap e$
is the singular point of $S$.
If $e$ is an edge of $\Sigma_s$, 
then $S^*$ is 
isomorphic to the teardrop $S^2(w(e))$,
which contradicts the fact that $M_0$ is good.
If $e$ is an edge of $\Sigma_{\infty}$, 
then $S^*$ is a disc whose boundary is
an essential simple loop on $P$.
This contradicts the fact that $P$ is incompressible in $M_0$.

Next, suppose that 
$S\cong S^2(m,n)$ for $m>n\ge 2$.
Let $e_1$ and $e_2$ be the edges of $\Sigma$
corresponding to the singular point of $S$
of index $m$ and $n$, respectively.
Then $w'(e_1)=m\ne n =w'(e_2)$, and so $e_1$ and $e_2$ are distinct.
If both $e_1$ and $e_2$ are contained in $\Sigma_s$,
then $S^*\cong S^2(m^*,n^*)$ for some $m^*,n^*\ge 2$.
Since $M_0$ does not contain a bad $2$-suborbifold,
$m^*$ and $n^*$ must be equal,
and hence $S^*$ is an spherical suborbifold of $M_0$.
Since $M_0$ is irreducible, 
$S^*$ bounds a discal $3$-orbifold.
This implies $e_1$ and $e_2$ determine the same edge of $\Sigma(N_0)$.
By the condition (SC),
this in turn implies $e_1=e_2$, a contradiction.
If exactly one of $e_1$ and $e_2$ is contained in $\Sigma_s$,
then $S^*$ is a discal orbifold 
whose boundary is
an essential simple loop on $P$.
This contradicts the assumption that
$P$ is incompressible in $M_0$.
If none of $e_1$ and $e_2$ is contained in $\Sigma_s$,
then $S^*$ is an annulus whose boundary
consists of a pair of essential simple loops on $P$.
Thus $S^*$ is parallel to $P$ by Definition \ref{def:pared-orbifold}(4),
and so $e_1=e_2$, a contradiction.

Thus we have proved that $\OO=N_0$ does not contain a bad $2$-suborbifold.
The assertion that $\partial\OO$ does not contain a spherical orbifold
follows from the fact that $\OO=N_0$ is represented by the augmentation of $(W,\Sigma,w')$.
The assertion that $\OO$ is very good follows from \cite[Corollary 1.3]{BLP},
which is a consequence of the orbifold theorem.
\end{proof}

Another key tool for the proof of the main theorem is the homology with $\ZZ_2$ coefficient.
Under Notation \ref{notation}, we have
the following lemma, which can be easily deduced from the definition of 
$H_1(\OO;\ZZ_2)$ and the Alexander duality.

\begin{lemma}
\label{lem:homology}
Suppose an orbifold $\OO$ is represented by a 
weighted graph $(S^3,\Sigma,w)$ in $S^3$.
Let $\Sigma_{\mathrm{even}}$ 
be the subgraph of $\Sigma$ spanned by the edges of even weight.
Then $H_1(\OO;\ZZ_2)$ is determined by $H_1(\Sigma_{\mathrm{even}};\ZZ_2)$.
To be precise, we have the following natural isomorphisms.
\[
H_1(\OO;\ZZ_2)\cong 
H_1(S^3-\Sigma_{\mathrm{even}};\ZZ_2)\cong
H^1(\Sigma_{\mathrm{even}};\ZZ_2)\cong
\mathrm{Hom}(H_1(\Sigma_{\mathrm{even}};\ZZ_2),\ZZ_2)
\]
In particular, the following hold.
\begin{enumerate}
\item
$H_1(\OO;\ZZ_2)$ is generated by the meridians of 
edges of $\Sigma_{\mathrm{even}}$.
\item
The meridian of an edge of $\Sigma$ of odd degree 
represents the trivial element of $H_1(\OO;\ZZ_2)$.
\item
Let $e_i$ ($i=1,2,3$) be edges of $\Sigma$ incident on a vertex of $\Sigma$,
and suppose that $w(e_1)$ is odd and $w(e_2)$ and $w(e_3)$ are even.
Then the meridians of $e_2$ and $e_3$ represent the same element of $H_1(\OO;\ZZ_2)$.
\end{enumerate}
\end{lemma}

\section{Canonical horoball pairs for Kleinian groups generated by two parabolic transformations}
\label{Sec:canonical-horoball-pair}
{\it Throughout Sections \ref{Sec:canonical-horoball-pair} $\sim$ 
\ref{sec:exceptional-flexible},
$\Gamma =\langle \alpha, \beta\rangle$ denotes a non-elementary Kleinian group 
generated by two parabolic transformations $\alpha$ and $\beta$,
and $M=\HH^3/\Gamma$ denotes the quotient hyperbolic $3$-orbifold.}
Let $\eta$ be the geodesic 
joining the parabolic fixed points of $\alpha$ and $\beta$,
and let $h$ be the $\pi$-rotation around $\eta$.
Then we have
\[
(h\alpha h^{-1},h\beta h^{-1})=(\alpha^{-1},\beta^{-1}).
\]
We call $h$ the {\it inverting elliptic element} for 
the parabolic generating pair $\{\alpha,\beta\}$ of 
the Kleinian group $\Gamma$.
As shown in \cite[Section~5.4]{Thurston},
we can find a geodesic intersecting $\eta$ orthogonally,
such that the $\pi$-rotation, $f$, around it satisfies the following identity.
\[
(f\alpha f^{-1},f\beta f^{-1})=(\beta,\alpha).
\]
We call $f$ the {\it exchanging elliptic element} for 
the parabolic generating pair 
$\{\alpha,\beta\}$ 
of the Kleinian group $\Gamma$.
It should be noted that $fh$ is
the exchanging elliptic element for 
the parabolic generating pair $\{\alpha,\beta^{-1}\}$ of $\Gamma$. 

By abuse of notation,
we denote the isometries of $M$ induced by $f$ and $h$
by the same symbols $f$ and $h$, respectively.
Each of them is either the identity map or a (nontrivial)
involution of $M$,
i.e., its order is $1$ or $2$.
We call the isometries $f$ and $h$,
the {\it exchanging involution} and the {\it inverting involution} of $M$
associated with the parabolic generating pair $\{\alpha,\beta\}$.
It should be noted that if $\Gamma$ is isomorphic to 
a hyperbolic $2$-bridge link group $G(K(r))$
and $\{\alpha,\beta\}$ is the upper-meridian pair,
then the involutions $f$ and $h$ on $M\cong S^3-K(r)$
are the restrictions of the vertical and horizontal involutions
of $K(r)$ (see Figure \ref{fig.trivial-tangle}).
This is the reason why we use the symbols $f$ and $h$ 
with two different meanings.

Let $\hat\Gamma:=\langle \Gamma, f\rangle$
be the group generated by $\Gamma$ and the exchanging elliptic element $f$
associated with the parabolic generating pair $\{\alpha,\beta\}$ of $\Gamma$.
Then $\hat \Gamma$ is a Kleinian group 
which is either equal to $\Gamma$ or a $\ZZ_2$-extension of $\Gamma$
according to whether $f$ belongs to $\Gamma$ or not.
Let $\hat M:=\HH^3/\hat \Gamma$ be the quotient hyperbolic orbifold,
and let $\hat C_{\alpha,\beta}$ be the maximal cusp of $\hat M$
corresponding to the conjugacy class of $\hat\Gamma$
containing both $\alpha$ and $\beta=f\alpha f^{-1}$.
Then the inverse image 
$p^{-1}(\hat C_{\alpha,\beta})$ 
of $\hat C_{\alpha,\beta}$ 
by the projection $p:\HH^3\to \hat M$
is a union of horoballs with disjoint interiors 
but whose boundaries have nonempty tangential intersections.
We call it the {\it canonical horoball system} 
associated with the parabolic generating pair $\{\alpha,\beta\}$ of $\Gamma$.
If a parabolic element $\gamma$ of $\Gamma$ 
stabilises a member of the canonical horoball system,
we denote the horoball by $H_{\gamma}$.
We denote the translation length of $\gamma$ on the horosphere $\partial H_{\gamma}$
by the symbol $|\gamma|=|\gamma|_{\partial H_{\gamma}}$,
and call it the {\it length of $\gamma$ in the canonical horosphere}.
We call the pair $(H_{\alpha},H_{\beta})$ the {\it canonical horoball pair}
for the parabolic generating pair $\{\alpha,\beta\}$ of 
the Kleinian group $\Gamma$.

Note that the definition of $|\gamma|$ depends on 
the parabolic generating pair $\{\alpha,\beta\}$,
because the exchanging elliptic element $f$ is involved in the definition.
However, it actually depends only on the pair $\{\fix(\alpha),\fix(\beta)\}$,
because any orientation-preserving isometry,
which exchanges
$\fix(\alpha)$ and $\fix(\beta)$,
also exchanges the members $H_{\alpha}$ and $H_{\beta}$
of the canonical horoball pair associated with $\{\alpha,\beta\}$.
(Otherwise, the product of $f$ and an unexpected involution,
which exchanges 
$\fix(\alpha)$ and $\fix(\beta)$
but does not exchange $H_{\alpha}$ and $H_{\beta}$,
gives a loxodromic transformation which fixes the parabolic 
fixed points $\fix(\alpha)$ and $\fix(\beta)$.
This contradicts the assumption that $\Gamma$ is discrete.)

The following lemmas are proved by Adams \cite[Lemma 3.1, Theorem 3.2, and p.197]{Adams1}
(see also Brenner \cite{Brenner}).
Since they holds a key to the proof of the main theorem
and since we described the setting in a slightly different way,
we include the proof.

\begin{lemma}
\label{Lem:Brenner}
Under the above setting, the following hold.
\begin{enumerate}
\item
For any parabolic element $\gamma\in\Gamma$ 
which stabilises a member of the canonical horoball system,
we have $|\gamma|\ge 1$.
\item
$1\le |\alpha|=|\beta|$.
\item
If $\Gamma$ is non-free then $|\alpha|=|\beta|<2$.
\end{enumerate}
\end{lemma}

\begin{proof}
(1) We may assume $\partial H_{\gamma}$ is the horosphere $\CC\times\{1\}$
in the upper half space model $\HH^3=\CC\times \RR_+$. 
Then some other member, $H_g$, of the canonical horoball system touches $\partial H_{\gamma}$
and hence has Euclidean diameter $1$.
Since $\gamma(H_g)=H_{\gamma g \gamma^{-1}}$ is also a member of the canonical horoball system,
$H_g$ and $\gamma(H_g)$ have disjoint interiors.
Hence we have $|\gamma|\ge 1$.

(2) Since $\alpha$ and $\beta$ are conjugate in $\hat\Gamma$, $|\alpha|$ and $|\beta|$
are equal.
Moreover, $|\alpha|=|\beta|$ is $\ge 1$ by (1).

(3) We refer the proof to \cite[Theorem 3.2]{Adams1} and Brenner \cite{Brenner}.
\end{proof}

\begin{lemma}
\label{lem:primitive}
Both $\alpha$ and $\beta$ are primitive in $\Gamma$.
\end{lemma}

\begin{proof}
If $\Gamma$ is a free, then the assertion follows from the fact that
any member of a free-generating system of a free group is primitive.
So, we may assume $\Gamma$ is non-free.
Suppose on the contrary that one of the two elements, say $\alpha$, is imprimitive,
namely there is an element $\alpha_0\in\Gamma$
and an integer $n\ge 2$ such that $\alpha=\alpha_0^n$.
Then $|\alpha|=n|\alpha_0|\ge n\ge 2$ by Lemma \ref{Lem:Brenner}(1).
But, this contradicts Lemma \ref{Lem:Brenner}(3).
\end{proof}

\section{Outline of the proof of Theorem \ref{main-theorem}}
\label{sec:outline}
We now state an outline of the proof of Theorem \ref{main-theorem}.
Since the if part is clear (cf. Proposition \ref{prop:two-generator-Heckoid}),
we prove the only if part. 
To this end, we summarise
the setting of Theorem \ref{main-theorem}.

\begin{assumption}
\label{asmpt:setting-main-theorem}
{\rm
Let $\Gamma =\langle \alpha, \beta\rangle$  
be a non-free Kleinian group generated by 
two non-commuting parabolic transformations $\alpha$ and $\beta$, and 
let $M=\HH^3/\Gamma$ be the quotient hyperbolic orbifold.
Let $M_0$ be the non-cuspidal part of $M$,
and $P=\partial M_0$ the parabolic locus.
By Theorem \ref{thm:tameness},
$(M_0,P)$ admits a relative compactification $(\bar M_0,\bar P)$,
which is a pared orbifold by Theorem \ref{thm:geom-infinite}.
The pared orbifold $(\bar M_0,\bar P)$ can be represented by
a weighted graph $(W,\Sigma,w)$,
where $W$ is a compact $3$-manifold, $\Sigma$ is a trivalent graph
properly embedded in $W$, and $w$ is a weight function on the edge set of $\Sigma$ (see Convention \ref{conv:pared-orbifold}).
{\it We abuse notation to denote the (compact) pared orbifold 
$(\bar M_0,\bar P)$ by $\orbp$.}
We denote the components of $\bar P$,
which is now denoted by $P$,
corresponding to the cusps 
$C_{\alpha}$ and $C_{\beta}$
by $P_{\alpha}$ and $P_{\beta}$, respectively.
}
\end{assumption}

\begin{proof}[Outline of the proof of Theorem \ref{main-theorem}]
Under Assumption \ref{asmpt:setting-main-theorem},
the proof is divided into the following two cases.

\begin{enumerate}
\item[Case 1.]
$P_{\alpha}\cong P_{\beta}$ is a flexible cusp
(Section \ref{sec:flexible-cusp-case} for generic case 
and Section \ref{sec:exceptional-flexible} for exceptional case).
\item[Case 2.]
$P_{\alpha}\cong P_{\beta}$ is a rigid cusp
(Section \ref{sec:rigid-cusp-case}).
\end{enumerate}

In both cases, the first task is to find an orbifold surgery
that yields an orbifold $\OO$ with dihedral orbifold fundamental group.

In Case 1, 
this can be generically done by using Lemma \ref{lem:primitive}.
In fact, if $P_{\alpha}\cong P_{\beta}$ is a flexible cusp,
then Lemma \ref{lem:primitive} implies that 
each of the parabolic elements $\alpha$ and $\beta$
can be represented by simple loops of $P_{\alpha}$ and $P_{\beta}$, respectively.
Generically, these simple loops are disjoint, 
and such an surgery obviously exists.
This generic case is treated in Section \ref{sec:flexible-cusp-case}.

However, there is an exceptional case where $P_{\alpha}=P_{\beta}\cong S^2(2,2,2,2)$
and the simple loops representing $\alpha$ and $\beta$ 
intersect nontrivially (Lemma \ref{lem:simple}).
In this case, the exchanging elliptic element $f$ does not belong to $\Gamma$,
and we need to consider the $\ZZ_2$-extension 
$\hat\Gamma:=\langle \Gamma, f\rangle$
of $\Gamma$ and consider the corresponding pared orbifold
$(\hat M_0,\hat P):=(M_0,P)/f$,
where $\hat P_{\alpha\beta}$ is isomorphic to the rigid cusp $S^2(2,4,4)$.
The treatment of this case is deferred to Section \ref{sec:exceptional-flexible},
after the treatment of the rigid cusp Case 2 in Section \ref{sec:rigid-cusp-case},
described below.

In Case 2, if $P_{\alpha}\cong P_{\beta}$ is isomorphic to
either $S^2(2,4,4)$ or $S^2(2,3,6)$,
the dihedral surgery can be found by using
an estimate of the shortest, second shortest, and third shortest 
lengths of parabolic elements on the maximal rigid cusp,
which in turn is based on Lemma \ref{Lem:Brenner}. 
If $P_{\alpha}\cong P_{\beta}$ is isomorphic to $S^2(3,3,3)$,
the inverting parabolic element $h$ does not belong to $\Gamma$,
and we consider the $\ZZ_2$-extension 
$\Gamma_h:=\langle\Gamma, h\rangle$
and the corresponding pared orbifold 
$(M_{h,0},P_{h}):=(M_0,P)/h$.
The images of $P_{\alpha}$ and $P_{\beta}$ in this quotient
is isomorphic to $S^2(2,3,6)$,
and this case can be treated by using arguments in the case where
$P_{\alpha}\cong P_{\beta}\cong S^2(2,3,6)$.

After finding an orbifold surgery
that yields an orbifold $\OO$ with dihedral orbifold fundamental group,
we can appeal to the classification Theorem \ref{thm:dihedral-orbifold}
of the dihedral orbifolds,
because Lemma \ref{lem:relabeled-orbifold-generic}
guarantees that the orbifold $\OO$ satisfies the
three conditions in Theorem \ref{thm:dihedral-orbifold}.
So, $\OO$ belongs to the list in the theorem.
The original pared orbifold $(M_0,P)$ is obtained from the dihedral orbifold $\OO$
by inverse surgery operations.
Through case-by-case arguments,
by using the homology with $\ZZ_2$-coefficients,
a result concerning the symmetries of the spherical dihedral orbifold
(Corollary \ref{cor:dihedral-orbifold-isometry2}),
and a \lq surgery trick' (the last paragraph in Case 1 in Section \ref{sec:rigid-cusp-case}
and Case 1 in Section \ref{sec:exceptional-flexible}),
we prove the following.
\begin{enumerate}
\item
If $P_{\alpha}\cong P_{\beta}$ is a flexible cusp,
then, in the generic case,
the pared orbifold $(M_0,P)$ is isomorphic to 
either a hyperbolic $2$-bridge link exterior or a Heckoid orbifold
(Section \ref{sec:flexible-cusp-case}):
in the exceptional case, we encounter a contradiction
(Section \ref{sec:exceptional-flexible}).
\item
If $P_{\alpha}\cong P_{\beta}$ is a rigid cusp,
then we encounter a contradiction (Section \ref{sec:rigid-cusp-case}).
\end{enumerate}
This ends an outline of the proof of the main Theorem \ref{main-theorem}.
\end{proof}

\section{Proof of Theorem \ref{main-theorem} - flexible cusp: generic case -}
\label{sec:flexible-cusp-case}

Under Assumption \ref{asmpt:setting-main-theorem},
suppose that $P_{\alpha}\cong P_{\beta}$
is a flexible cusp.
Then the $2$-orbifold $P_{\alpha}\cong P_{\beta}$ is isomorphic to the torus $T^2$, 
the pillowcase $S^2(2,2,2,2)$,
the annulus $A^2$, or $D^2(2,2)$.
The following fact is the starting point of this section.

\begin{lemma}
\label{lem:simple}
Under the above setting, 
$\alpha$ and $\beta$ are represented by simple loops
on $P_{\alpha}$ and $P_{\beta}$, respectively.
Moreover, if $P_{\alpha}=P_{\beta}$,
then one of the following holds.
\begin{enumerate}
\item
The parabolic elements $\alpha$ and $\beta$ are represented by
the same (possibly oppositely oriented) simple loop.
\item
$P_{\alpha}=P_{\beta}\cong S^2(2,2,2,2)$,
$f\notin\Gamma$,
and $P_{\alpha}/ f=P_{\beta}/f \cong S^2(2,4,4)$,
where the first $f$ is the exchanging elliptic element associated with $\{\alpha,\beta\}$
and the last two $f$'s
denote the involution on $(M_0,P)$ induced by 
the exchanging elliptic element $f$
(see Figure \ref{fig.involution-f}).
\end{enumerate}
\end{lemma}

\begin{proof}
The first assertion directly follows from Lemma \ref{lem:primitive},
because any primitive parabolic element 
in the orbifold fundamental group of
the $2$-dimensional orbifold 
$T^2$, $S^2(2,2,2,2)$, $A^2$, or $D^2(2,2)$
is represented by a simple loop on the $2$-orbifold.
For the proof of the second assertion, 
suppose that $P_{\alpha}=P_{\beta}$.
If the exchanging elliptic element $f$ belongs to $\Gamma$, 
then $\beta$ is conjugate to $\alpha$ in $\Gamma$,
and so they are represented by the same simple loop.
Thus we may suppose $f\notin \Gamma$.
Then $f$ descends to a nontrivial orientation-preserving involution on $M$,
which we continue to denote by $f$, on 
the flexible cusp $P_{\alpha}$.
By the classification of orientation-preserving involutions on flexible cusps,
we can observe that either
(a) the involution $f$ on $M_0$ preserves or reverses the homotopy class  
of each essential simple loop on $P_{\alpha}$, or 
(b) 
$P_{\alpha}\cong S^2(2,2,2,2)$ and 
$P_{\alpha}/ f\cong S^2(2,4,4)$.
In the first case, $\alpha$ and $\beta^{\pm 1}$ are  represented
by the same simple loop, and so
we obtain the desired conclusion.
\end{proof}

In this section, we treat the case 
where either $P_{\alpha}\ne P_{\beta}$ or 
$P_{\alpha}=P_{\beta}$ and the conclusion (1) in Lemma \ref{lem:simple} holds.
Thus we assume the following condition in the remainder of this section.
The other case is treated in Section \ref{sec:exceptional-flexible}.

\begin{assumption}
\label{assumption:flexible1}
{\rm
Under Assumption \ref{asmpt:setting-main-theorem},
we further assume that (a) $P_{\alpha}\cong P_{\beta}$
is a flexible cusp and that
(b) either $P_{\alpha}\ne P_{\beta}$ or 
$P_{\alpha}=P_{\beta}$ and the conclusion (1) in Lemma \ref{lem:simple} holds.
It should be noted that 
either $\alpha$ and $\beta$ are represented by disjoint simple loops 
or they are represented by
the same (possibly oppositely oriented) simple loop.
}
\end{assumption}

Under this assumption, we can apply an orbifold surgery on $\orbp$
to the pared orbifold $\orbp$ to obtain a dihedral orbifold, $\OO$, as follows.
Note that Assumption \ref{assumption:flexible1} implies that
the pared orbifold $\orbp$ is represented by
a weighted graph $(W,\Sigma,\tilde w)$,
such that
there are (possibly identical) edges 
$e_{\alpha}$ and $e_{\beta}$ of $\Sigma$
whose meridians represent $\alpha$ and $\beta$, respectively.
Let $w$ be a weight function on $\Sigma$ 
which is identical with $\tilde w$, except that
$w(e_{\alpha})=w(e_{\beta})=2$.
Then the orbifold $\OO$  
represented by the augmentation of the weighted graph $(W,\Sigma,w)$
is a result of an \lq\lq order $2$'' orbifold surgery on $\orbp$,
and $\pi_1(\OO)$ is dihedral, as shown below.

Note that there is a natural epimorphism from $\Gamma=\pi_1(M_0)$
to $\pi_1(\OO)$, and the images of $\alpha$ and $\beta$ in $\pi_1(\OO)$
have order $\le 2$.
Moreover, the images of $\alpha$ and $\beta$ have the same order,
because (a) if $f\in\Gamma$ then $\alpha$ and $\beta$ are conjugate in 
$\Gamma$ and so in $\pi_1(\OO)$, and 
(b) if $f\notin\Gamma$ then $f$ descends to an involution on $\OO$
which interchanges the images of $\alpha$ and $\beta$.
So $\pi_1(\OO)$ is either the trivial group or a dihedral group.
Since $\OO$ is very good by Lemma \ref{lem:relabeled-orbifold-generic}
and since $\OO$ has nonempty singular set,
$\pi_1(\OO)$ is nontrivial and so isomorphic to a dihedral group.  

Thus $\OO$ satisfies the three conditions in Theorem \ref{thm:dihedral-orbifold}
and so $\OO$ belongs to the list in the theorem.
We have the following lemma.

\begin{lemma}
\label{lem:relabeled-orbifold2}
The orbifold $\OO$ is isomorphic to the 
spherical dihedral orbifold $\OO(r;d_+,d_-)$
for some $r\in\QQ$ and coprime positive integers $d_+$ and $d_-$.
\end{lemma}

\begin{proof}
We show that the possibilities (2), (3) and (4) in Theorem \ref{thm:dihedral-orbifold} 
cannot happen. 
Suppose (2) happens. Then we can see that
one of the following holds,
by recalling the fact that
$\OO$ is obtained from the pared orbifold $\orbp$
an order $2$ orbifold surgery.
\begin{enumerate}[(i)]
\item
$M_0$ is the exterior of the two-component trivial link, $P=\partial M_0$,
and the singular set of $M_0$ is empty.
\item
The underlying space of $M_0$ is the solid torus (the exterior of a trivial knot), 
$P=\partial M_0$,
and the singular set is a trivial knot in the solid torus with index $2$.
\end{enumerate}
In each case, $\orbp$ is reducible, a contradiction.

By the same reasoning, we can see that (3) cannot happen.

If (4) happens,
then as in the above, we can see that one of the following holds,
where $(B^3, t_1\cup t_2)$ is a two-strand trivial tangle.
\begin{enumerate}[(i)]
\item
$\orbp\cong(\cl(B^3-N(t_1\cup t_2)),\fr N(t_1\cup t_2))$
and the singular set of $M_0$ is empty.
\item
$\orbp\cong(\cl(B^3-N(t_1)),\fr N(t_1))$ and the singular set of $M_0$ is $t_2$
 with index $2$.
\end{enumerate}
In the first case, $\Gamma=\pi_1(M_0)$ is a rank $2$ free group,
which contradicts the assumption that $\Gamma$ is non-free.
In the second case, note that $H_1(M_0)$,
the abelianization 
of the orbifold fundamental group $\pi_1(M_0)$, 
is $\ZZ\oplus\ZZ_2$.
On the other hand,
both $\alpha$ and $\beta$
are represented by the core loop of the annulus $P=\fr N(t_1)$,
and the pair $\{\alpha,\beta\}$ cannot generate $H_1(M_0)$, a contradiction.
\end{proof}

By Lemma \ref{lem:relabeled-orbifold2},
the original orbifold $\orbp$ is recovered from $\OO=\OO(r;d_+,d_-)$
by applying the inverse orbifold surgery operation.
This leads us to the following proposition.

\begin{proposition}
\label{prop:flexible-cusp-case}
Under the notation in Lemma \ref{lem:relabeled-orbifold2},
the following hold, if necessary by replacing $r=q/p$ with $q'/p$
where $q'=q+p$ or $qq'\equiv 1 \pmod{p}$.

\begin{enumerate}
\item
If $|K(r)|=1$,
then one of the following holds.
\begin{enumerate}[\rm(i)]
\item
$d_+=d_-=1$ and $\orbp\cong (E(K(r)),\partial E(K(r)))$,
where $q\not\equiv \pm 1\pmod p$. 
\item
$d_+=1$,  $d_-\ge 2$, and $\orbp\cong \orbm_0(r;d_-)$.
\item
$d_+=1$, $d_-$ is an odd integer $\ge 3$, and $\orbp\cong \orbm_1(r;d_-)$.
\item
$d_+=2$, $d_-$ is an odd integer $\ge 3$, and $\orbp\cong\orbm_2(r;d_-)$.
\end{enumerate}
\item
If $|K(r)|=2$,
then one of the following holds.
\begin{enumerate}[\rm(i)]
\item
$d_+=d_-=1$ and $\orbp\cong (E(K(r)),\partial E(K(r)))$,
where $q\not\equiv \pm 1\pmod p$. 
\item
$d_+=1$, $d_-\ge 2$, and $\orbp\cong\orbm_0(r;d_-)$.
\item
$d_+=2$, $d_-$ is an odd integer $\ge 3$, and $\orbp\cong\orbm_2(r;d_-)$.
\end{enumerate}
\end{enumerate}
\end{proposition}

\begin{figure}
\includegraphics[width=0.7\hsize, bb=0 0 1354 1387]{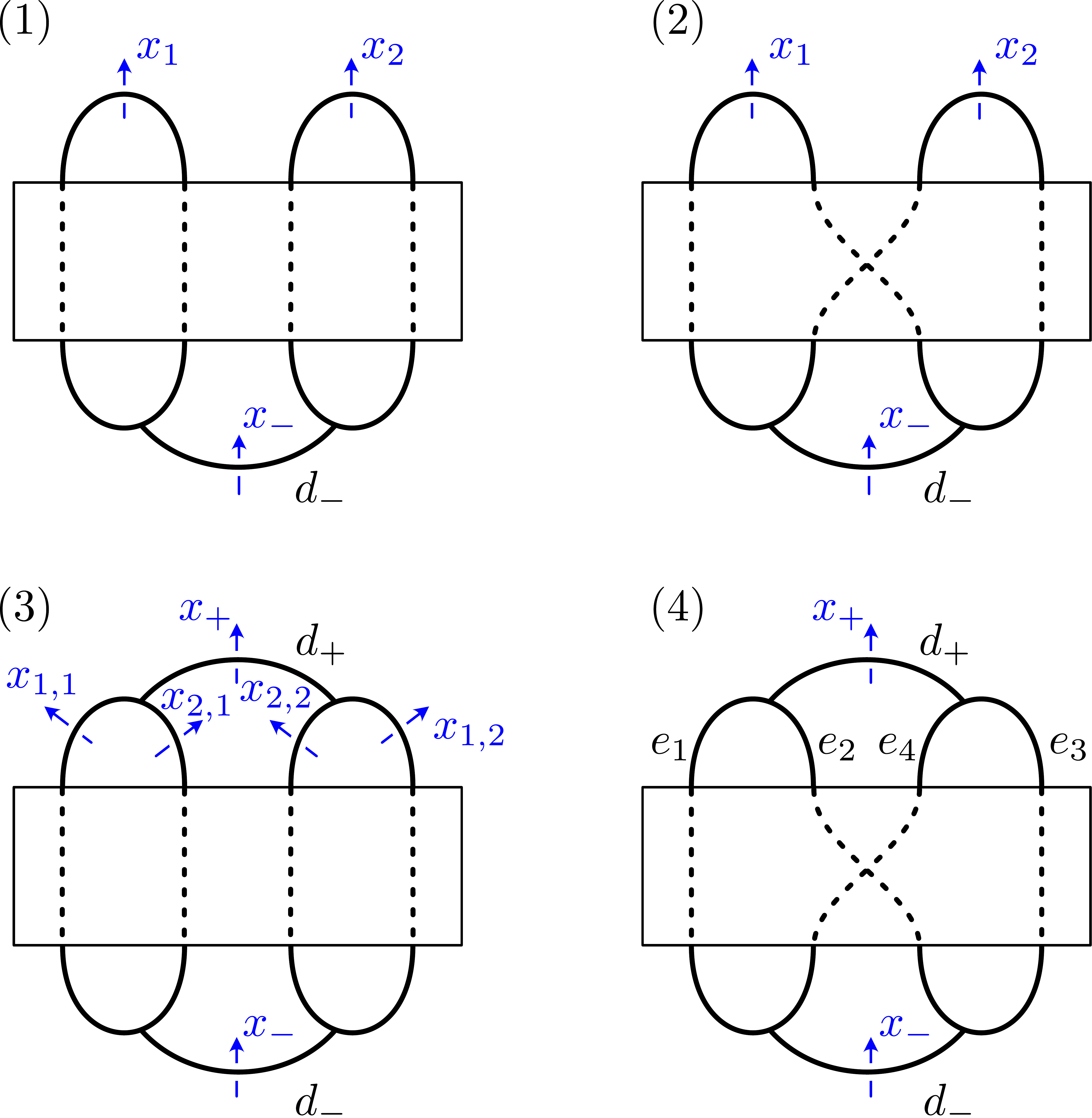}
\caption{$H_1(\OO;\ZZ_2)\cong H_1(S^3-\Sigma_{\mathrm{even}};\ZZ_2)$,
where $\Sigma_{\mathrm{even}}$ is the subgraph of
$\Sigma=K(r)\cup\tau_+\cup\tau_-$ spanned by the edges of even weight. 
}
\label{fig.homology1}
\end{figure}

\begin{proof}
Recall that
$\OO=\OO(r;d_+,d_-)$ is represented by the weighted graph
$(S^3,K(r)\cup\tau_+\cup\tau_-,w)$
for some $r\in\QQ$ and for some coprime positive integers $d_+$ and $d_-$,
and $w$ is given by the following rule (see Figure \ref{fig.Dihedral-orbifold}):
\[
w(K(r))=2, \quad w(\tau_+)=d_+, \quad w(\tau_-)=d_-
\]
Then $\orbp$ is represented by the weighted graph
$(S^3,K(r)\cup\tau_+\cup\tau_-,\tilde w)$,
where $\tilde w$ is obtained from $w$ 
by replacing the label $2$ of the edges $e_{\alpha}$ and $e_{\beta}$,
which correspond to $P_{\alpha}$ and $P_{\beta}$ respectively,
with the label $\infty$.
By Remark \ref{remark:upside-down},
we may assume $1\le d_+\le d_-$,
if necessary by replacing $r=q/p$ with $q'/p$
where $qq'\equiv 1 \pmod{p}$.

Case 1. $d_+=d_-=1$.
Then $\Sigma(\OO)$ is the $2$-bridge link $K(r)$.
Thus $\orbp\cong (S^3,K(r),\tilde w)$,
where either (a) $\tilde w(K(r))=\infty$ or
(b) $K(r)$ is a $2$-component link $K_1\cup K_2$
and $(\tilde w(K_1),\tilde w(K_2))=(\infty,2)$.
In the first case, 
$\orbp\cong (E(K(r)),\partial E(K(r)))$,
and so $\HH^3/\Gamma$ is the hyperbolic $2$-bridge link complement,
$S^3-K(r)$: in particular, $q\not\equiv \pm 1\pmod p$.
In the second case, both $\alpha$ and $\beta$ are meridians 
of the component $K_1$,
which contradicts the fact that $H_1(\OO;\ZZ_2)\cong (\ZZ_2)^2$ .

Case 2. $d_+=1 < d_-$.

Subcase 2.1. $|K(r)|=2$ (see Figure \ref{fig.homology1}(1)).
Then the edge set of $\Sigma(\OO)$ consists of $\tau_-$ and the two components $K_1$, $K_2$
of $K(r)$.
Let $x_-$, $x_1$ and $x_2$ be the meridians of  
$\tau_-$, $K_1$ and $K_2$, respectively.
By Lemma \ref{lem:homology}, 
$H_1(\OO;\ZZ_2)\cong (\ZZ_2)^2$ is freely
generated by $\{x_1,x_2\}$, and moreover we have $x_-=0$. 
Since $H_1(\OO;\ZZ_2)$ is generated by (the images of) $\alpha$ and $\beta$,
we may assume $e_{\alpha}=K_1$ and $e_{\beta}=K_2$.
Thus $\orbp$ is represented by $(S^3,K(r)\cup\tau_-,\tilde w)$,
where $\tilde w(K_1)=\tilde w(K_2)=\infty$ and $\tilde w(\tau_-)=d_-$.
Hence $\orbp\cong \orbm_0(r;d_-)$.

Subcase 2.2. $|K(r)|=1$ (see Figure \ref{fig.homology1}(2)).
Then the edge set of $\Sigma(\OO)$ consists of $\tau_-$ and 
the two subarcs $J_1$ and $J_2$ of $K(r)$ bounded by $K(r)\cap \tau_-$.
Let $x_-$, $x_1$ and $x_2$ be the meridians of  
$\tau_-$, $J_1$ and $J_2$, respectively.

Suppose first that $d_-$ is odd.
Then we see by Lemma \ref{lem:homology} that 
$x_-=0$ in
$H_1(\OO;\ZZ_2)\cong \ZZ_2$ and that
$H_1(\OO;\ZZ_2)\cong \ZZ_2$ is generated by $x_1=x_2$. 
Hence one of the following holds.
\begin{enumerate}
\item
$\{e_{\alpha},e_{\beta}\}=\{J_1,J_2\}$ and so
$\orbp$ is represented by $(S^3,K(r)\cup\tau_-,\tilde w)$,
where $\tilde w(J_1)=\tilde w(J_2)=\infty$ and $\tilde w(\tau_-)=d_-$.
Hence $\orbp\cong \orbm_0(r;d_-)$.
\item
$e_{\alpha}=e_{\beta}=J_i$ for $i=1$ or $2$.
By the symmetry of $\OO$, we may assume $i=1$ and so
$\orbp$ is represented by $(S^3,K(r)\cup\tau_-,\tilde w)$,
where $\tilde w(J_1)=\infty$, $\tilde w(J_2)=2$ and $\tilde w(\tau_-)=d_-$.
Hence $\orbp\cong \orbm_1(r;d_-)$.
\end{enumerate}

Suppose next that $d_-$ is even.
Then $x_1+x_2+x_-=0$ in $H_1(\OO;\ZZ_2)\cong (\ZZ_2)^2$.
Since $H_1(\OO;\ZZ_2)$ is generated by $\alpha$ and $\beta$,
we have $e_{\alpha}\ne e_{\beta}$.
This implies that the exchanging elliptic element $f$ for $\{\alpha,\beta\}$
does not belong to $\Gamma$, and 
$f$ descends to an involution on $\OO$
interchanging  
$e_{\alpha}$ with $e_{\beta}$.
We now use Corollary \ref{cor:dihedral-orbifold-isometry2}
on the symmetry of the orbifold $\OO(r;d_+,d_-)$.
We first consider 
the generic case where $p\ne 1$ (i.e., $K(r)$ is a nontrivial knot) or $d_-> 2$.
(Recall the current assumption $d_+=1$.)
Then, by Corollary \ref{cor:dihedral-orbifold-isometry2}(1),
any orientation-preserving involution of $\OO$ preserves $\tau_-$.
So, $e_{\alpha}$ and $e_{\beta}$ are different from $\tau_{\pm}$, and therefore
$\{e_{\alpha},e_{\beta}\}=\{J_1,J_2\}$.
Hence, as in the previous case, we can conclude $\orbp\cong \orbm_0(r;d_-)$.
In the exceptional case where $p=1$ and $d_-= 2$,
The orbifold $\OO\cong\OO(0/1;1,2)$ has the $3$-fold cyclic symmetry as illustrated in  
Figure \ref{symmetry-dihedral-orbifold-except}.
Thus, if necessary after applying this symmetry,
we may assume $\{e_{\alpha},e_{\beta}\}=\{J_1,J_2\}$.
Hence we have $\orbp\cong \orbm_0(r;d_-)\cong \orbm_0(0/1;2)$.

Since we repeatedly use the above argument
in the remainder of the proof of Proposition \ref{prop:flexible-cusp-case},
we state an expanded version of the argument as a lemma.

\begin{lemma}
\label{lem:repeat}
Under the setting of Proposition \ref{prop:flexible-cusp-case},
suppose $(d_+,d_-)\ne(1,1)$ and $H_1(\OO;\ZZ_2)\cong (\ZZ_2)^2$.
Then $e_{\alpha}\ne e_{\beta}$, and the exchanging elliptic element
$f$ does not belong to $\Gamma$ and it
descends to an orientation-preserving involution of $\OO=\OO(r;d_+,d_-)$
interchanging $e_{\alpha}$ and $e_{\beta}$.
Moreover, the following hold.
\begin{enumerate}
\item
Except when $p=1$ and $\{d_+,d_-\}=\{1,2\}$, 
$e_{\alpha}$ and $e_{\beta}$ are different from $\tau_{\pm}$.
\item
If $d_+, d_-\ge 2$, then the inverting elliptic element $h$ belongs to $\Gamma$.
\end{enumerate}
\end{lemma}

\begin{proof}
We have only to prove (2).
If $h$ does not belong to $\Gamma$, then it descends to an  
orientation-preserving involution of $\OO(r;d_+,d_-)$
which preserves both $e_{\alpha}$ and $e_{\beta}$.
However, if $d_+, d_-\ge 2$, then by Corollary \ref{cor:dihedral-orbifold-isometry2}(2),
no orientation-preserving involution of 
$\OO(r;d_+,d_-)$ preserves an edge of the singular set different from $\tau_{\pm}$.
This contradicts the assertion (1).
\end{proof}

Case 3. $2\le d_+ \le d_-$.
Since $d_+$ and $d_-$ are coprime, we see
$2\le d_+ < d_-$ and 
one of $d_+$ and $d_-$ is odd.

Subcase 3.1. $|K(r)|=2$ (see Figure \ref{fig.homology1}(3)).
Let $K_1$ and $K_2$ be the components of $K(r)$,
and let $J_{i,j}$ ($1\le i,j\le 2$) be the edges of $\Sigma(\OO)$
such that 
$K_j=J_{1,j}\cup J_{2,j}$ for $j=1,2$
and that the vertical involution of $K(r)$ interchanges 
$J_{i,1}$ and $J_{i,2}$ for $i=1,2$.
Let $x_{\pm}$ and $x_{i,j}$ be the meridians of $\tau_{\pm}$
and $J_{i,j}$, respectively.
Then by using Lemma \ref{lem:homology}
and the fact that 
one of $d_+$ and $d_-$ is odd,
we see that $H_1(\OO;\ZZ_2)\cong (\ZZ_2)^2$ is freely
generated by 
$x_1:=x_{1,1}=x_{2,1}$ and $x_2:=x_{1,2}=x_{2,2}$:
moreover we have $x_{\pm}=0$.
Hence, 
we may assume $e_{\alpha}\subset K_1$
and $e_{\beta}\subset K_2$.
Since the horizontal involution of $K(r)$ interchanges $J_{1,j}$ and $J_{2,j}$ ($j=1,2$),
we may assume $e_{\alpha}=J_{1,1}\subset K_1$ 
and $e_{\beta}=J_{i,2}\subset K_2$ for some $i=1$ or $2$.
By Lemma \ref{lem:repeat}(2), 
we have $h\in\Gamma$, and so
$P_{\alpha}\cong P_{\beta}$
is homeomorphic to $D^2(2,2)$ or $S^2(2,2,2,2)$.
Since $2\le d_+<d_-$, we must have $d_+=2$.
If $i=1$, i.e. $e_{\beta}=J_{1,2}$, then $\tilde w$ is given by 
\[
\tilde w(J_{1,1})=\tilde w(J_{1,2})=\infty, \quad 
\tilde w(J_{1,2})=\tilde w(J_{2,2})=2, \quad \tilde w(\tau_+)=2, \quad \tilde w(\tau_-)=d_-.
\]
Since the vertical involution of $K(r)$ preserves $J_1:=J_{1,1}\cup J_{1,2}$,
we see that
$\orbp$ is isomorphic to $\orbm_2(r;d_-)$.
If $i=2$, i.e. $e_{\beta}=J_{2,2}$,
then the planar involution of $K(r)$ preserves $J_1:=J_{1,1}\cup J_{2,2}$.
Hence, we see by Remark \ref{rem.well-defined-orbifold} that $\orbp$
is isomorphic to $\orbm_2(r';d_-)$,
where $r'=(p+q)/p$.

Subcase 3.2. $|K(r)|=1$ (see Figure \ref{fig.homology1}(4)).
Suppose first that one of $d_+$ and $d_-$ is even.
Then $H_1(\OO;\ZZ_2)\cong (\ZZ_2)^2$ by Lemma \ref{lem:homology}.
Hence, by Lemma \ref{lem:repeat}(2), 
both $e_{\alpha}$ and $e_{\beta}$ are contained in $K(r)$,
and $h\in\Gamma$.
In particular, $P_{\alpha}\cong P_{\beta}\cong D^2(2,2)$ or $S^2(2,2,2,2)$.
Let $e_i$ ($1\le i \le 4$) be the edges of the singular set of $\OO$ 
contained in the knot $K(r)$ in this cyclic order.
We also assume that $\partial\tau_+=(e_1\cap e_2)\cup (e_3\cap e_4)$
and $\partial\tau_-=(e_2\cap e_3)\cup (e_4\cap e_1)$. 
Since the $(\ZZ_2)^2$-symmetry of $\OO(r;d_+,d_-)$ acts transitively 
on the edge set $\{e_i\}_{1\le i \le 4}$ 
(see Figure \ref{symmetry-dihedral-orbifold}),
we may assume $e_1=e_{\alpha}$ and so $\tilde w(e_1)=\infty$.
Since $e_{\alpha}$ joins 
$\tau_+$ with $\tau_-$ 
and since $d_{\pm}$ are coprime integers such that $2\le d_+\le d_-$,
the condition that 
$P_{\alpha}\cong D^2(2,2)$ or $S^2(2,2,2,2)$ implies that
$d_+=2$ and $d_-\ge 3$.
This in turn implies that $P_{\alpha}\cong P_{\beta}\cong D^2(2,2)$.
Since $\partial D^2(2,2)$ is isotopic to the simple loop $\alpha$ in $\partial M_0$,
we must have $\tilde w(e_2)=2$.
Thus $e_{\beta}$ is equal to $e_3$ or $e_4$.
However, if $e_{\beta}=e_4$ then $e_{\alpha}$, $e_{\beta}$, 
and the odd index edge $\tau_-$ 
share a vertex, it follows from Lemma \ref{lem:homology}(3) that
the meridian $\alpha$ of $e_1$ and the meridian $\beta$ of $e_4$
represent the same element of $H_1(\OO;\ZZ_2)\cong (\ZZ_2)^2$, a contradiction.
Hence $e_{\beta}=e_3$.
Set $J_1=e_1\cup e_3$ and $J_2=e_2\cup e_4$.
Then $J_1$ and $J_2$ are disjoint, $K(r)=J_1\cup J_2$ and
the following hold.
\[
\tilde w(J_1)=\infty,\ \tilde w(J_2)=2, \ \tilde w(\tau_+)=2, \
\tilde w(\tau_-)=d_-
\]
Hence we have $\orbp\cong\orbm_2(r;d_-)$
(cf. Remark \ref{rem.well-defined-orbifold}(2)).

Suppose finally that both $d_+$ and $d_-$ are odd.
Then, by Lemma \ref{lem:homology}, the meridians $x_{\pm}$ of $\tau_{\pm}$
represent the trivial element of $H_1(\OO;\ZZ_2)\cong \ZZ_2$,
and hence both 
$e_{\alpha}$ and $e_{\beta}=f(e_{\alpha})$ are contained in $K(r)$.
On the other hand, since $d_{\pm}>2$,
we have $P_{\alpha}\cong P_{\beta}$ is homeomorphic to an annulus,
and hence 
the inverting elliptic element $h$ descends to an involution of $\OO$ 
which preserves each of the two mutually different edges $e_{\alpha}$ and $e_{\beta}$
and restricts to an orientation-reversing involution on each of the edges.
But, such an involution does not exist  
by Corollary \ref{cor:dihedral-orbifold-isometry2}(2), a contradiction.

This completes the proof of Proposition \ref{prop:flexible-cusp-case}.
\end{proof}

\section{Proof of Theorem \ref{main-theorem} - rigid cusp case -}
\label{sec:rigid-cusp-case}

Under Assumption \ref{asmpt:setting-main-theorem},
suppose that $P_{\alpha}\cong P_{\beta}$
is a rigid cusp.
Thus the $2$-orbifold $P_{\alpha}\cong P_{\beta}$ is 
isomorphic to $S^2(p,q,r)$ where 
$(p,q,r)=(2,4,4)$, $(2,3,6)$, or $(3,3,3)$.

Let $G<\Gamma$ be the orbifold fundamental group $\pi_1(P_{\alpha})$,
and let 
$\Lambda$ be the subgroup of $G$ consisting of parabolic transformations.
We may assume that 
(a) $G$ stabilises the ideal point $\infty$
of the upper-half space model of $\HH^3$, and
(b) the boundary $\partial H_{\alpha}$ of the canonical horoball $H_{\alpha}$ is identified with 
the horosphere $\CC\times \{1\}\subset \HH^3$.
For each element $g\in\Lambda$,
let $|g|$ be the length of $g$ in the canonical horosphere 
(see Section \ref{Sec:canonical-horoball-pair}),
namely $|g|=|g|_{\partial H_{\alpha}}$, 
the translation length of $g$
in $\partial H_{\alpha}$, and simply call it the {\it length} of $g$.
Let $L_1(\Lambda)>0$ be the minimum of the lengths of 
nontrivial elements of $\Lambda$.
More generally, for each $n\in\NN$,
let $L_n(\Lambda)$ be the $n$-th shortest length
of nontrivial elements of $\Lambda$.

\medskip

{\bf Case 1.}
$P_{\alpha}\cong S^2(2,4,4)$. 
Then $G\cong \langle a,b,c \ | \ a^2, \ b^4,\ c^4, abc\rangle$,
and $\Lambda$ is the rank $2$ free abelian group with free basis
$\{b^2a, c^2a\}$.
We may assume the action of $G$ on the horosphere $\partial H_{\alpha}=\CC\times 1\cong\CC$
is given by the following rule.
There is a positive real $\ell$ such that
$a$ is the $\pi$ rotation about $0$,
and $b$ and $c$ are the $\pi/2$ rotations about $\ell$ and $\ell i$, respectively.
We can easily observe the following.
\begin{enumerate}
\item[(i)]
The shortest length $L_1(\Lambda)$ is equal to $2\ell$,
and it is attained precisely by the conjugates of $b^2a$ in $G$.
(Note that $c^2a=(b^{-1}a^{-1})^2a=b^{-1}a^{-1}b^{-1}=b^3 a b^{-1} 
=b(b^2 a) b^{-1}$ is conjugate to $b^2a$.)
\item[(ii)]
The second shortest length $L_2(\Lambda)$ is equal to $2\sqrt{2}\ell$,
and it is attained precisely by the conjugates of $b^2ac^2a$ in $G$.
\item[(iii)]
The third shortest length $L_3(\Lambda)$ is equal to $4\ell$,
and it is attained precisely by the conjugates of $(b^2a)^2$ in $G$.
\end{enumerate}

By Lemma \ref{Lem:Brenner}(1),
$2\ell=L_1(\Lambda)\ge 1$, and so $\ell\ge \frac{1}{2}$.
Since $\Gamma$ is non-free,
Lemma \ref{Lem:Brenner}(3) implies that
the length $|\alpha|$ of the parabolic element $\alpha\in \Lambda$
is less than $2$.
Since $L_3(\Lambda)=4\ell\ge 2$,
$|\alpha|$ is equal to either $L_1(\Lambda)$ or $L_2(\Lambda)$.
By using this fact, we obtain the following lemma.

\begin{lemma}
\label{lem:candidate-alpha(2,4,4)}
The parabolic element $\alpha\in \Lambda$
is conjugate to $b^2a$ or $b^2ac^2a$ in $G$.
Moreover the following hold.
\begin{enumerate}
\item
If $\alpha$ is conjugate to $b^2a$,
then the images of $\alpha$ by the natural epimorphisms
from $G\cong \pi_1(S^2(2,4,4))$ to $\pi_1(S^2(2,2,2))$, $\pi_1(S^2(2,2,4))$, and
$\pi_1(S^2(2,4,2))$ have order $2$.
\item
If $\alpha$ is conjugate to $b^2ac^2a$,
then the images of $\alpha$ by the natural epimorphisms
from $G\cong \pi_1(S^2(2,4,4))$ to $\pi_1(S^2(2,2,2))$, $\pi_1(S^2(2,2,4))$, and
$\pi_1(S^2(2,4,2))$ have order $1$, $2$ and $2$, respectively.
Moreover, the $\ZZ_2$-homology class of $\alpha$ vanishes.
\end{enumerate}
\end{lemma}

\begin{proof}
The assertion in the first line
follows from the observations preceding the lemma.
The assertions (1) and (2)
can be checked easily, by using the fact that
$b^2a$ is conjugate to $c^2a$ in $G$.
\end{proof}

\begin{figure}
\includegraphics[width=0.6\hsize, bb=0 0 1422 672]{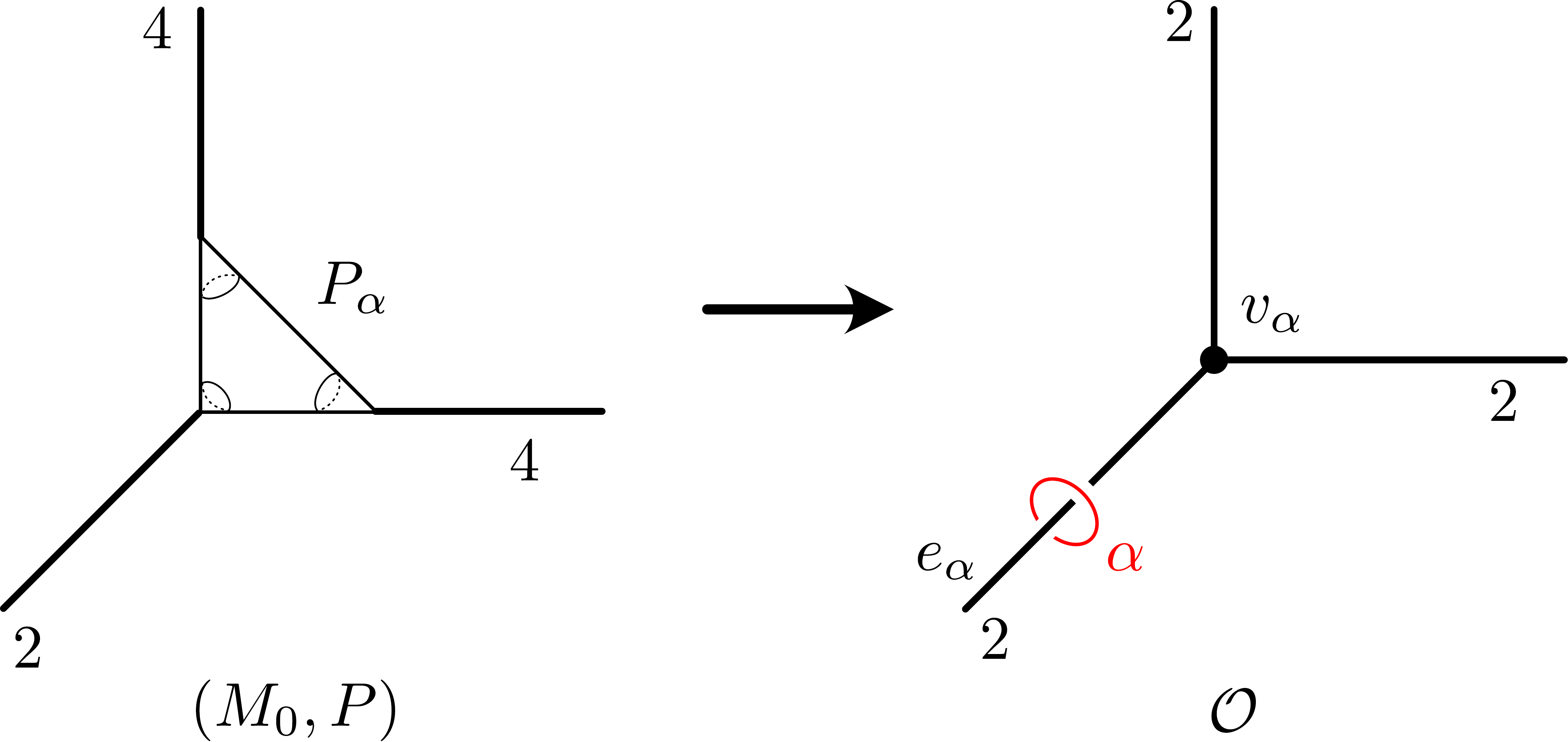}
\caption{Orbifold surgery on the rigid cusp $S^2(2,4,4)$:
The parabolic locus $P_{\alpha}$ of the pared orbifold $(M_0,P)$ 
shrinks into the vertex $v_{\alpha}$ of the singular set of the orbifold $\OO$.
By Lemma \ref{lem:geometric-orb},
the homology class $[\alpha]\in H_1(\OO;\ZZ_2)$ determined by
the parabolic element $\alpha\in\Gamma$
is represented by the meridian of the edge $e_{\alpha}$
of the singular set $\Sigma(\OO)$
incident on $v_{\alpha}$ whose index in the original orbifold $M_0$ is $2$.
}
\label{fig.orb-surgery2}
\end{figure}

Now let $\OO$ be the orbifold obtained from the pared orbifold $\orbp$ by the orbifold surgery
as illustrated in Figure \ref{fig.orb-surgery2}.
Namely, for each index $4$ edge of the singular set which has an endpoint in $P_{\alpha}$
or $P_{\beta}$, we replace the index $4$ with the index $2$,
and then cap all resulting spherical boundary components with discal $3$-orbifolds.
Then each of $P_{\alpha}$ and $P_{\beta}$ shrinks into a vertex of $\OO$
with link $S^2(2,2,2)$,
which we denote by $v_{\alpha}$ and $v_{\beta}$, respectively.
We denote by $e_{\alpha}$ (resp. $e_{\beta}$) the edge of the singular set $\Sigma(\OO)$
incident on $v_{\alpha}$ (resp. $v_{\beta}$) whose index in the original orbifold $M_0$ is $2$. 

\begin{lemma}
\label{lem:geometric-orb}
The orbifold $\OO$ is isomorphic to a spherical dihedral orbifold 
$\OO(r;d_+,d_-)$
for some $r\in\QQ$ and coprime positive integers $d_+$ and $d_-$.
Moreover, $\alpha\in \Lambda$ is conjugate to $b^2a$ in $G$,
and the homology class $[\alpha]\in H_1(\OO;\ZZ_2)$ determined by $\alpha$ is
equal to the meridian of the edge $e_{\alpha}$. 
Similarly, the homology class $[\beta]\in H_1(\OO;\ZZ_2)$ is
equal to the meridian of the edge $e_{\beta}$. 
\end{lemma}

\begin{proof}
By Lemma \ref{lem:candidate-alpha(2,4,4)},
$\alpha$ is conjugate to $b^2a$ or $b^2ac^2a$ in $G=\pi_1(S^2(2,4,4))$,
its image in $\pi_1(S^2(2,2,2))$ has order $2$ or $1$ accordingly.
Hence the image of $\alpha$ in $\pi_1(\OO)$ 
has order $\le 2$.
Moreover, the images of $\alpha$ and $\beta$ have the same order,
because (a) if the exchanging involution $f$ belongs to $\Gamma$ then $\alpha$ and $\beta$ are conjugate in 
$\Gamma$ and so in $\pi_1(\OO)$, and 
(b) if $f\notin\Gamma$ then $f$ descends to an involution on $\OO$
which interchanges the images of $\alpha$ and $\beta$.
Hence $\pi_1(\OO)$ is either the trivial group or a dihedral group.
Since $\OO$ is very good by Lemma \ref{lem:relabeled-orbifold-generic}
and since it has a singular point with link $S^2(2,2,2)$,
$\pi_1(\OO)$ is a noncyclic dihedral group.
Hence, by Theorem \ref{thm:dihedral-orbifold},
$\OO$ is isomorphic to a spherical dihedral orbifold $\OO(r;d_+,d_-)$.

We prove the remaining assertions.
If $\alpha\in\Gamma$ is conjugate to $b^2ac^2a$,
then it
descends to the trivial element of $\pi_1(S^2(2,2,2))$,
and so it represents the trivial element of $\pi_1(\OO)$.
This contradicts the fact that $\pi_1(\OO)$ is a dihedral group
generated by the images of $\alpha$ and $\beta$.
Hence $\alpha$ is conjugate to $b^2a$.
This implies that the $\ZZ_2$-homology class $[\alpha]\in H_1(\OO;\ZZ_2)$
is equal to that represented by the element $a$, and so
it is the meridian of the edge $e_{\alpha}$.  
The existence of the exchanging elliptic element $f$
implies the corresponding assertion for $[\beta]$.
\end{proof}

\begin{lemma}
\label{lem_problematic-weight}
The pared orbifold $\orbp$ is represented by the weighted graph
$(S^3,K(r)\cup\tau_+\cup\tau_-,\tilde w)$
for some $r\in\QQ$,
where
$\tilde w$ is determined by the following rule
(see Figure \ref{fig.problematic-orbifold}):
\[
\tilde w(J_1)=2, \quad \tilde w(J_2)=4, \quad \tilde w(\tau_+)=4, \quad \tilde w(\tau_-)=m, 
\]
for some odd integer $m\ge 3$,
where $J_1$ and $J_2$ are unions of two mutually disjoint edges of the graph
$K(r)\cup\tau_+\cup\tau_-$ distinct from $\tau_{\pm}$,
such that $K(r)=J_1\cup J_2$.
Moreover, $P_{\alpha}$ and $P_{\beta}$ correspond to
distinct endpoints of $\partial \tau_+$
(in the sense of Convention \ref{conv:pared-orbifold}(3)).
\end{lemma}

\begin{proof}
By Lemma \ref{lem:geometric-orb},
$\OO$ is represented by the weighted graph 
$(S^3,K(r)\cup\tau_+\cup\tau_-,w)$
for some $r\in\QQ$,
where $w$ is given by the rule
\[
w(K(r))=2, \quad w(\tau_+)=d_+, \quad w(\tau_-)=d_-
\]
for some coprime positive integers $d_+$ and $d_-$.
Since $\OO$ is obtained from $\orbp$ by an orbifold surgery,
there is a weight function $\tilde w$
on the graph $K(r)\cup\tau_+\cup\tau_-$
such that the pared orbifold $\orbp$ is represented by the 
weighted graph $(S^3,K(r)\cup\tau_+\cup\tau_-,\tilde w)$.
By Remark \ref{remark:upside-down}, we may assume $d_-$ is odd,
if necessary by replacing $r=q/p$ with $q'/p$
where $qq'\equiv 1 \pmod{p}$.
Hence, we see $H_1(\OO;\ZZ_2)\cong (\ZZ_2)^2$ by Lemma \ref{lem:homology}.
Since $H_1(\OO;\ZZ_2)$ is generated by $[\alpha]$ and $[\beta]$,
which are the meridians
of the edges $e_{\alpha}$ and $e_{\beta}$, respectively 
(see Lemma \ref{lem:geometric-orb}), 
we have $e_{\alpha}\ne e_{\beta}$.

Since the links of $v_{\alpha}$ and $v_{\beta}$ are isomorphic to $S^2(2,2,2)$,
we see $w(\tau_+)=2$ and $\{v_{\alpha}, v_{\beta}\}\subset \partial \tau_+$.
Since $e_{\alpha}$ (resp. $e_{\beta}$) is the unique edge of the trivalent graph
$K(r)\cup\tau_+\cup\tau_-$ 
incident on the vertex $v_{\alpha}$ (resp. $v_{\beta}$)
with $\tilde w$-weight $2$,
and since $e_{\alpha}\ne e_{\beta}$,
we see that  $v_{\alpha}$ and $v_{\beta}$ are distinct
endpoints of $\tau_+$.
(If $v_{\alpha}=v_{\beta}$, 
then its \lq link' in $M_0$ is of the form $S^2(2,2,*)\not\cong S^2(2,4,4)$.)
Hence $P_{\alpha}$ and $P_{\beta}$ correspond to
distinct endpoints of $\partial \tau_+$.

We observe that $e_{\alpha}$ and $e_{\beta}$ are not equal to $\tau_+$.
If, say $e_{\alpha}$ was equal to $\tau_+$,
then it is incident on $v_{\beta}\in\partial\tau_+$.
Since $\tilde w(e_{\alpha})=2$, this implies
we have $e_{\alpha}= e_{\beta}$, a contradiction.
This observation implies that
both $e_{\alpha}$ and $e_{\beta}$ are contained in $K(r)$.

We next observe that $d_-\ge 3$.
If $d_-=1$, then the endpoints of $e_{\alpha}$ and $e_{\beta}$ 
are all contained in $\partial\tau_+=\{v_{\alpha},v_{\beta}\}$.
This together with the previous observation
implies that the \lq links' of $v_{\alpha}$ and $v_{\beta}$ in $M_0$
are isomorphic to $S^2(2,2,4)$, a contradiction.

We now show that $e_{\alpha}$ and $e_{\beta}$ are disjoint.
If they are not disjoint, then they share an endpoint of $\tau_-$,
which has odd weight $d_-$.
This implies that the meridians of $e_{\alpha}$ and $e_{\beta}$
represent an identical element of $H_1(\OO;\ZZ_2)$ (see Lemma \ref{lem:homology}(3)),
and so $[\alpha]=[\beta]$, a contradiction. 

Set $J_1:=e_{\alpha}\cup e_{\beta}$
and let $J_2:=\cl(K(r)-J_1)$.
Then $J_1$ and $J_2$ satisfy the desired conclusion with $m=d_-$.
\end{proof}

\begin{figure}
\includegraphics[width=0.6\hsize, bb=0 0 1619 893]{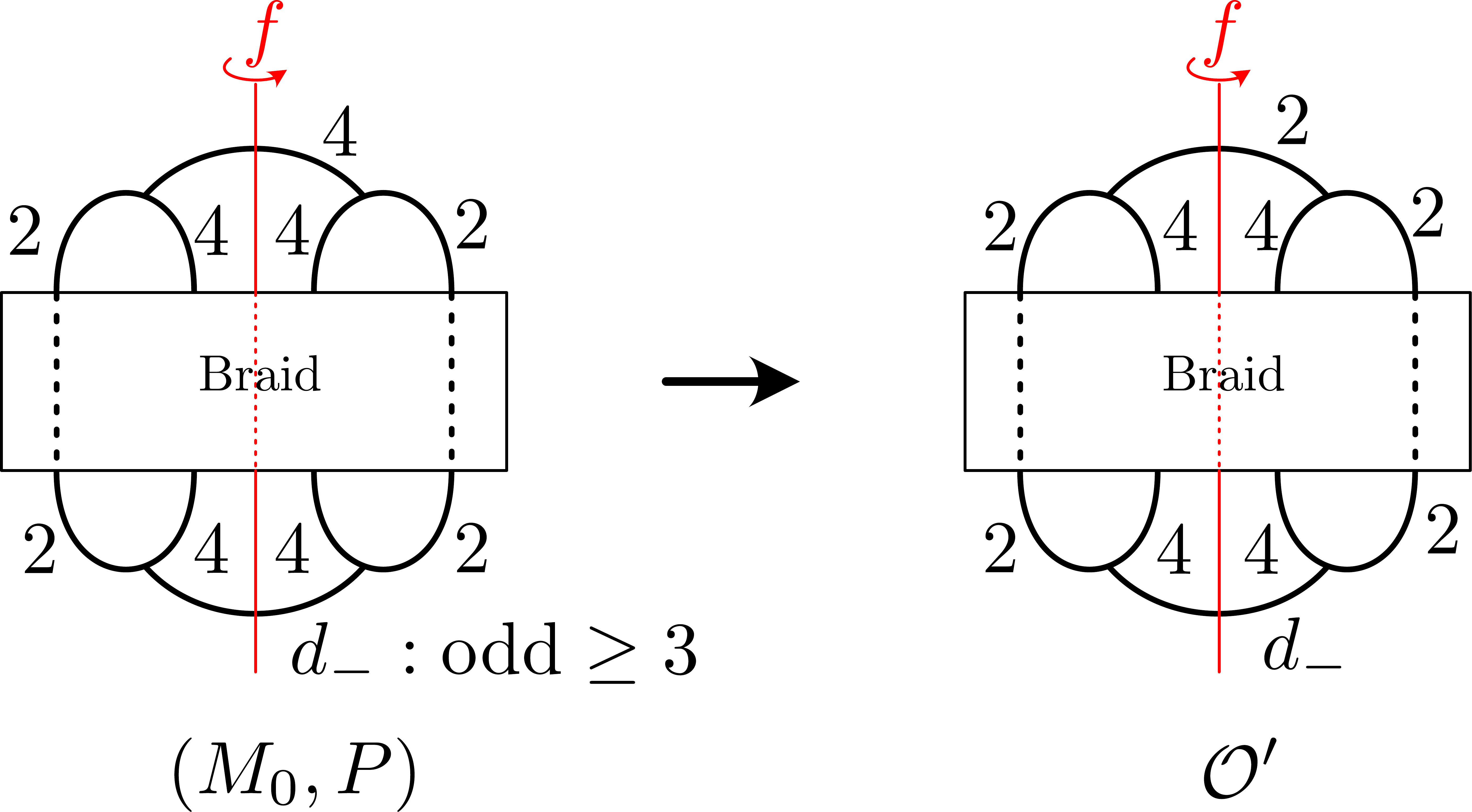}
\caption{The possible pared orbifold $\orbp$ in Lemma \ref{lem_problematic-weight}
and the orbifold $\OO'$ obtained by the orbifold surgery.
In this figure, we apply further normalisation so that $J_1$ and $J_2$ are
invariant by the vertical involution $f$
(cf. Remark \ref{rem.well-defined-orbifold}(3)).  
}
\label{fig.problematic-orbifold}
\end{figure}

We show that the situation described in Lemma \ref{lem_problematic-weight}
cannot happen.
To this end, 
we perform another orbifold surgery on $(M_0,P)$
which replaces the weight $4$ of $\tau_+$ with $2$.
To be precise, we consider the orbifold $\OO'$ represented by
the weighted graph
$(S^3,K(r)\cup\tau_+\cup\tau_-,\tilde w')$
for some $r\in\QQ$,
where $\tilde w'$ is given by the following rule.
\[
\tilde w'(J_1)=2, \quad \tilde w'(J_2)=4, \quad \tilde w'(\tau_+)=2, \quad \tilde w'(\tau_-)=m
\]
Note that $P_{\alpha}\cong S^2(2,4,4)$ shrinks into a singular point of $\OO'$
with link $S^2(2,2,4)$ or $S^2(2,4,2)$.
Since $\alpha$ is conjugate to $b^2a$ in $\pi_1(S^2(2,4,4))< \pi_1(M_0)=\Gamma$,
we see by Lemma \ref{lem:candidate-alpha(2,4,4)}(1) that
the image of $\alpha$ in $\pi_1(\OO')$ has order $\le 2$.
The same argument can be applied to $\beta$ and we see that
the image of $\beta$ in $\pi_1(\OO')$ also has order $\le 2$.
Since $\OO'$ is very good by Lemma \ref{lem:relabeled-orbifold-generic}
and since the singular set of $\OO'$ contains a trivalent vertex,
$\pi_1(\OO')$ is a noncyclic dihedral group.
Since the singular set of $\OO'$ contains four trivalent vertices,
Theorem \ref{thm:dihedral-orbifold} implies that 
$\OO'$ must be isomorphic to a 
spherical dihedral orbifold $\OO(r';d_+',d_-')$
with $d_+',d_-'\ge 2$.
In particular, the singular set $\Sigma(\OO')$ of $\OO'$ must contain precisely
four or five edges with index $2$.
This contradicts the fact that 
$\Sigma(\OO')$ contains precisely three edges
of index $2$ (see Figure \ref{fig.problematic-orbifold}).

\medskip

{\bf Case 2.}
$P_{\alpha}\cong S^2(2,3,6)$. 
Then $G\cong \langle a,b,c \ | \ a^2,\ b^3,\ c^6,\ abc\rangle$,
and $\Lambda$ is the rank $2$ free abelian group with free basis
$\{ac^3, c(ac^3)c^{-1}=cac^2\}$.
We may assume the action of $G$ on the horosphere $\partial H_{\alpha}=\CC\times 1\cong\CC$
is given by the following rule.
There is a positive real $\ell$ such that
$a$ is the $\pi$ rotation about $\sqrt{3}\ell$,
$b$ is the $2\pi/3$ rotation about $2\ell e^{\pi i/6}=\sqrt{3}\ell+\ell i$, and
$c$ is the $\pi/3$ rotations about $0$.
The action of the generators of $\Lambda$ is given by 
\[
ac^3(z)=z+2\sqrt{3}\ell, \quad cac^2(z)=z+2\sqrt{3}\ell e^{\pi i/3}.
\]
We can easily observe the following.
\begin{enumerate}
\item[(i)]
$L_1(\Lambda)=2\sqrt{3}\ell$,
and it is attained precisely by the
conjugates of $ac^3$ in $G$.
\item[(ii)]
$L_2(\Lambda)=6\ell$,
and it is attained precisely by the
conjugates of $(ac^3)(cac^2)=ac^4ac^2$ in $G$.
\item[(iii)]
$L_3(\Lambda)=4\sqrt{3}\ell$, 
and it is attained precisely by the
conjugates of $(ac^3)^2$ in $G$.
\end{enumerate}
By Lemma \ref{Lem:Brenner}(1),
$2\sqrt{3}\ell=L_1(\Lambda)\ge 1$, and so $\ell\ge \frac{1}{2\sqrt{3}}$.
Since $\Gamma$ is non-free,
Lemma \ref{Lem:Brenner}(3) implies that
the length $|\alpha|$ of the parabolic element $\alpha\in \Lambda$
is less than $2$.
Hence we obtain the following.

\begin{lemma}
\label{lem:candidate-alpha(2,3,6)}
The parabolic element $\alpha\in \Lambda$
is conjugate to $ac^3$ or $ac^4ac^2$ in $G$.
\end{lemma}

\begin{figure}
\includegraphics[width=0.6\hsize, bb=0 0 1375 668]{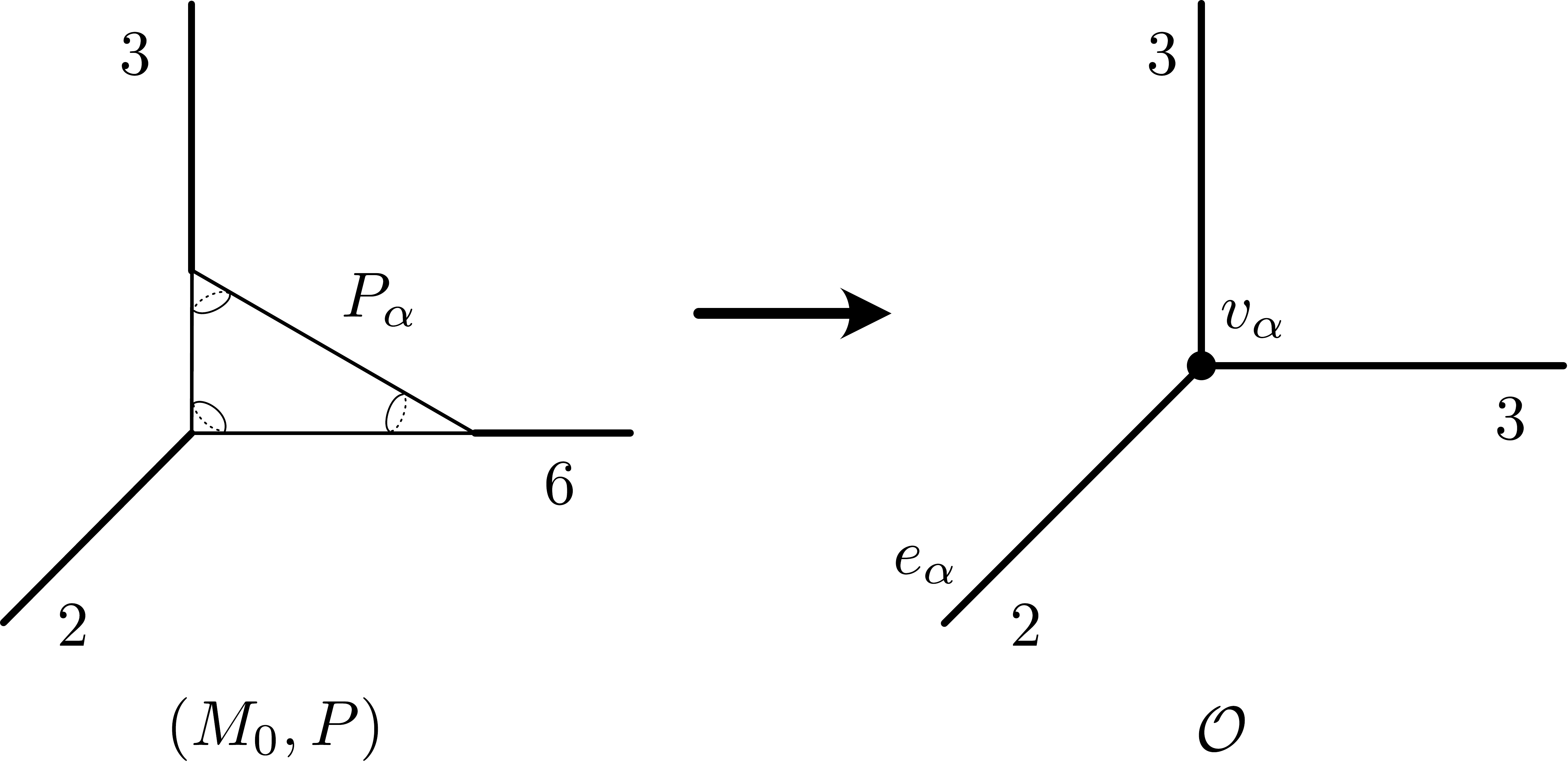}
\caption{Orbifold surgery on the rigid cusp $S^2(2,3,6)$}
\label{fig.orb-surgery3}
\end{figure}

Now let $\OO$ be the orbifold obtained from the pared orbifold
$\orbp$ by the orbifold surgery
as illustrated in Figure \ref{fig.orb-surgery3}.
Namely, for each edge of the singular set which has the index $6$ cone point of $P_{\alpha}$
or $P_{\beta}$ as an endpoint, we replace the weight $6$ with the new weight $3$,
and then cap all  resulting spherical boundary components with discal $3$-orbifolds.
Then $P_{\alpha}$ and $P_{\beta}$ shrink into 
singular points, $v_{\alpha}$ and $v_{\beta}$, 
of $\OO$
with link $S^2(2,3,3)$.

\begin{lemma}
\label{lem:(2,3,3)}
The image of $\alpha$ by the natural epimorphism 
from $\pi_1(S^2(2,3,6))$ to $\pi_1(S^2(2,3,3))$ 
has order $2$.
\end{lemma}

\begin{proof}
By Lemma \ref{lem:candidate-alpha(2,3,6)},
$\alpha$ is conjugate to either $ac^3$ or $ac^4ac^2$ in $G=\pi_1(S^2(2,3,6))$.
Moreover, the images
of $ac^3$ and $ac^4ac^2$ in 
$\pi_1(S^2(2,3,3))\cong \langle a,b,c \ | \ a^2,\ b^3,\ c^3,\ abc \rangle$
have order $2$.
This is obvious for $ac^3$, and 
the assertion for $ac^4ac^2$
is verified as follows.
In $\pi_1(S^2(2,3,3))$, we have
$1=b^3=(ac^{-1})^3$ and so $ac^{-1}=(ac^{-1})^{-2}=(ca)^2$.
Hence the image of $ac^4ac^2$ in $\pi_1(S^2(2,3,3))$ is equal to
$acac^{-1}=ac(ca)^{2}=ac^2aca=(ca)^{-1}a(ca)$.
Thus it is conjugate to $a$,
and so has order $2$, as desired.
\end{proof}

By the above lemma, the image of $\alpha$ in $\pi_1(\OO)$
has order $\le 2$.
The existence of the exchanging elliptic element $f$
implies that the images of $\alpha$ and $\beta$ in $\pi_1(\OO)$ have the same order.
Thus $\pi_1(\OO)$ is either a dihedral group or the trivial group.
Since $\OO$ is very good by Lemma \ref{lem:relabeled-orbifold-generic}
and since the singular set of $\OO$ contains a trivalent vertex,
$\pi_1(\OO)$ is a noncyclic dihedral group.
Hence,
Theorem \ref{thm:dihedral-orbifold} implies that 
$\OO$ must be isomorphic to a spherical dihedral orbifold $\OO(r;d_+,d_-)$
with $(d_+,d_-)\ne (1,1)$.
However, the orbifold $\OO(r;d_+,d_-)$ does not
contain a singular point with link $S^2(2,3,3)$, a contradiction.

\medskip

{\bf Case 3.}
$P_{\alpha}\cong S^2(3,3,3)$.
Then the inverting elliptic element $h$ does not belong to $\Gamma$,
and the group, $\Gamma_h$, obtained from $\Gamma$ by adding $h$
is a $\ZZ_2$-extension of $\Gamma$.
Consider the hyperbolic orbifold $M_h: =\HH^3/ \Gamma_h$.
Then $ M_h$ is the quotient of $M=\HH^3/\Gamma$ by the isometric involution
induced by $h$, which we continue to denote by $h$.
Set $( M_{h,0}, P_h):=(M_0/h, P/h)$,
$P_{h,\alpha}:=P_{\alpha}/h$ and $P_{h,\beta}:=P_{\beta}/h$.
Then $P_{h,\alpha}\cong  P_{h,\beta}$
is isomorphic to $S^2(2,3,6)$.
Thus $G_h:=\pi_1(P_{h,\alpha})\cong
\langle a,b,c \ | \ a^2,\ b^3,\ c^6,\ abc \rangle$.
Since the subgroup $\Gamma$ of $\Gamma_h$ generated by $\alpha$ and $\beta$
is non-free, we see
by the arguments in Case 2
that $\alpha$ is conjugate to $ac^3$ or 
$ac^4ac^2$ in $ G_h$.

Let $\OO_h$ be the orbifold
obtained from the pared orbifold $( M_{h,0}, P_h)$ 
by the orbifold surgery 
as illustrated in Figure \ref{fig.orb-surgery3}
at both $P_{h,\alpha}$ and $P_{h,\beta}$.
Then $P_{h,\alpha}$
and $ P_{h,\beta}$ shrink into singular points, $v_{h,\alpha}$ and $v_{h,\beta}$,
of $\OO_h$ with link $S^2(2,3,3)$.
The images of $\alpha$ and $\beta$ in $\pi_1(\OO_h)$ have the same order $\le 2$,
and so the subgroup of $\pi_1(\OO_h)$ they generate
is either a dihedral group or the trivial group.
This subgroup has index $\le 2$ in $\pi_1(\OO_h)$, because
$\Gamma=\langle \alpha, \beta\rangle$ has index $2$ in $\Gamma_h$.
Hence the group $\pi_1(\OO_h)$ is a trivial group, a dihedral group,
$\ZZ_2$ (the $\ZZ_2$-extension of the trivial group) or 
a $\ZZ_2$-extension of a dihedral group.

Since $\OO_h$ is very good by Lemma \ref{lem:relabeled-orbifold-generic}
and $\OO_h$ contains a singular point
with link $S^2(2,3,3)$,
$\pi_1(\OO)$ is either a noncyclic dihedral group 
or a $\ZZ_2$-extension of a noncyclic dihedral group.
Hence Theorem \ref{thm:dihedral-orbifold} implies that 
$\OO_h$ is isomorphic to 
(a) a spherical dihedral orbifold $\OO(r;d_+,d_-)$
or 
(b) the quotient of $\OO(r;d_+,d_-)$ by an isometric involution,
where $(d_+,d_-)\ne (1,1)$.
Since $\OO(r;d_+,d_-)$ does not have a singular point with link $S^2(2,3,3)$,
(a) cannot happen,
and so we may assume (b) holds.
Since $\pi_1(S^2(2,3,3))$ does not have an index $2$ subgroup,
the link of an inverse image of the singular point $v_{h,\alpha}$
in the double cover $\OO(r;d_+,d_-)$ of $\OO_h$
is also isomorphic to $S^2(2,3,3)$.
But, this is impossible.
Hence $P_{\alpha}$ cannot be isomorphic to $S^2(3,3,3)$.

\medskip

Thus we have proved that $P_{\alpha}\cong P_{\beta}$
cannot be a rigid cusp.

\section{Proof of Theorem \ref{main-theorem} - flexible cusp: exceptional case -}
\label{sec:exceptional-flexible}

In this section, we treat the case where the following assumption is satisfied,
and prove that this assumption is never satisfied.

\begin{assumption}
\label{assumption:flexible2}
{\rm
Under Assumption \ref{asmpt:setting-main-theorem},
we further assume that $P_{\alpha}=P_{\beta}$ and it
is a flexible cusp $S^2(2,2,2,2)$ and that the conclusion (2) in Lemma \ref{lem:simple} holds. 
Namely, $f\notin\Gamma$,
and $P_{\alpha}/ f=P_{\beta}/f \cong S^2(2,4,4)$
(see Figure \ref{fig.involution-f}).
}
\end{assumption}

Let $\hat\Gamma:=\langle \Gamma, f\rangle$
be the group generated by $\Gamma$ and $f$.
Let $\hat M:=\HH^3/\hat \Gamma$ be the quotient hyperbolic orbifold.
Let $\hat M_0$ be the non-cuspidal part of $\hat M$,
and $\hat P=\partial \hat M_0$ the parabolic locus.
By abuse of notation, we denote the pared orbifold
obtained as the relative compactification of $(\hat M_0,\hat P)$
by the same symbol $(\hat M_0,\hat P)$.
We denote the component of the compact euclidean $2$-orbifold $\hat P$
corresponding to the conjugacy class containing $\alpha$ and $\beta=f\alpha f^{-1}$
by $\hat P_{\alpha\beta}$.
Thus $\hat P_{\alpha\beta}\cong P_{\alpha}/ f=P_{\beta}/f \cong S^2(2,4,4)$
and $(M_0,P)/f\cong(\hat M_0,\hat P)$,
where $f$ denotes the involution on the pared orbifold $(M_0,P)$
induced by the exchanging involution $f$.
In particular,
$M_0$ is the double orbifold covering of
$\hat M_0$, 
associated with the homomorphism 
$\xi:\pi_1(\hat M_0)=\Gamma \to \ZZ_2$
such that $\xi(\alpha)=\xi(\beta)=0$ and $\xi(f)=1$.
We denote the homomorphism $H_1(\hat M_0;\ZZ_2) \to \ZZ_2$
induced by $\xi$ by the same symbol.

Note that
$\pi_1(\hat P_{\alpha\beta})\cong\pi_1(S^2(2,4,4))\cong
\langle a,b,c \ | \ a^2, \ b^4,\ c^4, abc\rangle$.
As in Case 1 in Section \ref{sec:rigid-cusp-case},
we identify $\pi_1(\hat P_{\alpha\beta})$ with 
the stabiliser $\Stab_{\hat\Gamma}(\fix(\alpha))$.
Then the proof of Lemma \ref{lem:candidate-alpha(2,4,4)}
also works in this setting,
because $\{\alpha,\beta\}$ generates the non-free subgroup $\Gamma$
of the Kleinian group $\hat\Gamma$,
and we have the following lemma. 

\begin{lemma}
\label{lem:candidate-alpha(2,4,4)-2}
The parabolic element $\alpha$
is conjugate to $b^2a$ or $b^2ac^2a$ in 
$\Stab_{\hat\Gamma}(\fix(\alpha))$ $\cong$ $\pi_1(S^2(2,4,4))$,
and of course,
the assertions (1) and (2) in Lemma \ref{lem:candidate-alpha(2,4,4)}
also hold.
\end{lemma}

\begin{figure}
\includegraphics[width=0.8\hsize, bb=0 0 1060 532]{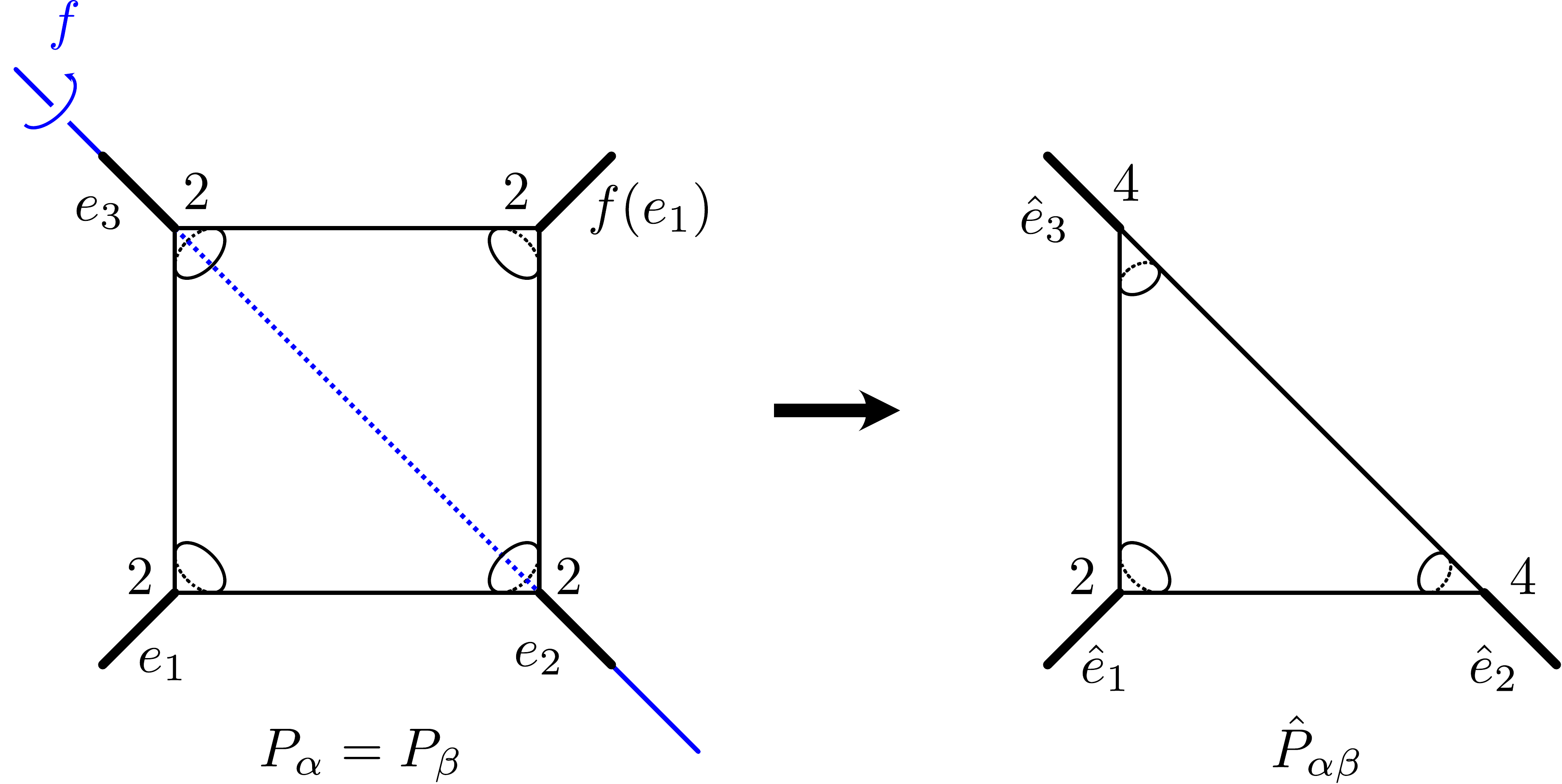}
\caption{
Assumption \ref{assumption:flexible2} assumes that $f\notin\Gamma$
descends to an involution, $f$, on $P_{\alpha}=P_{\beta}\cong S^2(2,2,2,2)$
such that $P_{\alpha}/ f=P_{\beta}/f \cong S^2(2,4,4)$.
}
\label{fig.involution-f}
\end{figure}

Let $e_i$ and $\hat e_i$ ($i=1,2,3$) be the edges of the singular sets
$\Sigma(M_0)$ and $\Sigma(\hat M_0)$ as illustrated in Figure \ref{fig.involution-f}.
Thus $e_2$ and $e_3$ are contained in the fixed point set of the
involution $f$ on $M_0$,
and $\hat e_i$ is the image of $e_i$ by the covering projection 
$M_0\to\hat M_0$.
(Note that it can happen that some of them are identical,
though their germs near the parabolic locus are different.)
Then the following holds.

\begin{lemma}
\label{lem:monodormy}
The homomorphism $\xi:H_1(\hat M_0;\ZZ_2) \to \ZZ_2$,
that determines the double orbifold covering $M_0\to \hat M_0$,
satisfies
\[
\xi(m_{1})=0, \quad \xi(m_{2})=\xi(m_{3})=1,
\]
where $m_{i}$ denotes the meridian of the edge $\hat e_i$.
Moreover, the homology class $[f]\in H_1(\hat M_0;\ZZ_2)$ 
determined by $f\in\hat\Gamma$
is equal to either $m_{2}$ or $m_{3}$.
\end{lemma}

\begin{proof}
The formula for $\xi$ follows from the fact that 
the fixed point set of the involution $f$ on $M_0$
contains $e_2$ and $e_3$, which project to $\hat e_2$ and $\hat e_3$,
respectively.
It is also obvious that $\xi([f])=1$.
So, if $H_1(\hat M_0;\ZZ_2)\cong \ZZ_2$,
we have $[f]=m_{2}=m_{3}$.
Suppose that $H_1(\hat M_0;\ZZ_2)\cong (\ZZ_2)^2$.
Then, since $H_1(\hat M_0;\ZZ_2)$ is generated by $[f]$ and $[\alpha]$,
we see $[\alpha]\ne 0$.
So, $\alpha$ is conjugate to $b^2a$ by Lemma \ref{lem:candidate-alpha(2,4,4)-2}.
(Otherwise $\alpha$ is conjugate to $b^2ac^2a$ and so $[\alpha]=0$.)
Therefore 
$[\alpha]=[a]=m_{1}$ is contained in $\Ker(\xi)\cong\ZZ_2$.
Thus $\Ker(\xi)$ is generated by $m_{1}$.
Since $\xi([f])=\xi(m_{2})$, 
it follows that $[f]$ is equal to either $m_{2}$ or
$m_{1}+m_{2}=m_{3}$.
\end{proof}

Let $\hat\OO$ be the orbifold obtained from 
the pared orbifold $(\hat M_0,\hat P_{\alpha\beta})$ 
by the orbifold surgery that replaces the index $4$ of the edges $\hat e_2$ and $\hat e_3$
with the index $2$.
Then $\hat P_{\alpha\beta}$ shrinks into a singular point, 
$v_{\alpha\beta}$, with link $S^2(2,2,2)$, and
the image of $\alpha$ in $\pi_1(\hat\OO)$ has order $\le 2$
by Lemma \ref{lem:candidate-alpha(2,4,4)-2}.
Since $\hat\Gamma$ is generated by $f$ and $\alpha$,
$\pi_1(\hat\OO)$ is either trivial, $\ZZ_2=D_1$ or a noncyclic dihedral group.
By using Lemma \ref{lem:relabeled-orbifold-generic},  
Theorem \ref{thm:dihedral-orbifold} and the fact that
$\hat\OO$ has a singular point with link $S^2(2,2,2)$,
we see that $\hat\OO$ is isomorphic to a spherical dihedral orbifold
$\OO(r;d_+,d_-)$ with noncyclic dihedral orbifold fundamental group.
Moreover, we may assume that
$d_+=2$ and that $v_{\alpha\beta}$ is an endpoint of $\tau_+$.
By Lemma \ref{lem:homology}, we have $H_1(\hat\OO;\ZZ_2)\cong (\ZZ_2)^2$.

\begin{lemma}
\label{lem:distinct-edge}
Under the above setting, 
$|K(r)|=1$ and so
the edges $\hat e_i$ ($i=1,2,3$) are all distinct.
\end{lemma}

\begin{proof}
We first observe that $\alpha$ cannot be conjugate to $b^2ac^2a$.
In fact, if $\alpha$ was conjugate to $b^2ac^2a$,
then its image in $\pi_1(\hat\OO)$ is trivial
by Lemma \ref{lem:candidate-alpha(2,4,4)-2} 
(cf. Lemma \ref{lem:candidate-alpha(2,4,4)}(2)),
and so $\pi_1(\hat\OO)$ is generated by the image of $f$.
This contradicts the fact that $\pi_1(\hat\OO)$ is a noncyclic dihedral group.
This observation together with Lemma \ref{lem:candidate-alpha(2,4,4)-2} 
implies that 
$\alpha$ is conjugate to $b^2a$ and so 
$[\alpha]=[a]=m_{1}\in \Ker(\xi)$.
Moreover, $[f]=m_2$ or $m_3$ by Lemma \ref{lem:monodormy}.
Hence $H_1(\hat\OO;\ZZ_2)\cong (\ZZ_2)^2$ is generated by the meridians of 
the three edges $\hat e_i$ ($i=1,2,3$) incident on
the vertex $v_{\alpha\beta}\in\partial\tau_+$.

Now suppose on the contrary that $|K(r)|=2$.
Then we see, by using Lemma \ref{lem:homology}, 
that the meridian of $\tau_+$ represents the trivial element of $H_1(\hat\OO;\ZZ_2)$
and the meridians of the remaining two edges incident on 
$v_{\alpha\beta}\in\partial\tau_+$ represent the identical element of
$H_1(\hat\OO;\ZZ_2)$.
This contradicts the fact that $H_1(\hat\OO;\ZZ_2)\cong (\ZZ_2)^2$.
Hence $|K(r)|=1$.
This implies that the the edges $\hat e_i$ ($i=1,2,3$) incident on 
$v_{\alpha\beta}\in\partial\tau_+$ are all distinct,
as desired.
\end{proof}

Recall that the weights of the edges $\hat e_1$, $\hat e_2$, $\hat e_3$ of 
$\Sigma(\hat M_0)$ are $2,4,4$.
Since $\hat e_2\ne \hat e_3$ by Lemma \ref{lem:distinct-edge},
we can apply the orbifold surgery on $(\hat M_0, \hat P)$
of \lq\lq type $(2,4,4)\to (2,2,4)$'',
namely we can replace the index $4$ of the edge $\hat e_2$ of
the singular set $\Sigma(\hat M_0)$ with the index $2$,
and leave the other indices, including the index $4$ of $\hat e_3$, unchanged.
We denote the resulting orbifold by $\hat\OO_{(2,2,4)}$.
By Lemma \ref{lem:candidate-alpha(2,4,4)-2},
$\alpha$ has order at most $2$ in $\pi_1(\hat\OO_{(2,2,4)})$.
Hence, by using Lemma \ref{lem:relabeled-orbifold-generic}, 
Theorem \ref{thm:dihedral-orbifold},
and the fact that $\hat\OO_{(2,2,4)}$ has a singular point with link $S^2(2,2,4)$,
we see that
$\hat\OO_{(2,2,4)}$ is isomorphic to a spherical dihedral orbifold
$\OO(r;d_+,d_-)$ with noncyclic dihedral orbifold fundamental group.
Moreover, we may assume $d_+=4$
and that the parabolic locus $P_{\alpha\beta}$
degenerates into a singular point, $v_{\alpha\beta}$, 
which is an endpoint of $\tau_+$.
It should be noted that the edge $\hat e_3$ of $\Sigma(\hat M_0)$
corresponds to $\tau_+$.
(Here, we reset the notation, and the symbols $\OO(r;d_+,d_-)$ and $v_{\alpha\beta}$
now represent objects different from those they had represented
in the paragraph preceding Lemma \ref{lem:distinct-edge}.)

\smallskip

Case 1. $d_-\ge 3$. 
We apply the orbifold surgery on $(\hat M_0, \hat P)$
of \lq\lq type $(2,4,4)\to (2,4,2)$'',
namely we replace the index $4$ of the edge $\hat e_3=\tau_+$ of
the singular set $\Sigma(\hat M_0)$ with the index $2$,
and leave the other indices, including the index $4$ of $\hat e_2$, unchanged.
(This is possible by Lemma \ref{lem:distinct-edge}.) 
We denote the resulting orbifold by $\hat\OO_{(2,4,2)}$.
By Lemma \ref{lem:candidate-alpha(2,4,4)-2},
$\alpha$ has order at most $2$ in $\pi_1(\hat\OO_{(2,4,2)})$.
Hence, again by using Lemma \ref{lem:relabeled-orbifold-generic}, 
Theorem \ref{thm:dihedral-orbifold},
and the fact that $\hat\OO_{(2,4,2)}$
has a singular point with link $S^2(2,4,2)$,
we see that
$\hat\OO_{(2,4,2)}$ is isomorphic to a spherical dihedral orbifold
with noncyclic dihedral orbifold fundamental group.
Note that the edges $\hat e_2$ and $\tau_-$ of $\Sigma(\hat\OO_{(2,4,2)})$,
which have indices $4$ and $d_-\ge 3$, respectively, 
share a common endpoint (see Figure \ref{Bad-edge}).
But this cannot happen in any 
spherical dihedral orbifold, a contradiction.

\begin{figure}
\includegraphics[width=0.6\hsize, bb=0 0 1212 795]{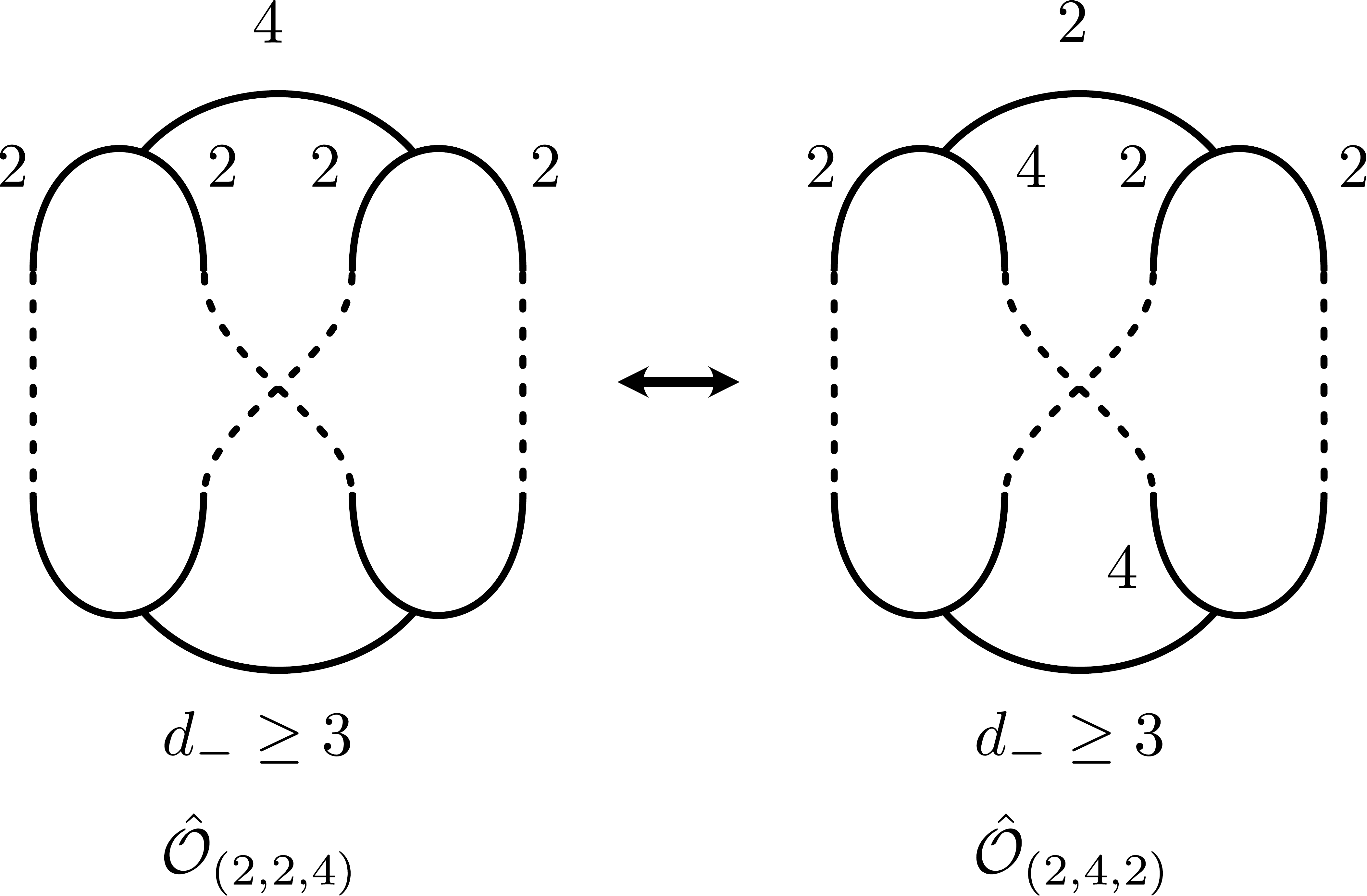}
\caption{Since $\hat\OO_{(2,2,4)}\cong\OO(r;d_+,d_-)$ with $\hat e_3=\tau_+$
is as in the left figure, $\hat\OO_{(2,4,2)}$ is as in the right figure.
The latter orbifold has a singular point with link $S^2(2,4, d_-)$ with $d_-\ge 3$,
and so it cannot be a spherical dihedral orbifold.
}
\label{Bad-edge}
\end{figure}

\smallskip

Case 2. $d_-=1$.
Then $(\hat M_0,\hat P)$ is represented by
a weighted graph 
$(S^3,K(r)\cup\tau_+,\hat w)$,
such that 
$K(r)=\hat e_1\cup \hat e_2$ is a knot, $\hat e_3=\tau_+$, and
\[
\hat w(\tau_+)=4, \quad \hat w(\hat e_1)=2, \quad \hat w(\hat e_2)=4.
\]
Recall that the subset $\hat e_2\cup\hat e_3=\hat e_2\cup \hat \tau_+$
of $\Sigma(\hat M_0)$ are the images of the fixed point set  
of the involution $f$ on $M_0$.
This implies that the map $|M_0|\to |\hat M_0|$ induced 
by the orbifold covering $M_0\to \hat M_0$
is the double branched covering branched over $\hat e_2\cup\hat e_3$.
Hence $(M_0,P)$ is represented by the weighted graph 
illustrated in Figure \ref{Covering-graph}.
Here, we assume the extended Convention \ref{conv:pared-orbifold2},
and the two $4$-valent vertices represent parabolic loci
isomorphic to $S^2(2,2,2,2)$.
Hence we see by Lemma \ref{lem:relabeled-orbifold-generic}
that $H_1(M_0;\ZZ)\cong(\ZZ_2)^3$, a contradiction.

\medskip

Thus we have proved that the situation in 
Assumption \ref{assumption:flexible2} cannot occur.
This completes the proof of the main Theorem \ref{main-theorem}.

\begin{figure}
\includegraphics[width=0.6\hsize, bb=0 0 1480 771]{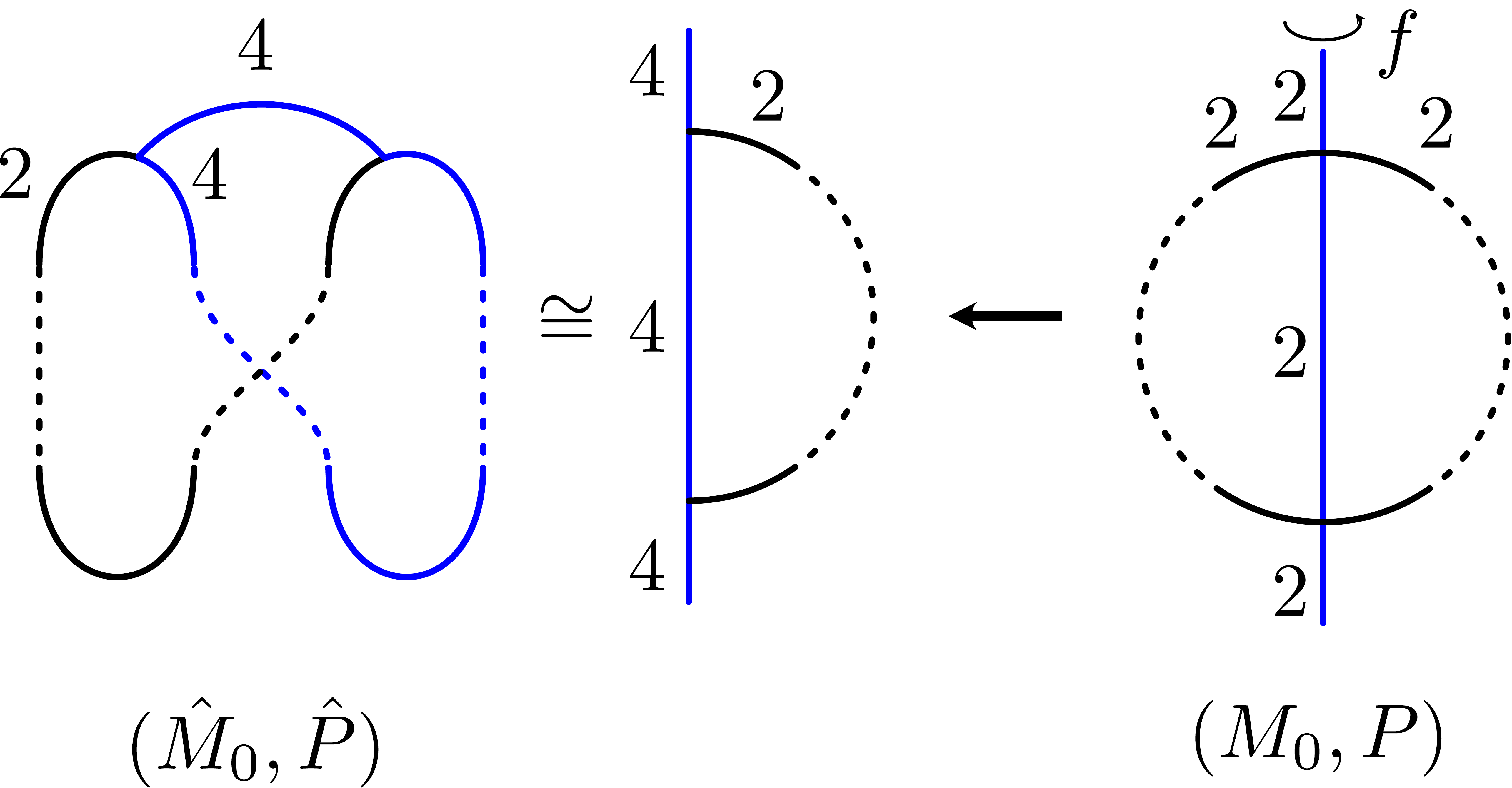}
\caption{Since the orbifold covering $M_0\to \hat M_0$
induces the double branched covering $|M_0|\to |\hat M_0|$
branched over $\hat e_2\cup\hat e_3=\hat e_2\cup \hat \tau_+$,
the orbifold $(M_0,P)$ is as illustrated in the right figure.
}
\label{Covering-graph}
\end{figure}

\section{Appendix 1: Spherical orbifolds with dihedral orbifold fundamental groups}
\label{sec:dihedral-orbifold}

In this appendix, we classify the orientable spherical $3$-orbifolds
with dihedral orbifold fundamental groups (Proposition \ref{prop:dihedral-orbifold1}),
and determine the (orientation-preserving) isometry groups of these orbifolds
(Propositions \ref{prop:dihedral-orbifold-isometry1} and \ref{prop:dihedral-orbifold-isometry2}).
Proposition \ref{prop:dihedral-orbifold1} is used in the proof of Theorem \ref{thm:dihedral-orbifold}, and Corollary \ref{cor:dihedral-orbifold-isometry2}
is used in Section \ref{sec:flexible-cusp-case}.
Propositions \ref{prop:dihedral-orbifold-isometry1} and \ref{prop:dihedral-orbifold-isometry2}
are used in the companion \cite{ALSS} of this paper.
The classification of the spherical dihedral orbifolds
is implicitly contained in Dunber's work \cite{Dunbar2},
which classifies the Seifert fibered orbifolds.
The isometry groups of the dihedral spherical orbifolds
obtained as the $\pi$-orbifolds associated with $2$-bridge links are calculated by
\cite{Sakuma2, Jeevanjee}.
Moreover, in the recent papers \cite{Mecchia-Seppi, Mecchia-Seppi2},
Mecchia and Seppi classified the Seifert fibered spherical $3$-orbifolds
and calculated the isometry groups of such orbifolds.
Since every spherical dihedral orbifold is Seifert fibered,
the results in this section are implicitly contained 
in \cite{Mecchia-Seppi, Mecchia-Seppi2}.
However, we give a self-contained proof,
because it is not a simple task to translate their results into the form 
we need.

We first recall basic facts concerning the $3$-dimensional spherical geometry
following \cite{Scott, Sakuma2}.
Let $\quaternion$ be the quaternion skew field.
We use the symbol $q$ to denote a generic quaternion 
\[
q=a+bi+cj+dk \in \quaternion \quad (a,b,c,d\in\RR).
\]
(We believe this does not cause any confusion,
even though $q$ is also used to denote the numerator of a rational number $r=q/p$.)
For each $q \in \quaternion$, 
$\bar q=a-bi-cj-dk$ denotes its conjugate,
$\Re(q)=a$ denotes its real part,
and $|q|$ denotes its norm $\sqrt{q\bar q}=\sqrt{a^2+b^2+c^2+d^2}$.
We identify $S^n$ ($n=1,2,3$) with the following subspaces of $\quaternion$.
\begin{align*}
S^3 &:=\{q\in \quaternion \svert |q|=1\}\\
S^2 &:=\{q\in \quaternion \svert |q|=1,\quad \Re(q)=0\}\\
S^1 &:=\{z\in\CC\subset \quaternion \svert |z|=1\}
\end{align*}
The norm $|\cdot|$ induces a Euclidean metric on $\quaternion$, and
$S^n$ ($n=1,2,3$) are endowed with the induced metrics.
The subspaces $S^3$ and $S^1$ form a Lie group
with respect to the restriction of the product in $\quaternion$.
The group $S^3$ acts on itself by conjugation leaving $S^2$ invariant.
This gives an epimorphism $\psi:S^3\to \Isom^+(S^2)$, 
with $\ker \psi =\langle -1\rangle$, defined by
\[
\psi(q)(x)=qx\bar q \quad (q\in S^3,\ x\in S^2).
\]
If $q=\cos\theta + q_0 \sin\theta $ with $q_0\in S^2$,
then $\psi(q)$ is the rotation of $S^2$, by angle $2\theta$, with fixed points $\pm q_0$.

For a positive integer $n$,
any cyclic subgroup of order $n$ 
(resp. any dihedral subgroup of order $2n$)
of $\Isom^+(S^2)$
is conjugate to the subgroup $\ZZ_n:=\psi(\ZZ_n^*)$
(resp. $\Di_n:=\psi(\Di_n^*)$),
where $\ZZ_n^*:=\langle \omega\rangle$ and 
$\Di_n^* :=\langle \omega,\ j \rangle$
with $\omega =\exp(\pi i/n)$. 
Note that these groups are contained in the subgroup 
$\Di_S:=\langle S^1, j\rangle=S^1\sqcup S^1j$ of $S^3$.
Then the following hold (see, e.g. \cite[Proposition 2.6]{Sakuma2}).

\begin{lemma}
\label{lem.2-dim-normaliser}
{\rm (1)} 
If $n\ge 2$, then the normaliser $N(\ZZ_n^*)$ of $\ZZ_n^*$ in $S^3$ is equal to 
$\Di_S$.

{\rm (2)} If $n\ge 3$, then the normaliser $N(\Di_n^*)$ of $\Di_n^*$ in $S^3$ is equal to 
$\Di_{2n}^*$.
If $n=2$, then $N(\Di_n^*)$ is equal to the binary octahedral group $O^*=\psi^{-1}(O)$,
where $O<\Isom^+(S^2)$ is the octahedral group obtained as the 
subgroup of $\Isom^+(S^2)$ preserving the regular octahedron in 
the $3$-dimensional Euclidean subspace $\langle i,j,k\rangle$ of $\quaternion$ 
spanned by the $6$ vertices $\{\pm i, \pm j, \pm k\}$.  
\end{lemma}

Let $\phi: S^3\times S^3 \to \Isom^+(S^3)$ be the homomorphism
defined by
\[
\phi(q_1,q_2)(q)=q_1 q q_2^{-1}.
\]
Then $\phi$ is an ephimorphism with $\Ker\phi =\langle (-1,-1)\rangle \cong \ZZ_2$.

We occasionally identify $S^3\subset\quaternion$ with the unit sphere
\[
S^3=\{(z_1,z_2)\in\CC^2 \svert |z_1|^2+|z_2|^2 = 1\}
\]  
in $\CC^2$
by the correspondence $q=z_1+z_2j \leftrightarrow (z_1,z_2)$.
Let $\LM:S^1\times S^1\to \Isom^+(S^3)$ be the injective homomorphism 
defined by 
\[
\LM(\omega_1,\omega_2)(z_1,z_2)=(\omega_1 z_1,\omega_2 z_2).
\]
When $\omega_{\ell}=\exp(2\pi i\frac{k_{\ell}}{n_{\ell}})$ (${\ell}=1,2$),
where $\frac{k_{\ell}}{n_{\ell}}$ is a rational number,
we write  
\begin{align}
\label{LM-F}
\LM(\omega_1,\omega_2)=\LM(\frac{k_1}{n_1},\frac{k_2}{n_2}),
\end{align}
because its restriction to the circles $S^3\cap(\CC\times\{0\})$ and
$S^3\cap(\{0\}\times \CC)$ are 
the \lq $\frac{k_1}{n_1}$-rotation' and \lq $\frac{k_2}{n_2}$-rotation', respectively.
Though the symbol $\LM(\cdot,\cdot)$ is used in two different ways,
we believe this does not cause any confusion,
because its meaning is clearly understood from the context
according to whether $\cdot$ is a unit complex or a rational number.

Observe that
\[
\phi(\eta_1,\eta_2)=\LM(\eta_1\bar\eta_2,\eta_1\eta_2) \quad ((\eta_1,\eta_2)\in S^1\times S^1).
\]
In particular, we have
\begin{align}
\label{LM-I}
\phi(S^1\times S^1)=\LM(S^1\times S^1)<\Isom^+(S^3).
\end{align}

Consider the isometries  $J:=\phi(j,j)$, $J_1:=\phi(1,j)$ and $J_2:=\phi(j,1)$,
which acts on $S^3\subset \CC^2$ as follows.
\[
J(z_1,z_2)=(\bar z_1,\bar z_2), \ J_1(z_1,z_2)=(z_2,-z_1),\ J_2(z_1,z_2)=(-\bar z_2,\bar z_1)
\]
Observe $J=J_1J_2$ and that
\begin{align}
\label{LM-D}
\phi(\Di_S\times \Di_S)
=\langle \LM(S^1\times S^1), J, J_1\rangle,
\quad 
\langle J, J_1\rangle \cong\ZZ_2\times\ZZ_2.
\end{align}
In fact,
$\phi(\Di_S\times \Di_S)$ is the split extension of 
$\LM(S^1\times S^1)$ by 
$\langle J, J_1\rangle \cong\ZZ_2\times\ZZ_2$,
where the action of $\langle J, J_1\rangle$ on 
$\LM(S^1\times S^1)$
by conjugation is given by the following formula.
\begin{align}
\label{LM-D2}
J \LM(\omega_1,\omega_2) J^{-1}
=\LM(\bar \omega_1,\bar \omega_2),
\quad
J_1 \LM(\omega_1,\omega_2) J_1^{-1}
=\LM(\omega_2,\omega_1)
\end{align}

The following proposition gives a classification 
of the orientable spherical $3$-orbifolds with dihedral orbifold fundamental groups.

\begin{proposition}
\label{prop:dihedral-orbifold1}
Let $\OO$ be an oriented spherical $3$-orbifold. 
Then $\pi_1(\OO)$ is isomorphic to
a dihedral group, 
if and only if
$\OO$ is isomorphic to the orbifold,
$\OO(r;d_1,d_2)$, represented by the weighted graph
$(S^3,K(r)\cup\tau_+\cup\tau_-,w)$ in Figure \ref{fig.Dihedral-orbifold}
for some $r\in\QQ$ and coprime positive integers $d_1$ and $d_2$,
where $w$ is given by the following rule.
\[
w(K(r))=2, \quad w(\tau_+)=d_1, \quad w(\tau_-)=d_2.
\]
In fact, $\OO(r;d_1,d_2)$ with $r=q/p$ is isomorphic to $S^3/\Gamma$,
where $\Gamma$ is the subgroup of $\Isom^+(S^3)$ given by
\begin{align}
\label{spherical-dihedral-group}
\Gamma=
\left\langle \LM(\frac{k_1}{pd_2},\frac{k_2}{pd_1}), \ J \right\rangle
\cong
D_{n}
\quad
\mbox{with $n=pd_1d_2$}
\end{align}
for some integers $k_1$ and $k_2$ such that
\begin{align}
\label{gcd-condition}
\gcd(pd_2,k_1)=1, \quad \gcd(pd_1,k_2)=1, \quad k_2\equiv qk_1 \pmod{p}.
\end{align}

Moreover, the spherical structure of 
$\OO(r;d_1,d_2)$ is unique,
i.e., if $\Gamma'$ is a subgroup of $\Isom^+(S^3)$
such that $S^3/\Gamma'$ is isomorphic to $\OO(r;d_1,d_2)$ as oriented orbifolds,
then $\Gamma'$ is conjugate to the subgroup $\Gamma$ defined by (\ref{spherical-dihedral-group}).
\end{proposition}

\begin{proof}
We first prove the only if part of the first assertion.
Let $\Gamma$ be a subgroup of $\Isom^+(S^3)$ isomorphic to the dihedral group $D_{n}$,
and let $f$ and $h$ be the elements of $\Gamma$ 
such that
\[
\Gamma \cong \langle f,h \svert f^n=1, h^2=1, hfh^{-1}=f^{-1} \rangle.
\]
(Though the symbols $\Gamma$, $f$ and $h$ are used in different meanings in the previous sections,
we believe this does not cause any confusion.)
We show that $S^3/\Gamma$ is isomorphic to some $\OO(q/p;d_1,d_2)$,
such that $n=pd_1d_2$.

\begin{claim}
\label{claim:normalise-f}
After taking conjugation in $\Isom^+(S^3)$,
we may assume $f=\LM(\frac{k_1}{pd_2},\frac{k_2}{pd_1})$,
where $p$, $d_1$, $d_2$, $k_1$, and $k_2$ are positive integers such that
$\gcd(d_1,d_2)=1$, $\gcd(pd_2,k_1)=1$, $\gcd(pd_1,k_2)=1$, and $n=pd_1d_2$.
\end{claim}

\begin{proof}[Proof of Claim \ref{claim:normalise-f}]
Since any element of $S^3$ is conjugate to an element in $S^1$,
we may assume, by taking conjugation, that 
$f \in \phi(S^1\times S^1)=\LM(S^1\times S^1)$ (see (\ref{LM-I})). 
Since $f$ has order $n$, we may assume
$f=\LM(\frac{k_1'}{n},\frac{k_2'}{n})$
for some integers $k_1'$ and $k_2'$ such that
$\gcd(n,k_1',k_2')=1$.
For $\ell=1,2$, set $d_{\ell}=\gcd(n,k_{\ell}')$, $n_{\ell}=\frac{n}{d_{\ell}}$ and $k_{\ell}=\frac{k_{\ell}'}{d_{\ell}}$,
so that $f=\LM(\frac{k_1'}{n},\frac{k_2'}{n})=\LM(\frac{k_1}{n_1},\frac{k_2}{n_2})$,
where $\gcd(k_1,n_1)=\gcd(k_2,n_2)=1$.
Note also that $\gcd(d_1,d_2)=\gcd(n,k_1',k_2')=1$.
Set $p=\gcd(n_1,n_2)$ and  $n_{\ell}'=\frac{n_{\ell}}{p}$ (${\ell}=1,2$).
Then $n=\lcm(n_1,n_2)=pn_1'n_2'$.
Thus $n_1d_1=n=pn_1'n_2'=n_1n_2'$ and so $d_1=n_2'$.
Similarly we have $d_2=n_1'$.
Hence we have $n=pd_1d_2$ and
$f=\LM(\frac{k_1}{n_1},\frac{k_2}{n_2})
=\LM(\frac{k_1}{pd_2},\frac{k_2}{pd_1})$.
\end{proof}

Now consider the subgroup $\langle f^p\rangle\cong \ZZ_{d_1d_2}$ generated by 
$f^p=\LM(\frac{k_1}{d_2},\frac{k_2}{d_1})$.
Since $\gcd(d_1,d_2)=1$, we have 
\[
\langle f^p\rangle \cong \langle f^{pd_2}\rangle \times \langle f^{pd_1}\rangle
\cong \ZZ_{d_1}\times \ZZ_{d_2}.
\]
Note that 
\[
\langle f^{pd_2}\rangle =\langle \LM(0, \frac{k_2d_2}{d_1})\rangle =
\langle \LM(0, \frac{1}{d_1})\rangle,
\quad
\langle f^{pd_1}\rangle =\langle \LM(\frac{k_1d_1}{d_2},0)\rangle =
\langle \LM(\frac{1}{d_2},0)\rangle.
\]
Hence we have
\[
\langle f^p\rangle = \langle \LM(0, \frac{1}{d_1})\rangle \times \langle \LM(\frac{1}{d_2},0)\rangle.
\]
Thus $S^3/\langle f^p\rangle$ is the orbifold 
with underlying space $S^3$
and with singular set the Hopf link,
where one component has index $d_1$ and the other component has index $d_2$.
To give a precise description of this orbifold,
identify $S^3$ with the join $S^1*S^1$,
by the correspondence
$(tz_1, \sqrt{1-t^2} z_2)\leftrightarrow t z_1+(1-t)z_2$.
Thus the first and second factor circles of $S^1*S^1$ correspond to
the circles $S^1\times \{0\}$ and $\{0\} \times S^1$ in $S^3\subset \CC^2$, respectively. 
For $\omega\in S^1$, let $\LM(\omega)$ be the isometry of $S^1$
defined by $\LM(\omega)(z)=\omega z$ ($z\in S^1$).
Then the isometry $\LM(\omega_1,\omega_2)$ is
identified with the self-homeomorphism
$\LM(\omega_1)*\LM(\omega_2)$ of $S^1*S^1$, defined by
\[
(\LM(\omega_1)*\LM(\omega_2))(t z_1+(1-t)z_2)
=t \omega_1z_1 +(1-t)\omega_2 z_2.
\]
Under the above convention, 
the orbifold $S^3/\langle f^p\rangle$ is described as follows.
The underlying space of the orbifold is given by
\[
|S^3/\langle f^p\rangle |
\cong
\left(S^1/\LM(\frac{1}{d_2})\right)*\left(S^1/\LM(\frac{1}{d_1})\right)
\cong
S^1*S^1\cong S^3,
\]
and the singular set is the union of the two circles 
which gives the join structure of $S^3$,
where the first factor circle (which corresponds to $S^1/\LM(\frac{1}{d_2})$)
has index $d_1$ and
the second factor circle (which corresponds to $S^1/\LM(\frac{1}{d_1})$)
has index $d_2$.
Here, $\LM(\frac{1}{d_{\ell}})$ denotes $\LM(e^{\frac{2\pi i}{d_{\ell}}})$
as in (\ref{LM-F}).

The isometry $f$ descends to the periodic isomorphism
of the orbifold $S^3/\langle f^p\rangle\cong S^1*S^1$
given by $\LM(\frac{k_1}{p})*\LM(\frac{k_2}{p})$,
because
the periodic map $\LM(\frac{k_1}{pd_2})$ 
(resp. $\LM(\frac{k_2}{pd_1})$)
on $S^1$
descends to the periodic map $\LM(\frac{k_1}{p})$
(resp. $\LM(\frac{k_2}{p})$)
on the circle $S^1/\LM(\frac{1}{d_2})$
(resp. $S^1/\LM(\frac{1}{d_1})$).
Note that
$\langle \LM(\frac{k_1}{p})*\LM(\frac{k_2}{p})\rangle
=
\langle \LM(\frac{1}{p})*\LM(\frac{q}{p})\rangle$,
with $q\equiv k_1^{-1}k_2 \pmod{p}$,
where $k_1^{-1}$ is the inverse of $k_1$ in 
the multiplicative group $(\ZZ_p)^{\times}$
(cf. Notation \ref{notation}(2)).
Hence we see that the orbifold
$S^3/\langle f\rangle$ is 
isomorphic to the orbifold, $\OO(L(p,q), d_1,d_2)$,
with underlying space the lens space, $L(p,q)$,
and with singular set the union of the core circles 
of the standard genus $1$ Heegaard splitting of $L(p,q)$
with indices $d_1$ and $d_2$, respectively.
(Though the notation $L(p,q)$ looks similar to the notation $\LM(\cdot,\cdot)$
in (\ref{LM-F}), we believe there is no fear of confusion.)

Since $h^2=1$ and $hfh^{-1}=f^{-1}$,
we see by using Lemma \ref{lem.2-dim-normaliser}(1) that
$h=\phi(q_1,q_2)$ for some $(q_1,q_2)\in S^1j\times S^1j$.
Since any element of $S^1j$ is conjugate to $j$ by an element of $S^1$,
we may assume $h=\phi(j,j)=J$, and so
$h(z_1,z_2)=(\bar z_1,\bar z_2)$.
This implies that the involution $h$ of $S^3$
descends to the hyper-elliptic involution of
$|S^3/\langle f\rangle|\cong L(p,q)$.
Recall that (i) the quotient map determined by the hyper-elliptic involution
gives the double branched covering of $S^3$ branched over  
the $2$-bridge link $K(q/p)$
and that (ii) the core circles of the genus $1$ Heegaard splitting
project to the upper and lower tunnels, respectively.
Hence, the quotient $S^3/\Gamma$, with $\Gamma=\langle f, h\rangle\cong D_n$,
is isomorphic to the orbifold
$\OO(q/p;d_1,d_2)$.
This completes the proof of the only if part of the first assertion.
The proof also shows that the group $\Gamma$ is given by the formula
(\ref{spherical-dihedral-group})
for some integers $k_1$ and $k_2$ satisfying the condition 
(\ref{gcd-condition}).

\medskip

The if part
of the first assertion follows from 
the above argument and the following claim.

\begin{claim}
\label{claim:congruence}
For any rational number $r=q/p$ and a pair of coprime integers $(d_1,d_2)$,
there is a pair $(k_1,k_2)$ of integers
which satisfies the condition (\ref{gcd-condition}).
\end{claim}

\begin{proof}[Proof of Claim \ref{claim:congruence}]
Consider the homomorphism 
\[
\Psi: (\ZZ_{pd_2})^{\times}\times (\ZZ_{pd_1})^{\times}
\to (\ZZ_{p})^{\times}\times (\ZZ_{p})^{\times}
\to (\ZZ_{p})^{\times},
\]
where the first homomorphism is the product of the natural projections
and the second homomorphism maps $(k_1,k_2)\in  (\ZZ_{p})^{\times}\times (\ZZ_{p})^{\times}$
to $k_1^{-1}k_2\in (\ZZ_{p})^{\times}$.
Then both of the two homomorphisms are surjective and so is their composition $\Psi$.
Regard the numerator $q$ of the rational number $r=q/p$ 
as an element of $(\ZZ_{p})^{\times}$,
and let $(k_1,k_2)$ be a pair of integers which projects to an element
in the inverse image $\Psi^{-1}(q)$.
Then $(k_1,k_2)$ satisfies the condition (\ref{gcd-condition}).
\end{proof}

Finally we prove the
uniqueness of the spherical structure on the orbifold $\OO(q/p;d_1,d_2)$.
The preceding arguments show that
the triple $(q/p,d_1,d_2)\in\QQ\times\NN\times\NN$
uniquely determines a dihedral subgroup $\Gamma<\Isom^+(S^3)$, up to conjugation,
such that $S^3/\Gamma$ is isomorphic to $\OO(q/p;d_1,d_2)$ as oriented orbifolds.
Thus we have only to show that 
there are no unexpected orientation-preserving topological isomorphism
between two orbifolds,  
$\OO(q/p;d_1,d_2)=S^3/\Gamma$ and $\OO(q'/p';d_1',d_2')=S^3/\Gamma'$.

Assume that $\OO(q/p;d_1,d_2)$ and $\OO(q'/p';d_1',d_2')$
are isomorphic as oriented orbifolds.
Then $pd_1d_2=p'd_1'd_2'$ and $\{d_1,d_2\}=\{d_1',d_2'\}$,
because they have isomorphic orbifold fundamental groups 
and the same index sets of the singular sets. 
In particular we have $p=p'$.

Suppose first that $n:=pd_1d_2\ge 3$.
Then $D_n$ has the unique cyclic subgroup of index $2$,
and so each of $\OO(q/p;d_1,d_2)$ and $\OO(q'/p';d_1',d_2')$
has the unique double orbifold covering with cyclic orbifold fundamental group.
The underlying spaces of the covering orbifolds are
the lens spaces $L(p,q)$ and $L(p',q')$, respectively.
Hence, by the classification of lens spaces,
we have $p=p'$ and either $q\equiv q' \pmod p$
or $qq'\equiv 1 \pmod p$.
Moreover, by using the uniqueness of the genus one Heegaard splittings
(see \cite{Bonahon, Bonahon-Otal}),
we see that $\OO(q/p;d_1,d_2)$ and $\OO(q'/p';d_1',d_2')$
are isomorphic as oriented orbifolds if and only if 
one of the following conditions hold.
\begin{enumerate}
\item
$p=p'$, $q\equiv q' \pmod p$, and $(d_1,d_2)=(d_1',d_2')$.
\item
$p=p'$, $qq'\equiv 1 \pmod p$, and $(d_1,d_2)=(d_2',d_1')$.
\end{enumerate}
In both cases, we can see that
the subgroups $\Gamma$ and $\Gamma'$ 
are conjugate in $\Isom^+(S^3)$.

In the exceptional case when $n:=pd_1d_2=2$,
we have either (i) $p=p'=1$ and $\{d_1,d_2\}=\{d_1',d_2'\}=\{1,2\}$
or (ii) $p=p'=2$ and $d_1=d_2=d_1'=d_2'=1$.
We can easily see that
the subgroups $\Gamma$ and $\Gamma'$
are conjugate in $\Isom^+(S^3)$.

This completes the uniqueness of the spherical structure.
\end{proof}

Next, we calculate the (orientation-preserving) isometry group
of the dihedral spherical $3$-orbifold $\OO(q/p;d_1,d_2)$.
If $(d_1,d_2)=(1,1)$, then $\OO(q/p;d_1,d_2)$ is 
the $\pi$-orbifold, $\OO(q/p)$, associated with 
the $2$-bridge link $K(q/p)$ (cf. \cite {Boileau-Zimmermann})  
and its isometry group is 
calculated by \cite[Theorem 4.1]{Sakuma2} and \cite[Corollary 3.2.11]{Jeevanjee}.
(There are errors in \cite[Theorem 4.1]{Sakuma2} for the special case when $p=1$, $2$.
There are also misprints for the generic case in the statement of Theorem 4.1, though the correct result can be found in the tables in 
\cite[p.184]{Sakuma2}.)

\begin{proposition}
\label{prop:dihedral-orbifold-isometry1}
The orientation-preserving isometry group of the spherical orbifold
$\OO(q/p):=\OO(q/p;1,1)$ is described as follows.
\begin{enumerate}
\item
If $q\not\equiv \pm 1 \pmod p$, then the following holds.
\[
\Isom^+(\OO(q/p))\cong
\begin{cases}
(\ZZ_2)^2 & \text{if $q^2\not\equiv 1\pmod{p}$}\\
D_4 & \text{if $p$ is odd and $q^2 \equiv 1\pmod{p}$}\\
         & \text{or if $p$ is even and $q^2 \equiv p+1\pmod{2p}$}\\
(\ZZ_2)^3 & \text{if $p$ is even and $q^2 \equiv 1\pmod{2p}$}
\end{cases}
\]
\item
If $q\equiv \pm 1 \pmod p$, then then the following holds.
\[
\Isom^+(\OO(q/p))\cong 
\begin{cases}
S^1\rtimes\ZZ_2 & \text{if $p$ is odd and $\ge 3$}\\
S^1\rtimes (\ZZ_2)^2 & \text{if $p$ is even and $\ge 4$}\\
(S^1\times S^1)\rtimes (\ZZ_2)^2 & \text{if $p=2$}\\
(S^1\times S^1)\rtimes \ZZ_2 & \text{if $p=1$}
\end{cases}
\]
\end{enumerate}
\end{proposition}

In the remainder of this section, we treat the 
remaining case $(d_1,d_2)\ne (1,1)$.
In the very special case, when $p=1$ and $\{d_1,d_2\}=\{1,2\}$,
we call $\OO(0/1;1,2)$ the {\it trivial $\theta$-orbifold},
because its singular set is the trivial $\theta$-curve in $S^3$.
Then we have the following proposition.

\begin{proposition}
\label{prop:dihedral-orbifold-isometry2}
The orientation-preserving isometry group of 
the spherical dihedral orbifold $\OO(q/p;d_1,d_2)$ with 
$(d_1,d_2)\ne (1,1)$
is described as follows.
\begin{enumerate}
\item
$\Isom^+(\OO(q/p;d_1,d_2))\cong (\ZZ_2)^2$,
except when 
$\OO(q/p;d_1,d_2)$ is isomorphic to the trivial $\theta$-orbifold $\OO(0/1;1,2)$,
i.e.
except when $p=1$ and $\{d_1,d_2\}= \{1,2\}$.
\item
For the the trivial $\theta$-orbifold $\OO(0/1;1,2)$,
we have
$\Isom^+(\OO(0/1;1,2))\cong D_3 \times \ZZ_2$.
\end{enumerate}
\end{proposition}

\begin{figure}
\includegraphics[width=1.0\hsize, bb=0 0 2636 835]{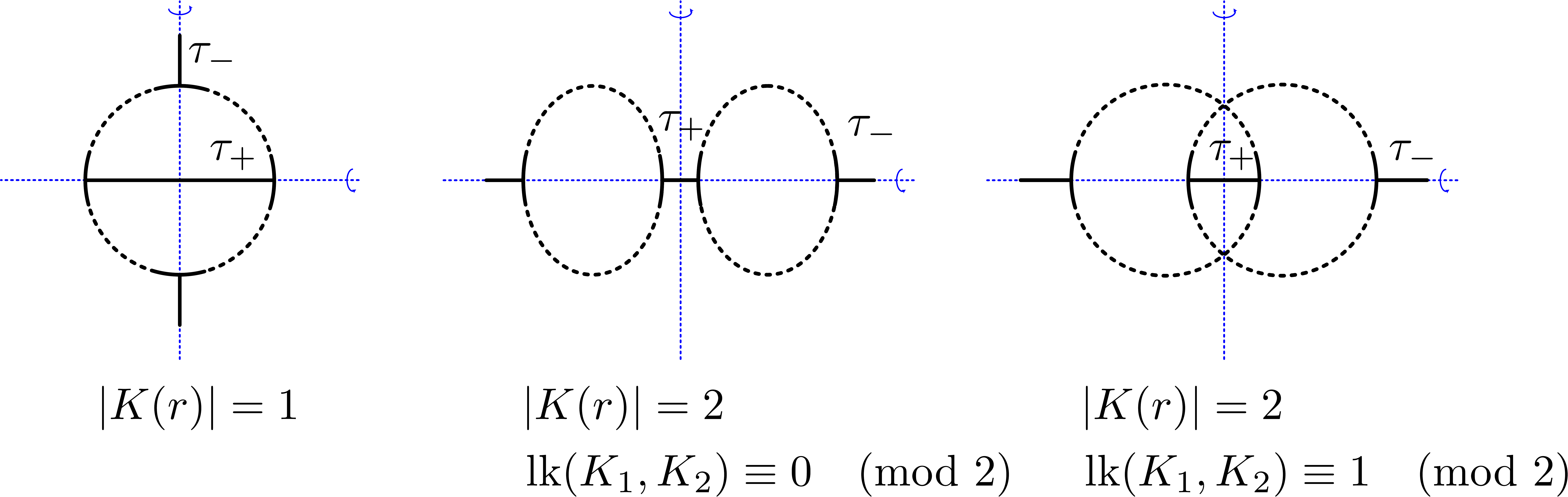}
\caption{
$\Isom^+(\OO(q/p;d_1,d_2))\cong (\ZZ_2)^2$
if $(d_1,d_2)\ne(1,1)$ and $\OO(q/p;d_1,d_2))\not\cong \OO(0/1;1,2)$.
}
\label{symmetry-dihedral-orbifold}
\end{figure}

\begin{figure}
\includegraphics[width=0.3\hsize, bb=0 0 423 594]{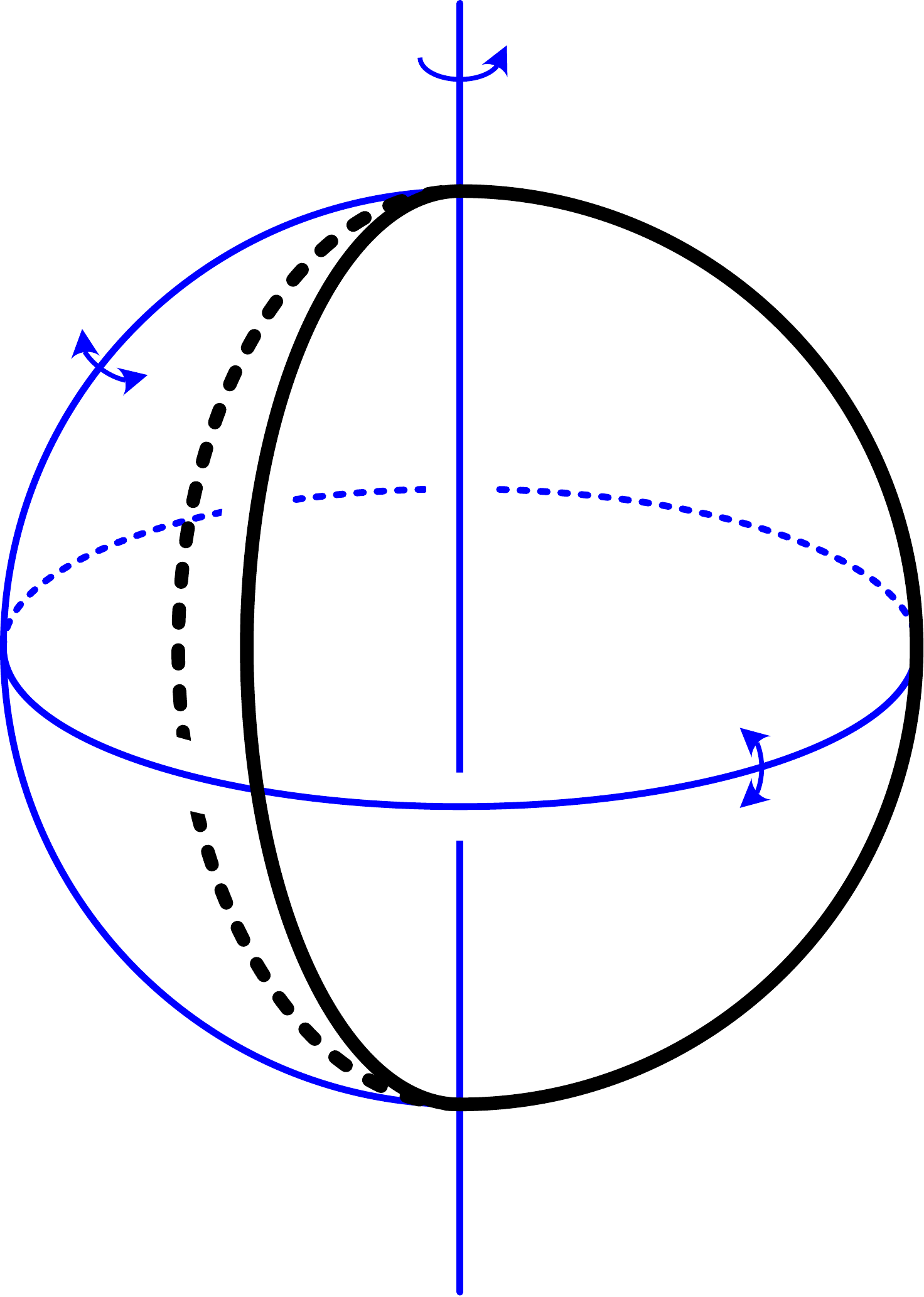}
\caption{$\Isom^+(\OO(0/1;1,2))\cong D_3\times\ZZ_2$.
The singular set of the trivial $\theta$-orbifold $\OO(0/1;1,2)$ 
is the standardly embedded $\theta$-graph in $S^2\subset S^3$,
consisting of three geodesics joining the north and south poles,
which are permuted by the $2\pi/3$-rotation around the earth axis. 
The orientation-preserving isometry group is visualised
as the product
of the dihedral group $D_3$ generated by 
the $\pi$-rotations about the three great circles containing the 
singular edges
and cyclic group $\ZZ_2$ generated by the $\pi$-rotation about the equator.
}
\label{symmetry-dihedral-orbifold-except}
\end{figure}

Before proving the proposition,
we give the following consequence of the proposition
(and the orbifold theorem),
which is used in the proof of the main theorem.

\begin{corollary}
\label{cor:dihedral-orbifold-isometry2}
Consider a spherical orbifold $\OO(q/p;d_1,d_2)$ with $(d_1,d_2)\ne (1,1)$,
and let $g$ be an orientation-preserving involution 
of the orbifold.
Then the following hold. 
\begin{enumerate}
\item
Except when $p=1$ and $\{d_1,d_2\}= \{1,2\}$,
(i.e. except when 
$\OO(q/p;d_1,d_2)$ is isomorphic to the trivial $\theta$-orbifold $\OO(0/1;1,2)$),
$g$ stabilises the edges $\tau_+$ and $\tau_-$ of the singular set
(when it is contained in the singular set).
\item
If $d_1, d_2\ge 2$, then $g$ does not stabilise any edge of the singular set different from $\tau_{\pm}$.
\end{enumerate}
\end{corollary}

\begin{proof}[Proof of Corollary \ref{cor:dihedral-orbifold-isometry2}]
By the orbifold theorem,
we may assume $g$ is an isometry of the spherical orbifold.
(This is proved by applying the orbifold theorem
to the finite group action on the universal cover $S^3$ of $\OO(q/p;d_1,d_2)$
generated by a lift of $g$ and the covering transformation group.)
On the other hand, the action of $\Isom^+(\OO(q/p;d_1,d_2))$ 
in the generic case
is as illustrated in Figure \ref{symmetry-dihedral-orbifold}.
(See also \cite[Figure 6-8]{ALSS}, and replace the weights $\infty$ with $2$,
then we obtain the desired visualisation, besides
the exceptional case.)
The exceptional case where the orbifold is the trivial $\theta$-orbifold
is illustrated in Figure \ref{symmetry-dihedral-orbifold-except}.
The assertion (1) is now obvious from Figure \ref{symmetry-dihedral-orbifold}.
The assertion (2) also follows from the figure 
by noting that $K(r)$ consists of four edges 
if $d_1, d_2\ge 2$
(otherwise, $K(r)$ consists of two edges).
\end{proof}

The proof of Proposition \ref{prop:dihedral-orbifold-isometry2}
presented below is parallel to that
of \cite[Theorem 4.1]{Sakuma2}.
Consider the dihedral spherical $3$-orbifold $\OO(q/p;d_1,d_2)$
with $(d_1,d_2)\ne (1,1)$.
By Proposition \ref{prop:dihedral-orbifold1},
the orbifold fundamental group $\pi_1(\OO(q/p;d_1,d_2))$
is identified with the subgroup 
\[
\Gamma
= \left\langle \LM(\frac{k_1}{pd_2},\frac{k_2}{pd_1}),\ J \right\rangle
=\langle \LM(\omega_1,\omega_2),\ J\rangle
< \Isom^+(S^3)
\]
where $\omega_1=\exp(2\pi i\frac{k_1}{pd_2})$,
$\omega_2=\exp(2\pi i\frac{k_2}{pd_1})$ 
for some integers $k_1$ and $k_2$ satisfying the condition (\ref{gcd-condition}).
Pick $(\eta_1,\eta_2)\in S^1\times S^1$ such that
$(\omega_1,\omega_2)=(\eta_1\bar\eta_2,\eta_1\eta_2)$.
Then $\Gamma=\langle \phi(\eta_1,\eta_2), \phi(j,j)\rangle$.
Set 
\[
\tilde \Gamma := \phi^{-1}(\Gamma)
=
\langle (\eta_1,\eta_2), (j,j) \rangle < S^3\times S^3.
\]
Then $\Isom^+\OO(q/p;d_1,d_2)\cong N(\tilde\Gamma)/\tilde\Gamma$,
where $N(\tilde\Gamma)$ is the normaliser of $\tilde\Gamma$ in $S^3\times S^3$.

For $\ell=1,2$, set $\tilde\Gamma_{\ell}=\mathrm{pr}_{\ell}(\tilde\Gamma)$, where
$\mathrm{pr}_{\ell}:S^3\times S^3\to S^3$ is the projection to the $\ell$-th factor.
Then $\tilde\Gamma_{\ell}=\langle \eta_{\ell}, j\rangle=\Di_{m_{\ell}}^*$
for some positive integer $m_{\ell}$.
Then the following lemma is obvious from the definition of $\Di_{m_{\ell}}^*$,
where $o(\cdot)$ denotes the order of a group element.

\begin{lemma}
\label{lem:dihedral-order}
(1) $o(\eta_{\ell}^2)=m_{\ell}$.

(2) If $m_{\ell}$ is even, then $o(\eta_{\ell})=2m_{\ell}$.
If $m_{\ell}$ is odd, then $o(\eta_{\ell})$ is either $m_{\ell}$ or $2m_{\ell}$.
\end{lemma}

Note that the orientation-reversing isometry $c:S^3\to S^3$, defined by $c(q)=\bar q$,
acts on $\Isom^+(S^3)$ by conjugation, as follows:
\[
c \phi(q_1,q_2) c^{-1}= \phi(q_2,q_1)
\]
Hence, we assume $m_1\leq m_2$ without loss of generality.

\begin{lemma}
\label{lem:d-d} 

(1) $2\le m_1\le m_2$.

(2) If $m_1=2$, then $m_2=2m_2'$ for some odd integer $m_2'$, and 
$\{d_1,d_2\}=\{1,2\}$.
Moreover, $m_1=m_2=2$ if and only if $p=1$.
\end{lemma}

\begin{proof}
(1) Suppose on the contrary that $m_1=1$. 
Then $\eta_1=\pm 1$, and so $(\omega_1,\omega_2)=\pm(\bar\eta_2,\eta_2)$. 
This implies $pd_2=o(\omega_1)=o(\omega_2)=pd_1$ and therefore $d_1=d_2$.
Since $\gcd(d_1,d_2)=1$, we have $d_1=d_2=1$, a contradiction.

(2) Suppose $m_1=2$. 
Then $\eta_1=\pm i$, and so
$(\omega_1,\omega_2)=\pm(i\bar\eta_2,i\eta_2)$.
By using Lemma \ref{lem:dihedral-order},
we can verify the following,
from which the assertion (2) follows.
\begin{enumerate}[(i)]
\item
If $m_2$ is odd, then $(o(\omega_1),o(\omega_2))=(4m_2, 4m_2)$ and so $d_1=d_2=1$ as in (1),
a contradiction.
\item
If $m_2=2m_2'$ for some odd integer $m_2'$, then
$\omega_1^{m_2'}=-\omega_2^{m_2'}=\pm 1$ 
and so $\{pd_2,pd_1\}=\{o(\omega_1),o(\omega_2)\}=\{m_2',2m_2'\}$.
Hence we have $p=m_2'$ and $\{d_1,d_2\}=\{1,2\}$.
\item
If $m_2=4m_2'$ for some integer $m_2'$,
then 
$\omega_1^{2m_2'}=-\omega_2^{2m_2'}=\pm i$
and so 
$o(\omega_1)=o(\omega_2)=8m_2'$.
Hence $d_1=d_2=1$, a contradiction.
\end{enumerate}
\end{proof}

\begin{lemma}
\label{lem:nomalizer2}
Except for the special case
where $p=1$ and $\{d_1,d_2\}=\{1,2\}$,
namely except when $\OO(q/p;d_1,d_2)$ is the trivial $\theta$-orbifold,
we have 
\[
N(\tilde\Gamma)<\Di_{2m_1}^*\times \Di_{2m_2}^*<\Di_S\times\Di_S.
\]
\end{lemma}

\begin{proof}
If $m_1\ge 3$, then Lemma \ref{lem.2-dim-normaliser} implies
$N(\tilde\Gamma_{\ell})=\Di_{2m_{\ell}}^*$ for each $\ell=1,2$
(because $m_2\ge m_1$ by assumption),
and hence we have
$N(\tilde\Gamma)<
N(\tilde\Gamma_{1})\times N(\tilde\Gamma_{2})
<\Di_{2m_1}^*\times \Di_{2m_2}^*$.

Since $m_2\ge m_1\ge 2$ by Lemma \ref{lem:d-d}(1),
we have only to treat the case where $m_1=2$.
Since we exclude the case where $p=1$ and $\{d_1,d_2\}=\{1,2\}$,
Lemma \ref{lem:d-d}(2) implies $m_2\ge 3$, and so 
$N(\tilde\Gamma_{2})=\Di_{2m_{2}}^*$. 
On the other hand, since $m_1=2$,
we see by Lemma \ref{lem.2-dim-normaliser} that $N(\tilde\Gamma_{1})=O^*$.
Hence $N(\tilde\Gamma)<O^*\times \Di_{2m_2}^*$.

Now observe that the  decomposition $\Di_S=S^1\sqcup S^1j$ induces the decomposition of
$\tilde\Gamma<\Di_S\times \Di_S$ into the following two non-empty subsets.
\[
\tilde\Gamma^{(1)}:=\tilde\Gamma\cap(S^1\times S^1)
\quad\mbox{and}\quad
\tilde\Gamma^{(j)}:=\tilde\Gamma\cap(S^1j\times S^1j).
\]
Note that $\mathrm{pr}_1(\tilde\Gamma^{(1)})=\langle i \rangle =\{\pm 1, \pm i\}$
and $\mathrm{pr}_1(\tilde\Gamma^{(j)})=\langle i \rangle j =\{\pm j, \pm k\}$.
Pick an arbitrary element $(q_1,q_2)\in N(\tilde\Gamma)$.
Then $q_2\in \Di_{2m_2}^*<\Di_S$, and so
the inner-automorphism of $S^3$ determined by $q_2$ preserves the subgroup 
$S^1<\Di_S$.
Thus the inner-automorphism of $S^3\times S^3$ determined by $(q_1,q_2)$ 
preserves the subset $\tilde\Gamma^{(1)}$ of 
$\tilde\Gamma=\tilde\Gamma^{(1)}\sqcup \tilde\Gamma^{(j)}$.
Hence
the inner-automorphism of $S^3$ determined by $q_1$ preserves the subgroup 
$\mathrm{pr}_1(\tilde\Gamma^{(1)})=\langle i \rangle$,
and so it preserves the subset $\{\pm i\}$,
i.e., $q_1 i \bar q_1=\pm i$.
By the description of $O^*$ in Lemma \ref{lem.2-dim-normaliser},
this implies that $q_1\in \Di_2^*<\Di_{2m_1}^*$. 
Hence $(q_1,q_2)\in \Di_{2m_1}^*\times \Di_{2m_2}^*$,
as desired.
\end{proof}

\begin{lemma}
\label{lem:nomalizer3}
The normaliser $N(\Gamma)$ of $\Gamma$ in $\Isom^+ S^3$ is contained in 
$\langle \LM(S^1\times S^1), J\rangle$.
\end{lemma}

\begin{proof}
By the formula (\ref{LM-D}) and Lemma \ref{lem:nomalizer2}, we have
\[
N(\Gamma)=\phi(N(\tilde\Gamma))< \phi(\Di_S\times \Di_S)=\langle \LM(S^1\times S^1), J, J_1\rangle.
\]
Since $J\in\Gamma$, we have only to show that $J_1\notin N(\Gamma)$.
To this end, recall that
$\Gamma=\langle \LM(\frac{k_1}{pd_2},\frac{k_2}{pd_1}), J \rangle$.
Now suppose on the contrary that $J_1\in N(\Gamma)$.
Then the conjugation by $J_1$ preserves the subgroup
$\langle \LM(\frac{k_1}{pd_2},\frac{k_2}{pd_1})\rangle$ and 
we have 
$J_1 \LM(\frac{k_1}{pd_2},\frac{k_2}{pd_1})J_1^{-1}=
\LM(\frac{k_2}{pd_1},\frac{k_1}{pd_2})$
by (\ref{LM-D2}).
Thus we have $d_1=d_2$ and so $d_1=d_2=1$,
a contradiction.
Hence $J_1\notin N(\Gamma)$ as desired. 
\end{proof}

\begin{lemma}
\label{lem:nomalizer4}
Except when $\OO(q/p;d_1,d_2)$ is the trivial $\theta$-orbifold,
we have the following. 
\begin{align*}
N(\Gamma)
&=
\left\langle\LM(\frac{k_1}{2pd_2},\frac{k_2}{2pd_1}), \ \LM(\frac{1}{2},0), \
\LM(0, \frac{1}{2}),\ J \right\rangle\\
&\cong
\left\langle\LM(\frac{k_1}{2pd_2},\frac{k_2}{2pd_1}), 
\ \LM(\frac{1}{2},0), \
\LM(0, \frac{1}{2})\right\rangle
\rtimes \langle J \rangle
\end{align*}
\end{lemma}

\begin{proof}
Recall that 
$\Gamma=\langle \LM(\omega_1,\omega_2), J\rangle$
where $\omega_1=\exp(2\pi i\frac{k_1}{pd_2})$ and
$\omega_2=\exp(2\pi i\frac{k_2}{pd_1})$.
Set $\sqrt{\omega_1}=\exp(\pi i\frac{k_1}{pd_2})$ and
$\sqrt{\omega_2}=\exp(\pi i\frac{k_2}{pd_1})$.
Suppose an element $\LM(\zeta_1,\zeta_2)\in \LM(S^1\times S^1)$
belongs to $N(\Gamma)$.
Then $\LM(\zeta_1^2,\zeta_2^2)J=\LM(\zeta_1,\zeta_2) J \LM(\zeta_1,\zeta_2)^{-1}\in \Gamma$,
and hence $(\zeta_1,\zeta_2)$ belongs to the subgroup
$\langle (\sqrt{\omega_1},\sqrt{\omega_2}),\ (-1,1),\ (1,-1)\rangle$.

Conversely, the image by $\LM$ of any element in the above subgroup belongs to $N(\Gamma)$.
Hence
\begin{align*}
N(\Gamma)\cap \LM(S^1\times S^1)
&=
\langle \LM(\sqrt{\omega_1},\sqrt{\omega_2}),\ \LM(-1,1),\ \LM(1,-1)\rangle\\
&=
\left\langle\LM(\frac{k_1}{2pd_2},\frac{k_2}{2pd_1}), 
\ \LM(\frac{1}{2},0), \
\LM(0, \frac{1}{2})\right\rangle.
\end{align*}
(Recall the abuse of notation given by (\ref{LM-F}).)
This together with Lemma \ref{lem:nomalizer3} implies
the desired result.
\end{proof}

\begin{proof}[Proof of Proposition \ref{prop:dihedral-orbifold-isometry2}]
Consider the spherical dihedral orbifold 
$\OO(q/p;d_1,d_2)$ with $(d_1,d_2)\ne (1,1)$.
We first treat the generic case where
$\OO(q/p;d_1,d_2)$ is not the trivial $\theta$-orbifold $\OO(0/1;1,2)$.
Then, by using Lemma \ref{lem:nomalizer4} and the fact that $J\in \Gamma$,
$\Isom^+\OO(q/p;d_1,d_2)\cong N(\Gamma)/\Gamma$ is isomorphic to
the quotient of the group
$N:=\langle\LM(\frac{k_1}{2pd_2},\frac{k_2}{2pd_1}), \LM(\frac{1}{2},0), 
\LM(0, \frac{1}{2}) \rangle$
by its subgroup $G:=\langle\LM(\frac{k_1}{pd_2},\frac{k_2}{pd_1})\rangle$.
Set $a=\LM(\frac{k_1}{2pd_2},\frac{k_2}{2pd_1})$, 
$b_1=\LM(\frac{1}{2},0)$ and $b_2=\LM(0, \frac{1}{2})$.
It should be noted that 
$\langle b_1, b_2\rangle \cong (\ZZ_2)^2$ and that
the subset $\{b_1, b_2, b_1b_2\}$
is equal to the set of all order $2$ elements of $\LM(S^1\times S^1)$.

Note that the order of $L(\frac{k_1}{2pd_2})$ is equal to
$2pd_2$ or $pd_2$ according to whether $k_1$ is odd or even.
Similarly, the order of $L(\frac{k_2}{2pd_1})$ is equal to
$2pd_1$ or $pd_1$ according to whether $k_2$ is odd or even.
Thus the order of $a=\LM(\frac{k_1}{2pd_2},\frac{k_2}{2pd_1})$ is 
$2pd_1d_2$ or $pd_1d_2$,
where the latter happens if and only if both $k_1$ and $k_2$ are even.

Case 1. $o(a)=2pd_1d_2$.
Then the element $a^{pd_1d_2}\in L(S^1\times S^1)$ has order $2$.
Hence it is equal to one of the elements of $\{b_1, b_2, b_1b_2\}$.
Thus $\langle a\rangle \cap \langle b_1, b_2\rangle \cong \ZZ_2$.
This implies that 
$N \cong \langle a \svert a^{2pd_1d_2} \rangle \oplus 
\langle b_{\ell} \svert b_{\ell}^2 \rangle$
for some ${\ell}\in\{1,2\}$.
Since $G$ corresponds to the subgroup $\langle a^2 \rangle$,
we have 
\[
\Isom^+\OO(q/p;d_1,d_2)\cong N/\langle a^2\rangle \cong
\langle a \svert a^{2} \rangle \oplus 
\langle b_{\ell} \svert b_{\ell}^2 \rangle
\cong (\ZZ_2)^2.
\]

Case 2. $o(a)=pd_1d_2$.
Since $o(a^2)=o(\LM(\frac{k_1}{pd_2},\frac{k_2}{pd_1}))=pd_1d_2$,
we have $o(a)=o(a^2)$, and so  $o(a)=pd_1d_2$ is odd.
Thus $\langle a\rangle \cap \langle b_1, b_2\rangle = \{1\}$,
and therefore 
$N \cong 
\langle a \svert a^{pd_1d_2} \rangle \oplus 
\langle b_1 \svert b_1^2 \rangle \oplus
\langle b_2 \svert b_2^2 \rangle$.
Hence, we have
\[
\Isom^+\OO(q/p;d_1,d_2)\cong N/\langle a^2\rangle
\cong N/\langle a\rangle
 \cong
\langle b_1 \svert b_1^2 \rangle \oplus
\langle b_2 \svert b_2^2 \rangle
\cong (\ZZ_2)^2.
\]

This completes the proof of Proposition \ref{prop:dihedral-orbifold-isometry2}
in the generic case.

In the exceptional case, where
$\OO(q/p;d_1,d_2)$ is the trivial $\theta$-orbifold $\OO(0/1;1,2)$,
we may assume
\[
\tilde\Gamma=\langle (i,i), (j,j) \rangle =
\{\pm (1,1), \pm (i,i), \pm(j,j), \pm(k,k)\}.
\]
Then, by using Lemma \ref{lem.2-dim-normaliser},
we can see that
$N(\tilde\Gamma)=\{(q,q) \svert q\in O^*\} \rtimes \langle J_1\rangle$.
Hence we have
\[
\Isom^+\OO(0/1;1,2)
\cong
\left(O^*/(\ZZ_2)\right) \rtimes \ZZ_2
\cong
D_3 \times \ZZ_2.
\]
\end{proof}

\section{Appendix 2: Non-spherical geometric orbifolds with dihedral orbifold fundamental groups}
\label{sec:dihedral-orbifold2}

In this section, we classify the non-spherical geometric orbifolds
with dihedral orbifold fundamental groups
(Propositions \ref{prop:dihedral-orbifold2} and \ref{prop:dihedral-orbifold3}).
These results are used in the proof of Theorem \ref{thm:dihedral-orbifold}.

We first deal with the dihedral orbifolds with $S^2\times \RR$ geometry. 

\begin{proposition}
\label{prop:dihedral-orbifold2}
Let $\OO$ be a compact orientable $S^2\times \RR$ orbifold
with nonempty singular set
which satisfies the following conditions.
\begin{enumerate}
\item[{\rm(i)}]
No component of $\partial\OO$ is spherical.
\item[{\rm(ii)}]
$\pi_1(\OO)$ is a dihedral group.
\end{enumerate}
Then $\OO$ is isomorphic to one of the following orbifolds.
\begin{enumerate}
\item
$\OO(\infty)$, the orbifold  
represented by the weighted graph $(S^3,K(\infty),w)$,
where $w$ takes the value $2$ at each component of
the $2$-bridge link $K(\infty)$ of slope $\infty$, i.e.
the  $2$-component trivial link.
\item
$\OO(\RP^3,O)$, the orbifold with underlying space $\RP^3$
whose singular set is the trivial knot (i.e., the boundary of an embedded disc in $\RP^3$)
with index $2$.
\end{enumerate}
\end{proposition}

\begin{proof}
By the assumption that $\OO$ has the geometry $S^2\times \RR$,
we have $\pi_1(\OO)<\Isom(S^2\times\RR)\cong\Isom(S^2)\times\Isom(\RR)$
and $\interior \OO \cong (S^2\times\RR)/\pi_1(\OO)$.
By the condition (ii), $\pi_1(\OO)\cong D_n$ for some $n\in\NN\cup\{\infty\}$.

If $n\in\NN$, then the action of the finite dihedral group $\pi_1(\OO)$ on $S^2\times\RR$
extends to an action on the compact $3$-manifold $S^2\times [-\infty, \infty]$,
where $[-\infty, \infty]\cong I$ is a compactification of $\RR$,
and $\OO$ is identified with  $S^2\times [-\infty, \infty]/\pi_1(\OO)$.
Thus $\OO$ has a spherical boundary component,
which contradicts the condition (i).
So $n=\infty$
and $\pi_1(\OO)\cong\langle f, h \ | \ h^2,\ hfh=f^{-1}\rangle$.
Since the action of $\pi_1(\OO)$ on $S^2\times\RR$
is properly discontinuous,
$f\in \Isom(S^2\times\RR)$ 
is the product of a (possibly trivial) rotation of $S^2$
and a nontrivial translation of $\RR$.
Thus the orbifold 
$\OO(f):=(S^2\times\RR)/\langle f \rangle$
is homeomorphic to the manifold $S^2\times S^1$.
The isometry $h$ descends to a fiber-preserving involution of $\OO(f)\cong S^2\times S^1$
which acts on the second factor as a reflection.
Thus $\OO=\OO(f)/h$ is the quotient of $S^2\times [0,1]$
by an equivalence relation
$(x,0)\sim (\gamma_0(x), 0)$ and $(x,1)\sim (\gamma_1(x), 1)$
where $\gamma_0$ and $\gamma_1$ are orientation-reversing involutions of $S^2$.
Thus $\gamma_i$ is conjugate to either the reflection in a great circle or
the antipodal map.
According to the combination
(reflection, reflection), (reflection, antipodal map), or
 (antipodal map, antipodal map),
$\OO$ is isomorphic to $\OO(\infty)$, $\OO(\RP^3,O)$, or $\RP^3\#\RP^3$. 
The last case cannot happen because $\OO$ has the empty singular set.
\end{proof}

The following proposition 
deals with 
the dihedral orbifolds with the remaining $6$ geometries.

\begin{proposition}
\label{prop:dihedral-orbifold3}
Let $\OO$ be a compact orientable $3$-orbifold
with nonempty singular set
which has one of the $6$ geometries different from $S^3$ and $S^2\times\RR$
and satisfies the following conditions.
\begin{enumerate}
\item[{\rm(i)}]
$\pi_1(\OO)$ is a dihedral group.
\item[{\rm(ii)}]
No component of $\partial\OO$ is spherical.
\end{enumerate}
Then $\OO$ is isomorphic to $D^2(2,2)\times I$.
\end{proposition}

\begin{proof}
Let $X$ be the geometry which $\OO$ possesses. 
Then $X$ is $\HH^3$, $\EE^3$, 
$\widetilde{SL_2(\RR)}$, $Nil$ or $Sol$,
and $\interior \OO$ is isomorphic to $X/\Gamma$
for some discrete subgroup $\Gamma\cong \pi_1(\OO)$ of $\Isom(X)$.
Note that the underlying topological space of $X$
is homeomorphic to $\RR^3$.
The proof is divided into two cases
according to whether $\pi_1(\OO)$ is finite or infinite.

\medskip
Case 1. 
Suppose that $\pi_1(\OO)$ is a finite dihedral group $D_n$.
Then, as will be shown below, 
the action of $D_n$ on $X$ has a global fixed point $x$.
Then the exponential map from $T_xX$,
the tangent space to $X$ at $x$, to $X$
is a $D_n$-equivariant homeomorphism.
This implies that 
$\partial\OO\cong S^2(2,2,n)$,
contradicting the condition (ii).

The existence of a global fixed point can be proved as follows.
For the constant curvature case $X=\HH^3$ or $\EE^3$,
this is well-known. 
We shall first deal with the case where $X$ is 
$Nil$, $\HH^2 \times \RR$, or $\widetilde{\SL_2 \RR}$.
Recall that there is an exact sequence 
\[
1 \to \Isom(\RR) \to \Isom(X) \to \Isom(E) \to 1,
\] 
where $E$ is the Euclidean plane $\EE^2$ when $X$ is $Nil$ and the hyperbolic plane $\HH^2$ when $X$ is $\HH^2 \times \RR$ or $\widetilde{\SL_2 \RR}$.
We also note that the projection 
$\Isom(X)\to\Isom(E)$ 
above is induced by
a fibration $p: X \to E$.
Let $\bar D_n$ be the image of $D_n$ in $\Isom(E)$ and $K$ the kernel in $D_{\infty}$
of the projection to $\bar D_n$.
Then the action of $\bar D_n$ on $E$ has a global fixed point $y$,
and the action of $K$ on the fibre $p^{-1}(y)$ has a global fixed point since both of them are finite.
Thus $D_n$ has a global fixed point on $X$
when $X$ is $Nil$, $\HH^2 \times \RR$, or $\widetilde{\SL_2 \RR}$.

We shall now show the same property when $X=Sol$.
In this case, there is an exact sequence 
\[
1 \to \Isom(\EE^2) \to \Isom(Sol) \to \Isom(\RR) \to 1,
\]
and the projection $\Isom(Sol)\to\Isom(\RR)$
is induced by a fibration $q: Sol \to \RR$.
Let $\bar D_n$ be the projection of $D_n$ in $\Isom(\RR)$.
Then $\bar D_n$ is either trivial or $\ZZ_2$ generated by a reflection on $\RR$.
In either case, it fixes a point $y$ on $\RR$.
In the former case, 
$D_n\cong\langle g, h \ | \ g^2, h^2,\ (gh)^n\rangle$ acts on the fibre $q^{-1}(y)$ by Euclidean isometries in such a way that $g$ and $h$ correspond to reflections, and hence $D_n$ fixes a point on the fibre.
In the latter case, the kernel $K$ of the projection $D_n \to \bar D_n$ is isomorphic to $\ZZ_n$, and fixes a point on the fibre in the same way.
Thus we have shown that $D_n$ has a fixed point also in the case when $X=Sol$.

\medskip
Case 2.
Suppose $\pi_1(\OO)$ is the infinite dihedral group 
$D_\infty \cong\langle g, h \ | \ g^2, h^2\rangle$.
First we shall consider the case when $X$ has constant curvature.
Then $g$ and $h$ are order $2$ elliptic transformations, and hence fix pointwise axes $a_g$ and $a_h$ respectively.
They do not meet each other since otherwise the action fixes their intersection and 
cannot be faithful and discrete.
Let $\ell$ be the common perpendicular to $a_g$ and $a_h$ if it exists.
(This does not exist when $X=\HH^3$ and $a_g$ touches $a_h$ at infinity.
This exceptional case will be considered later.)
Let $\Pi_g$ be the totally geodesic plane containing $a_h$ and perpendicular to $\ell$.
We define $\Pi_h$ in the same way.
Then the region cobounded by $\Pi_g$ and $\Pi_h$ constitutes a fundamental domain of the action of $D_\infty$.
Suppose now that $X=\HH^3$ and 
$a_g$ touches $a_h$ at infinity.
Then there is a totally geodesic plane $H$ containing both $a_g$ and $a_h$ and it is preserved by $D_\infty$.
We then let $\Pi_g$ and $\Pi_h$ be totally geodesic planes containing $a_g$ and $a_h$ respectively, which are perpendicular to $H$.
Then a fundamental domain is cobounded by $\Pi_g$ and $\Pi_h$ again.
Therefore, in either case, we  can see 
$\interior \OO\cong\interior D^2(2,2) \times \RR$
and so $\OO\cong D^2(2,2) \times I$.

Next, we shall consider the case when $X$ is $Nil$ or $\HH^2 \times \RR$ or $\widetilde{\SL_2 \RR}$.
As before, let $\bar D_\infty$ be the projection of $D_\infty$ to $\Isom(E)$, and $K$ the kernel of the projection.
We first observe that the images $\bar g$ and $\bar h$ of $g$ and $h$ in $\bar D_{\infty}$
are nontrivial.
In fact, if say $\bar g$ is trivial, then $g$ acts on $\RR$ as a nontrivial order $2$ 
isometry. 
Thus $g$ acts on $\RR$ by a reflection, and so the image $\bar g$ must be an orientation-reversing isometry on $E$, 
which contradicts our assumption that $\bar g$ is trivial.

Since $g$ and $h$ have order $2$,
their images $\bar g$ and $\bar h$ in $\bar D_{\infty}$ also have order $2$.
We first deal with the case where both of them are orientation-preserving, i.e. $\pi$-rotations.
Let $y_g$ and $y_h$ be the centres of the $\pi$-rotations
$\bar g$ and $\bar h$, respectively.
Then $g$ and $h$ are $\pi$-rotations about the geodesics 
$p^{-1}(y_g)$ and  $p^{-1}(y_h)$, respectively.
Since the action of $D_{\infty}$ is faithful,
we have $y_g\ne y_h$.
Now consider the geodesic
line $\ell$ in $E$ containing $y_g$ and $y_h$,
and let $\ell_g$ and $\ell_h$ be the lines which intersects $\ell$
perpendicularly at $y_g$ and $y_h$, respectively.
Then $\ell_g$ and $\ell_h$ are disjoint, 
and they cobound a region $R$ in $E$.
We see that $p^{-1}(R)$ is a fundamental region of the action of $D_\infty$, and we have $\interior \OO\cong\interior D^2(2,2) \times \RR$.

We next treat the case where both of $\bar g$ and $\bar h$ 
are orientation-reversing, i.e. reflections.
Let $a_g$ and $a_h$ be the axes of the reflections
$\bar g$ and $\bar h$, respectively.
Then $g$ and $h$ are the \lq symmetries' with respect to the
geodesics $\tilde a_g$ and $\tilde a_h$, respectively,
where $\tilde a_g$ and $\tilde a_h$ are lifts of
$a_g$ and $a_h$, respectively.
If $a_g$ and $a_h$ are disjoint, they cobound a region $R$ in $E$,
and $p^{-1}(R)$ is a fundamental region of the action of $D_\infty$, and we have $\interior \OO\cong\interior D^2(2,2) \times \RR$.
If $a_g$ and $a_h$ meet each other at a point $y\in E$.
Then $D_\infty$ acts effectively and discretely on the fiber $p^{-1}(y)$,
and so the axes $\tilde a_g$
and $\tilde a_h$ intersect $p^{-1}(y)$ perpendicularly at distinct points,
$z_g$ and $z_h$, respectively.
Let $P_g$ and $P_h$ be the ruled surfaces in $X$
obtained as the unions of the geodesics 
which intersect $p^{-1}(y)$ 
perpendicularly at 
$z_g$ and $z_h$, respectively.
Then $P_g$ and $P_h$ are disjoint planes in $X$, and
the domain they cobound is a fundamental domain of $D_{\infty}$, and 
we can see $\interior \OO\cong\interior D^2(2,2) \times \RR$.

We now treat the case where one of $\bar g$ and $\bar h$ is orientation-preserving
and the other is orientation-reversing.
We may assume $\bar g$ is orientation-preserving and  
$\bar h$ is orientation-reversing.
Let $y_g$ be the center of the $\pi$-rotation $\bar g$, and let
$a_h$ be the axis of the reflection $\bar h$.
If $y_g$ belongs to $a_h$, then the axes of the $\pi$-rotations of $g$ and $h$
intersect, 
and the action of $D_{\infty}$ cannot be discrete and faithful.
So $y_g$ is not contained in $a_g$.
Let $\ell$ be a geodesic line
in $E$ 
which passes through $y$ and is disjoint from $a_h$.
Let $R$ be the region in $E$ bounded by $a_h$ and $\ell$.
Then $p^{-1}(R)$ is a fundamental region of the action of $D_\infty$, 
and we have $\interior \OO\cong\interior D^2(2,2) \times \RR$.

Finally, suppose that $X=Sol$.
Then the projection $\bar D_\infty$ of $D_\infty$ to $\Isom(\RR)$ is either trivial or $\ZZ_2$ or $D_\infty$ itself.
In the case when $\bar D_\infty$ is trivial, 
the generators $g$ and $h$ act
on each fibre by $\pi$-rotations, and their fixed points must differ.
Thus $\interior\OO\cong X/D_{\infty}$ is a bundle over $\RR$ with fiber
$\EE^2/D_{\infty}\cong \interior D^2(2,2)$.
So we have $\interior\OO\cong  \interior D^2(2,2)\times \RR$.
In the case when $\bar D_\infty$ is $\ZZ_2$, the action of $\bar D_\infty$ is a reflection with respect to a point $x$.
We set $P=q^{-1}(x)$.
Then the $g$ and $h$ act on $P$ by reflections, and by the same argument as in the previous paragraph, we have a homeomorphism $\interior \OO \cong D^2(2,2) \times \RR$.
Finally, suppose that $\bar D_\infty=D_\infty$.
Then $g$ and $h$ fix points $x_g$ and $x_h$ on $\RR$ respectively, and they differ.
We consider fibres $\Pi_g=q^{-1}(x_g)$ and $\Pi_h=q^{-1}(x_h)$.
The elements act on $\Pi_g$ and $\Pi_h$ as reflections with axes $a_g \subset \Pi_g$ and $a_h \subset \Pi_h$.
Then the region cobounded by $\Pi_g$ and $\Pi_h$ constitutes a fundamental region for the action of $D_\infty$, and we see that 
$\interior \OO\cong\interior D^2(2,2) \times \RR$.
\end{proof}

\bibstyle{plain}

\end{document}